\documentclass[a4paper,12pt]{article}
\usepackage{amsmath}
\usepackage{amssymb}
\usepackage{amsthm}

\usepackage{amsfonts}
\usepackage{mathscinet}
\usepackage{enumerate}
\usepackage{pdfsync}
\usepackage{setspace}
\usepackage{mathrsfs}
\usepackage{stmaryrd}
\usepackage{tikz}
\usetikzlibrary{arrows,chains,matrix,positioning,scopes}
\usepackage{cite}
 \usepackage{enumitem}

\usepackage{bookmark}
\usepackage{authblk}

\newcommand {\new} {\newcommand}
\newcommand \oper[2] {\new #1 {\operatorname{#2}}}

\newcommand \gcode[1] {\ulcorner\! #1 \!\urcorner}	%godel code
\newcommand \se [1] { \{ #1 \}}			%{x}
\newcommand\set [2]{ \{#1:#2\} }			%{x: x is}
\newcommand\res {\!\upharpoonright\!}		%restriction
\newcommand\eqiv {\leftrightarrow}			%<-->
\newcommand\corner[1] {
  \langle #1 \rangle}		%<x>
\newcommand\power {{\mathscr{P}}}			%powerset
\newcommand\coll{{\text{Coll}}}
\newcommand{\card}{\operatorname{card}}

\newcommand\concat {{^\frown}}   %concatenation
\oper{\Ord}{Ord}				%ordinal
\oper{\hull}{Hull}				%hull
\oper{\ppt}{ppt} 				%ppt
\newcommand{\ord} {\Ord}
\oper{\ZFC}{ZFC}				%ZFC
\oper{\rank}{rank}				%rank
\oper{\crit}{cr}					%critical point
\oper{\crt}{crt}					%critical point
\oper{\cf}{cf}					%cofinality

\oper{\height}{ht}				%height
\oper{\wfcore}{wfcore}					%wellfounded core
\oper{\core}{core}					%core

\oper{\Ult}{Ult}				%ult
\oper{\Cone}{Cone}				%cone
\oper{\dirlim}{dirlim}
\oper{\rud}{rud}		%rud closure
\oper{\const}{const}
\oper{\OD}{OD}
\oper{\final}{final}
\oper{\HYP}{HYP}
\oper{\wfp}{wfp}
\oper{\Hom}{Hom}
\new{\ult} {\Ult}
\oper{\dom}{dom}
\oper{\rep}{rep}

\newcommand{\bardom}{\dom^{*}}
\newcommand{\exdesc}{\desc^{*}}
\newcommand{\exexdesc}{\desc^{**}}

\oper{\suc}{succ}
\oper{\fac}{fac}

\oper{\Code}{Code}

\oper{\ran}{ran}
\oper{\maxdom}{maxdom}
\oper{\maxran}{maxran}

\oper{\lp}{Lp} %lower part closure

\newcommand{\Los}{{\L}o\'{s}}
\newcommand\iniseg {\vartriangleleft}		%proper initial segment
 \newtheorem{theorem}{Theorem}[section]
 
 \newtheorem{lemma}[theorem]{Lemma}

 \newtheorem{claim}[theorem]{Claim}
 \theoremstyle{definition}
 
 \newtheorem{definition}[theorem]{Definition}

\newcommand{\game}{{\Game}}
\newcommand{\boldpi}[1]{{\boldsymbol{\Pi}^1_{#1}}}
\newcommand{\boldsigma}[1]{{\boldsymbol{\Sigma}^1_{#1}}}

\newcommand{\bolddelta}[1]{{\boldsymbol{{\delta}}^1_{#1}}}
\newcommand{\boldDelta}[1]{{\boldsymbol{{\Delta}}^1_{#1}}}

\newcommand{\WO}{{\textrm{WO}}}
\newcommand{\LO}{{\textrm{LO}}}
\newcommand{\DEF}{=_{{\textrm{DEF}}}}

\newcommand{\desc}{{\operatorname{desc}}}

\newcommand{\Diff}{\operatorname{Diff}}

\newcommand{\id}{\operatorname{id}}
\newcommand{\tree}{\operatorname{tree}}
\newcommand{\node}{\operatorname{node}}
\newcommand{\seed}{\operatorname{seed}}
\newcommand{\pred}{\operatorname{pred}}

\newcommand{\ot}{\mbox{o.t.}}

\newcommand{\sharpcode}[1]{  \left|   #1  \right|}

\newcommand{\wocode}[1]{  \|   #1  \|}

\newcommand{\comm}[1]{{}}
\newcommand{\comp}[2]{{}^{#1}\!#2}
\oper{\ucf}{ucf}					%cofinality
\oper{\sign}{sign}					%signature
\oper{\lh}{lh}

\newcommand{\ooo}{^{(3)}_{\omega^{\omega^{\omega}}}}

\begin{document}

\title{The higher sharp IV: the higher levels}

\author{Yizheng Zhu}

\affil{Institut f\"{u}r mathematische Logik und Grundlagenforschung \\
Fachbereich Mathematik und Informatik\\
Universit\"{a}t M\"{u}nster\\
Einsteinstr. 62 \\
48149 M\"{u}nster, Germany}

\maketitle

\begin{abstract}
We establish the descriptive set theoretic representation of the mouse $M_n^{\#}$, which is called $0^{(n+1)\#}$. This part deals with the case $n>3$. 
\end{abstract}

\section{Introduction}
\label{sec:introduction}

This is the final part of a series starting with \cite{sharpI}. 
In this paper, we generalize the previous three papers to the higher levels in the projective hierarchy. Section~\ref{sec:synt-defin} makes the purely syntactical definitions on trees of uniform cofinalities and descriptions that will show up in the higher levels. Section~\ref{sec:induction-hypothesis} writes down all the inductive definitions and hypotheses under $\boldDelta{2n}$-determinacy. Section~\ref{sec:induction-under-pi_det} proves a part of the inductive hypotheses in Section~\ref{sec:induction-hypothesis} under $\boldpi{2n+1}$-determinacy. Section~\ref{sec:level-2n+2-sharp} proves the rest of inductive hypotheses under $\boldDelta{2n+2}$-determinacy, thereby finishing a cycle of the induction.

  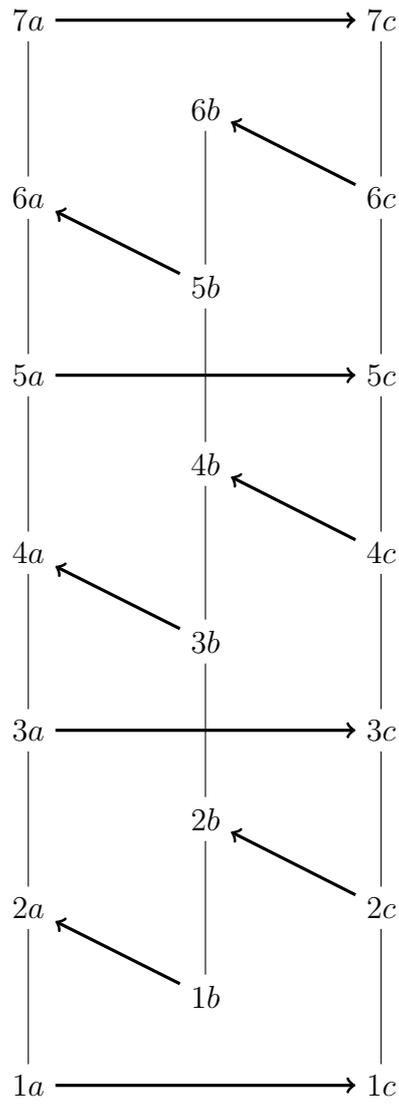
\begin{figure}
    \centering
  \begin{tikzpicture}
    \matrix  (m)  [matrix of math nodes, row sep=1.5em, column sep = 4em]
{7a&&7c\\
&6b&\\
6a&&6c\\
&5b&\\
5a&&5c\\
&4b&\\
4a&&4c\\
&3b&\\
3a & & 3c \\
     & 2b &  \\
2a & & 2c \\
  & 1b & \\
1a & & 1c \\
};
\draw[->, very thick](m-1-1) edge  (m-1-3) 
(m-5-1) edge  (m-5-3) 
(m-3-3) edge (m-2-2) 
(m-4-2) edge  (m-3-1)
(m-9-1) edge  (m-9-3) 
(m-13-1) edge  (m-13-3)
(m-7-3) edge (m-6-2)
(m-11-3) edge (m-10-2)  
(m-8-2) edge  (m-7-1)
(m-12-2) edge  (m-11-1)
;
\draw[-] (m-1-1) edge (m-3-1)
(m-5-1) edge (m-3-1)
(m-5-1) edge (m-7-1)
(m-7-1) edge (m-9-1)
(m-9-1) edge (m-11-1)
(m-11-1) edge (m-13-1)
(m-1-3) edge (m-3-3)
(m-5-3) edge (m-3-3)
(m-5-3) edge (m-7-3)
(m-7-3) edge (m-9-3)
(m-9-3) edge (m-11-3)
(m-11-3) edge (m-13-3)
(m-2-2) edge (m-4-2)
(m-4-2) edge (m-6-2)
(m-6-2) edge (m-8-2)
(m-8-2) edge (m-10-2)
(m-10-2) edge (m-12-2);
% \draw[dotted]  (-3,1.5) -- (3 , 1.5); 
% \draw[dotted]  (-3,-.5) -- (3, -.5); 
% \node at (4.5,2) {level-3};
% \node at (4.5,.4) {level-2};
% \node at (4.5,-1.3) {level-1};
  \end{tikzpicture}
\caption{The longer stack of definitions}
\label{fig:1}
     \end{figure}

As introduced in \cite[Section 1.1]{sharpIII}, a technical component in the level-2 and level-3 analysis is a self-similar stack of definitions. This stack grows to the higher levels. The whole picture is in Fig.~\ref{fig:1}. Every node has a distinguished name in this diagram and denotes a tree of uniform cofinality. The number denotes the level of the tree, e.g.\ 5a denotes a level-5 (or level $\leq 5$, to be exact) tree of uniform cofinality. An arrow stands for a factoring map. A solid line stands for membership, e.g.\ 4a is the tree component of an entry of 5a. The stack of definitions goes in the following order:
\begin{itemize}
\item $1c$-description,
\item $(1a,1c)$-factoring map,
\item $(2a,1c)$-description,
\item $(1b,2a,1c)$-factoring map,
\item $(2b,2a,1c)$-description,
\item $(2c,2b,2a)$-factoring map,
\item $(3c,2b,2a)$-description,
\item $(3a,3c,2b)$-factoring map,
\item $(4a,3c,2b)$-description,
\item $(3b,4a,3c)$-factoring map,
\item $(4b,4a,3c)$-description,
\item $(4c,4b,4a)$-factoring map,
\item $(5c,4b,4a)$-description,
\item $(5a,5c,4b)$-factoring map,
\item $(6a,5c,4b)$-description,
\item $(5b,6a,5c)$-factoring map,
\item $(6b,6a,5c)$-description,
\item $(6c,6b,6a)$-factoring map,
\item $(7c,6b,6a)$-description,
\item $(7a,7c,6b)$-factoring map,
\item $\dots$.
\end{itemize}

\section{Syntactical definitions at the higher levels}
\label{sec:synt-defin}

The definitions related to trees of uniform cofinalities are purely syntactical. 
They will be defined inductively. The base of the inductive definition are in \cite{sharpI,sharpII,sharpIII}. For notational consistency, a level-1 tree $P$ will be identified with a function on $P$ with constant value $\emptyset$. 
A level $\leq 1$ tree is $P'=(P)$ where $P$ is a level-1 tree. If $P' =(P)$ is a level $\leq 1$ tree, put $P = \comp{1}{P}'$ and put $\dom(P') = \set{(1,p)}{p \in P}$. $P'$ is regarded as a function on $\dom(P')$ with constant value $\emptyset$. A partial level $\leq 1$ tree $(P,p)$ will be identified with $( (P), (d,p,\emptyset))$, where $d=0$ if $p=-1$, $d=1$ if $p \neq -1$. 
A potential partial level $\leq 1$ tower $(P, (p_i)_{i \leq k})$ will be identified with $ ((P), (d_i,p_i, \emptyset)_{ i \leq k})$ where $d_i = 1$ for any $i < k$, $d_k = 0$ if $p_k = -1$, $d_k = 1$ if $p_k \neq -1$. 
% If $(Q,(d,q,P))$ is a partial level $\leq 2$ tree, its cofinality will be changed from $d$, a member of $\se{0,1,2}$, to a level-3 tree of cardinality 1 $R$ such that the degree of $R(((0)))$ is $d$. 
% a function, still called $P$, that sends $p \in P$ to $1$.

 % Suppose $n\geq 1$. 
A level $\leq 2n+1$ tree $P$ is \emph{regular} iff $\comp{2n+1}{P}$ is regular and if $\comp{2n+1}{P} = \emptyset $, then $\dom(P) = \set{(2i,\emptyset)}{1 \leq i \leq  n}$. A \emph{partial level $\leq 2n+1$ tree} is a pair $(P, (d,p,Z))$ such that $P$ is a finite regular level $\leq 2n+1$ tree and either
\begin{enumerate}
\item $(d,p) \notin \dom(P)  $, there is a regular level $\leq 2n+1$ tree $P^{+} $ extending $P$ such that $\dom(P^{+}) = \dom(P) \cup \se{(d,p)}$ and $\comp{d}{P}_{\tree}^{+}(p) = Z$, or
\item $\comp{2n+1}{P} \neq \emptyset$, $d=0$, $p=-1$, $Z = \emptyset$.
\end{enumerate}
The \emph{degree} of $(P,(d,p,Z))$ is $d$. Put $\dom(P,(d,p,Z)) =  \dom(P) \cup \se{(d,p)}$. If $d>0$, a \emph{completion} of $(P,(d,p,Z))$ is a level $ \leq 2n+1$ tree $P^{+}$ extending $P$ such that $\dom(P^{+}) = \dom(P) \cup \se{(d,p)}$ and $\comp{d}{P}^{+}_{\tree}(p) = Z$.
The \emph{uniform cofinality} of $(P,(d,p,Z))$ is
\begin{displaymath}
  \ucf(P,(d,p,Z)),
\end{displaymath}
defined as follows:
\begin{enumerate}
\item $\ucf(P,(d,p,Z)) = \ucf(\comp{\leq 2n}{P},(d,p,Z))$ if $d \leq 2n$;
\item $\ucf(P,(d,p,Z)) = (0,-1)$ if $n=1$, $d=0$;
\item $\ucf(P,(d,p,Z)) = (1, p^{-})$ if $n=d=1$;
\item $\ucf(P,(d,p,Z)) = (2n+1, (p', Z, \overrightarrow{(e,z,Q)}))$ if $n \leq 1$, $d=2n+1$, $\comp{2n+1}{P}[p'] = (Z, \overrightarrow{(e,z,Q)})$, and $p'$ is the $<_{BK}$-least element of $\comp{2n+1}{P}\se{p,+,Z}$;
\item $\ucf(P,(d,p,Z)) = (2n+1, (p^{-}, Z, \overrightarrow{(e,z,Q)}))$ if $n \leq 1$, $d=2n+1$, $p \neq ((0))$, $\comp{2n+1}{P}[p^{-}] = (Z^{-}, \overrightarrow{(e,z,Q)})$, and $\comp{2n+1}{P}\se{p,+,Z} = \emptyset$;
\item $\ucf(P,(d,p,Z)) = (2n+1, \emptyset)$ if $n\geq1$, $d=2n+1$, $p = ((0))$.
\end{enumerate}
% The \emph{cofinality} of $(P,(d,p,Z))$ is...

A \emph{partial level $\leq 2n+1$ tower of discontinuous type} is a nonempty finite sequence $(\vec{P}, \overrightarrow{(d,p,Z)}) = (P_i, (d_i,p_i,Z_i))_{i \leq k}$ such that $\comp{2n+1}{P}_0 = \emptyset$, each $(P_i, (d_i,p_i,Z_i))$ is a partial level $\leq 2n+1$ tree, and $P_{i+1}$ is a completion of $(P_i, (d_i,p_i,Z_i))$. Its \emph{signature} is $((d_i,p_i))_{i < k}$. Its \emph{uniform cofinality} is $\ucf(P_k,(d_k,p_k,Z_k))$. A \emph{partial level $\leq 2n+1$ tower of continuous type} is $ (P_i, (d_i,p_i,Z_i))_{i < k} \concat (P_{*})$ such that either $k=0 \wedge \comp{2n+1}{P}_{*} = \emptyset$ or $ (P_i, (d_i,p_i,Z_i))_{i < k}$ is a partial level $\leq 2n+1$ tower of discontinuous type $\wedge P_{*}$ is a completion of $(P_{k-1},(d_{k-1},p_{k-1},Z_{k-1}))$. Its \emph{signature} is $(d_i,p_i)_{i < k}$. When $k>0$, its \emph{uniform cofinality} is $(1, q_{k-1})$ if $d_{k-1} = 1$, $(d_{k-1}, (p_{k-1} ) \concat \comp{d_{k-1}}{P}_{*}[p_{k-1}])$ if $d_{k-1}>1$. A \emph{potential partial level $\leq 2n+1$ tower} is $(P_{*},\overrightarrow{(d,p,Z)})$ such that for some $\vec{P} = (P_i)_{i \leq k}$, either $P_{*} = P_k \wedge (\vec{P}, \overrightarrow{(d,p,Z)})$ is a partial level $\leq 2n+1$ tower of discontinuous type or $(\vec{P},\overrightarrow{(d,p,Z)}) \concat (P_{*})$ is a partial level $\leq 2n+1$ tower of continuous type. The signature, (dis-)continuity type, uniform cofinality of $(P_{*},\overrightarrow{(d,p,Z)})$ are defined according to the partial level $\leq 2n+1$ tree generating $(P_{*}, \overrightarrow{(d,p,Z)})$.
\begin{displaymath}
  \ucf(P_{*}, \overrightarrow{(d,p,Z)})
\end{displaymath}
denotes the uniform cofinality of $(P_{*}, \overrightarrow{(d,p,Z)})$. If $(P_{*}, (d_i,p_i,Z_i)_{i \leq k})$ is a potential partial level $\leq 2n+1$ tower of discontinuous type, then $P^{+} $ is one of its completions iff $P^{+}$ is a completion of $(P_{*}, (d_k,p_k,Z_k))$.

A \emph{level-($2n+2$) tree} is a function $Q$ such that $\dom(Q)$ is a tree of level-1 trees, $\emptyset \in \dom(Q)$ and 
 for any $q \in \dom(Q)$, $(Q (q \res l))_{l \leq  \lh(q)}$ is a partial level $\leq 2n+1$ tower of discontinuous type.
In particular,  $Q(\emptyset) = (P, (2n+1,((0)),Z))$ where $\dom(P) = \dom(Z) = \set{(2i,\emptyset)}{1 \leq i \leq n} $.
 If $Q(q) = (P_q, (d_q, p_q,Z_q))$, we denote $ Q_{\tree}(q) = P_q$, $Q_{\node}(q) = (d_q,p_q)$, $Q[q] = (P_q, (d_{q \res l}, p_{q \res l}, Z_{q \res l})_{l \leq \lh(q)})$. So $Q[q]$ is a potential partial level $\leq 2n+1$ tower of discontinuous type. If $P$ is a completion of $Q(q)$, put $Q[q,P] = (P, (d_{q \res l}, p_{q \res l}, Z_{q \res l})_{l \leq q})$, which is a potential partial level $\leq 2n+1$ tower of continuous type. For $q \in \dom(Q)$, put $Q \se{q} = \set{a \in \omega^{<\omega}}{q \concat (a) \in \dom(Q)}$, which is a level-1 tree; if $P$ is a level $\leq 2n+1$ tree, put $Q \se{q,P} = \set{a \in Q\se{q}}{ Q_{\tree}(q \concat (a)) = P}$.

For $Q$ a level-($2n+2$) tree, let $\bardom(Q) = \dom(Q) \cup \set{q\concat (-1)}{q \in \dom(Q)}$. If $q \neq \emptyset$, denote $Q \se{q,-} = \se{q^{-} \concat (-1)} \cup \set{q^{-} \concat (a)}{Q_{\tree}(q^{-} \concat (a)) = Q_{\tree}(q) \wedge a <_{BK} q(\lh(q)-1)}$, $Q \se{q,+,P} = \se{q^{-}} \cup  \set{q^{-} \concat (a)}{Q_{\tree}(q^{-} \concat (a)) = P \wedge a >_{BK} q(\lh(q)-1)}$. $Q \se{q,+,P}$ is defined even when $q \notin \dom(Q)$.  For $q \in \bardom(Q)$, $q$ is \emph{of discontinuous type} if $q \in \dom(Q)$; $q$ is \emph{of continuous type} if $q \in \bardom(Q)\setminus \dom(Q)$. 
A \emph{$Q$-description} is a triple
\begin{displaymath}
  \mathbf{q} = (q, P, \overrightarrow{(d,p,Z)})
\end{displaymath}
such that $q \in \bardom(Q)$ and either $q$ is of discontinuous type $\wedge (P,\overrightarrow{(d,p,Z)}) = Q[q]$ or $q$ is of continuous type $\wedge (P,\overrightarrow{(d,p,Z)}) = Q[q^{-},P]$. A $Q$-description $(q,P,\overrightarrow{(d,p,Z)})$ is \emph{of (dis-)continuous type} iff $q$ is of (dis-)continuous type. The \emph{constant $Q$-description} is $(\emptyset) \concat Q[\emptyset]$. If a $Q$-description $\mathbf{q}  =  (q, P, \overrightarrow{(d,p,Z)}) $ is of discontinuous type and $P^{+}$ is a completion of $Q(q)$, then $\mathbf{q} \concat (-1,P^{+}) = (q \concat (-1), P^{+}, \overrightarrow{(d,p,Z)})$.
$Q$ is \emph{$\Pi^1_{2n+2}$-wellfounded} iff
\begin{enumerate}
\item $\forall q \in \dom(Q) ~ Q\se{q} $ is $\Pi^1_{1}$-wellfounded,
\item $\forall y \in [\dom(Q)]~ Q(y) \DEF \cup_{n<\omega} Q_{\tree}(y \res n)$ is not $\Pi^1_{2n+1}$-wellfounded. 
\end{enumerate}

A \emph{level $\leq 2n+2$ tree} is a tuple $Q = (\comp{1}{Q},\dots,\comp{2n+2}{Q})$ such that $\comp{d}{Q}$ is a level-$d$ tree for $1 \leq d \leq 2n+2$. $\comp{d}{Q}$ always stands for the level-$d$ component of a level $\leq 2n+2$ tree $Q$. $\comp{\leq d}{Q} $ denotes the level $\leq d$ tree $(\comp{1}{Q},\dots,\comp{d}{Q})$.
$\dom(Q) = \cup_d (\se{d} \times \dom(\comp{d}{Q}))$. $Q$ is regarded as a function sending $(d,q) \in \dom(Q)$ to $\comp{d}{Q}(q)$. $\bardom(Q) = \cup_d (\se{d} \times \bardom(\comp{d}{Q}))$.  $\desc(Q) = \cup_d (\se{d} \times \desc(\comp{d}{Q}))$ is the set of $Q$-descriptions. $(d,\mathbf{q}) \in \desc(Q)$ is \emph{of continuous type} iff $d \geq 2$ and $\mathbf{q}$ is of continuous type; otherwise, $(d,\mathbf{q})$ is \emph{of discontinuous type}. $Q$ is \emph{$\Pi^1_{2n+2}$-wellfounded} iff $\comp{\leq 2n+1}{Q}$ is $\Pi^1_{2n+1}$-wellfounded and $\comp{2n+2}{Q}$ is $\Pi^1_{2n+2}$-wellfounded.

Suppose $Q$ is a level $\leq 2n+2$ tree. An \emph{extended $Q$-description} is either a $Q$-description or of the form $(d, (q,P,\overrightarrow{(e,p,Z)}))$ such that $(d, (q \concat (-1), P, \overrightarrow{(e,p,Z)}))$ is a $Q$-description of continuous type. $\exdesc(Q)$ is the set of extended $Q$-descriptions. $(d,\mathbf{q}) \in \exdesc(Q)$ is \emph{regular} iff either $(d,\mathbf{q}) \in \desc(Q)$ of discontinuous type or $(d,\mathbf{q} ) \notin \desc(Q)$. 

A \emph{partial level $\leq 2n+2$ tree} is a pair $(Q, (d,q,P))$ such that $Q$ is a finite level $\leq 2n+2$ tree, and either
\begin{enumerate}
\item $(d,q) \notin \dom(Q)$, there is a level $\leq 2n+2$ tree $Q^{+}$ extending $Q$ such that $\dom(Q^{+}) = \dom(Q) \cup \se{(d,q)}$ and $\comp{d}{Q}^{+}_{\tree}(q) = P$, or
\item $(d,q,P) = (0,-1,\emptyset)$. 
\end{enumerate}
The \emph{degree} of $(Q,(d,q,P))$ is $d$. If $d>0$, a \emph{completion} of $(Q,(d,q,P))$ is a level $\leq 2n+2$ tree $Q^{+}$ extending $Q$ such that $\dom(Q^{+}) = \dom(Q) \cup \se{(d,q)}$ and $\comp{d}{Q}^{+}_{\tree}(q) = P$. When $n=1$, the uniform cofinality of $(Q,(d,q,P))$ has been defined in \cite{sharpIII}.   When $n \geq 1$, the \emph{uniform cofinality} of $(Q,(d,q,P))$ is
\begin{displaymath}
  \ucf(Q,(d,q,P)),
\end{displaymath}
defined as follows:
\begin{enumerate}
\item $\ucf(Q,{(d,q,P)}) = \ucf(\comp{\leq 2n}{Q}, {(d,q,P)})$ if $d \leq 2n$;
\item $\ucf(Q,{(d,q,P)}) = (2n+2, (\emptyset) \concat \comp{2n+2}{Q}[\emptyset]) $ if $d=2n+1$, $q = \max_{<_{BK}}(\dom(\comp{2n+1}{Q}))$;
\item $\ucf(Q,(d,q,P)) = (d, (q', P, \overrightarrow{(e,p,Z)}))$ if $2n+1 \leq d \leq 2n+2$, $\comp{d}{Q}[q'] = (P, \overrightarrow{(e,p,Z)})$, and $q'$ is the $<_{BK}$-least element of $\comp{}{Q}\se{q,+,P}$, $q' \neq q^{-}$;
\item $\ucf(Q,(d,q,P)) = (d, (q^{-}, P, \overrightarrow{(e,p,Z)}))$ if $2n+1 \leq d \leq 2n+2$, $\comp{d}{Q}[q^{-}] = (P^{-}, \overrightarrow{(e,p,Z)})$, and $\comp{d}{Q}\se{q,+,P} = \se{q^{-}}$.
\end{enumerate}
% Suppose $\ucf(Q,(d,q,P)) = (d^{*},\mathbf{q}^{*})$ and if $d^{*}>0$ then $\mathbf{q}^{*} = (q^{*}, (d_i,q_i,P_i)_{i < \lh(\mathbf{q})})$. 
% The \emph{cofinality} of $(Q,(d,q,P))$ is
% \begin{displaymath}
%   \cf(Q,(d,q,P)) 
% \end{displaymath}
% is a level-($2n+3$) tree $R$ such that $\dom(R) = \se{((0))}$, $R(((0))) = (Q^{(2n)}_0, (d^{**},q^{**},P^{**}))$, and
% \begin{enumerate}
% \item if $d^{*}=0$, then $d^{**} = 0$;
% \item\label{item:cf_1} if $d^{*} = 1$ and $\mathbf{q}^{*} = \min_{<_{BK}}(\comp{1}{Q})$ then $d^{**} = 1$;
% \item if $1 \leq d^{*} \leq 2$ and the assumption of~\ref{item:cf_1} does not hold, then $d^{**} = 2$;
% \item if $d^{*} = 2m+1$, $m>0$, $q^{*} = \min_{<_{BK}}  (\comp{2m+1}{Q}\se{\emptyset, ...})$.....
% \end{enumerate}

A \emph{partial level $\leq 2n+2$ tower of discontinuous type} is a nonempty 
finite sequence $(Q_i,(d_i,q_i,P_i))_{1 \leq i \leq k}$ such that $\dom(Q_1) = \set{(2i,\emptyset)}{1 \leq i \leq n+1}$, each $(Q_i,(d_i,q_i,P_i))$ is a partial level $\leq 2n+2$ tree, and each $Q_{i+1}$ is a completion of $(Q_i,(d_i,q_i,P_i))$.  Its \emph{signature} is $(d_i,q_i)_{1 \leq i < k}$. 
Its \emph{uniform cofinality} is $\ucf(Q_k, (d_k,q_k, P_k))$. 
A \emph{partial level $\leq 2n+2$ tower of continuous type} is $(Q_i,(d_i,q_i,P_i))_{1 \leq i < k} \concat (Q_{*})$ such that either $k=0 \wedge Q_{*}$ is the level $\leq 2n+2$ tree with domain $\set{(2j,\emptyset)}{1 \leq j \leq n+1}$ or $(Q_i,(d_i,q_i,P_i))_{1 \leq i < k}$ is a partial level $\leq 2n+2$  tower of discontinuous type $\wedge Q_{*}$ is a completion of $(Q_{k-1}, (d_{k-1},q_{k-1},P_{k-1}))$. 
Its \emph{signature} is $(d_i,q_i)_{1 \leq i < k}$. If $k>0$, its \emph{uniform cofinality} is $(1, q_{k-1})$ if $d_{k-1}=1$,  $(d_{k-1}, (q_{k-1}) \concat \comp{d_{k-1}}{Q}[q_{k-1}])$ if $d_{k-1} >1$.
A \emph{potential partial level $\leq 2n+2$ tower} is $(Q_{*},\overrightarrow{(d,q,P)}) $ such that for some $\vec{Q} = (Q_i)_{1 \leq i \leq k}$, either $Q_{*} = Q_k$ $\wedge$ $(\vec{Q}, \overrightarrow{(d,q,P)})$ is a partial level $\leq 2n+2$ tower of discontinuous type or $(\vec{Q}, \overrightarrow{(d,q,P)}) \concat (Q_{*})$ is a partial level $\leq 2n+2$ tower of continuous type. The signature, (dis-)continuity type, uniform cofinality of $(Q_{*},\overrightarrow{(d,q,P)})$ are defined according to the partial level $\leq 2n+2$ tree generating $(Q_{*}, \overrightarrow{(d,q,P)})$. 
\begin{displaymath}
  \ucf(Q_{*}, \overrightarrow{(d,q,P)})
\end{displaymath}
denotes the uniform cofinality of $(Q_{*}, \overrightarrow{(d,q,P)})$.

A \emph{level-($2n+3$) tree} is a function $R$ such that $\emptyset \notin \dom(R)$, $\dom(R)\cup \se{\emptyset}$ is a tree of level-1 trees and for any $r \in \dom(R)$, $(R(r \res l))_{1 \leq l \leq \lh(r)}$ is a partial level $\leq 2n+2$ tower of discontinuous type. 
If $R(r) = (Q_r, (d_r,q_r, P_r))$, we denote  $R_{\tree}(r) = Q_r$, $R_{\node}(r) = (d_r,q_r)$, 
 $R[r] = (Q_r,  (d_{r \res l}, q_{r \res l}, P_{r \res l})_{1 \leq l \leq  \lh(r)})$. $R[r]$ is a potential partial level $\leq 2n+2$ tower of discontinuous type. If $Q$ is a completion of $R(r)$, put $R[r,Q] = (Q, (d_{r \res l}, q_{r \res l}, P_{r \res l})_{1 \leq l \leq  \lh(r)})$, which is a potential partial level $\leq 2n+2$ tower of continuous type. For $r \in \dom(R)\cup \se{\emptyset}$, put $R\se{r} = \set{a \in \omega^{<\omega}}{r \concat (a) \in \dom(R)}$, which is a level-1 tree. For $r \in \dom(R)$ and a level $\leq 2n+2$ tree $Q$, put $R \se{r,Q} = \set{a \in R\se{r}}{R_{\tree}(r) = Q}$. $R$ is \emph{regular} iff $((1)) \notin \dom(R)$.

Suppose  $R$ is a level-($2n+3$) tree. Let $\bardom(R) = \dom(R)\cup \set{r \concat (-1)}{r \in \dom(R) }$. For $r \in \bardom(R)$, $r$ is \emph{of discontinuous type} if $r \in \dom(R)$; $r$ is \emph{of continuous type} if $r \in \bardom(R)\setminus \dom(R)$. 
 For $r \in \dom(R)$, put $R\se{r,- } = \se{r^{-} \concat (-1)} \cup \set{r^{-}\concat (a)}{R_{\tree}(r^{-}\concat (a)) = R_{\tree}(r), a <_{BK} r (\lh(r)-1)}$. For $r \in \dom(R)$ and a level $\leq 2n+2$ tree $Q$, put
 $R\se{r,+,Q } = \se{r^{-}} \cup \set{r^{-}\concat (a)}{R_{\tree}(r^{-}\concat (a)) = Q, a >_{BK} r (\lh(r)-1)}$. The \emph{constant $R$-description} is $\emptyset$, which is of discontinuous type. An \emph{$R$-description} is either the constant $R$-description or a triple
\begin{displaymath}
  \mathbf{r} = (r, Q, \overrightarrow{(d,q,P)})
\end{displaymath}
such that $r \in \bardom(R)$ and either $r$ is of discontinuous type $\wedge (Q,\overrightarrow{(d,q,P)}) = R[r]$ or $r$ is of continuous type $\wedge (Q,\overrightarrow{(d,q,P)}) = R[r^{-},Q]$. A non-constant $R$-description $(r,Q,\overrightarrow{(d,q,P)})$ is \emph{of (dis-)continuous type} iff $r$ is of (dis-)continuous type.  If an $R$-description $\mathbf{r}  =  (r, Q, \overrightarrow{(d,q,P)}) $ is of discontinuous type and $Q^{+}$ is a completion of $R(r)$, then $\mathbf{r} \concat (-1,Q^{+}) = (r \concat (-1), Q^{+}, \overrightarrow{(d,q,P)})$.
An \emph{extended $R$-description} is either an $R$-description or a triple $(r, Q, \overrightarrow{(d,q,P)})$ such that $(r\concat (-1), Q, \overrightarrow{(d,q,P)})$ is an $R$-description of continuous type. 
$\exdesc(R)$ is the set of  extended $R$-descriptions. 
An extended $R$-description $\mathbf{r}$ is \emph{regular} iff either $\mathbf{r} \in \desc(R)$ of discontinuous type or $\mathbf{r} \notin \desc(R)$.
A \emph{generalized $R$-description} is either $(\emptyset,\emptyset,\emptyset)$ or of the form
\begin{displaymath}
  \mathbf{A} = (\mathbf{r}, \pi, T)
\end{displaymath}
so that $\mathbf{r} = (r, Q, \overrightarrow{(d,q,P)}) \in \desc(R) \setminus\se{\emptyset}$, $T$ is a finite level $\leq 2n+2$ tree, 
$\pi$ factors $(Q,T)$. 
 $\exexdesc(R)$ is the set of generalized $R$-descriptions.

$R$ is \emph{$\Pi^1_{2n+3}$-wellfounded} iff
\begin{enumerate}
\item $\forall r \in \dom(R) \cup \se{\emptyset}~ R\se{r}$ is $\Pi^1_1$-wellfounded, and
\item 
 $\forall z\in [\dom(R)] ~ R(z) \DEF \cup_{n<\omega}(R_{\tree}(z \res n))_{1 \leq n<\omega}$ is not $\Pi^1_{2n+2}$-wellfounded.
 \end{enumerate}
If $R$ is a level-($2n+3$) tree, $\mathcal{L}^R$ is the language $\set{\underline{\in}, \underline{c_r}}{r \in \dom(R)}$, and $\mathcal{L}^{\underline{x},R}$ is the language $\mathcal{L}^R \cup \se{\underline{x}}$. 

A \emph{level $\leq 2n+3$ tree} is a tuple $R = (\comp{1}{R},\dots, \comp{2n+3}{R})$ such that $\comp{d}{R}$ is a level-$d$ tree for any $d$. $\comp{d}{R}$ always stands for the level-$d$ component of $R$. $\comp{\leq d}{R}$ stands for the level $\leq d$ tree $(\comp{1}{R},\dots,\comp{d}{R})$.  If $Z$ is a level $\leq 2n+2$ tree and $W$ is a level-($2n+3$) tree, then $Z \oplus W$ denotes the level $\leq 2n+3$ tree $(\comp{1}{Z},\dots,\comp{2n+2}{Z},W)$.  A level $\leq 2n+3$ tree $R$ is \emph{$\Pi^1_{2n+3}$-wellfounded} iff $\comp{\leq 2n+2}{R}$ is $\Pi^1_{2n+2}$-wellfounded and $\comp{2n+3}{R}$ is $\Pi^1_{2n+3}$-wellfounded. If $R$ is level $\leq 2n+3$ tree, define $\dom(R) = \cup_d \se{d} \times \dom(\comp{d}{R})$, $\desc(R) = \cup_d \se{d} \times \desc(\comp{d}{R})$. $R$ is regarded as a function sending $(d,r)$ to $\comp{d}{R}(r)$.

If $R$ is a level $\leq m$ tree and $(d,r )\in \dom(R)$, $d>1$, put $R[d,r] = \comp{d}{R}[r]$. 

Suppose $\sigma$ factors level-1 trees $(P,W)$. % $\sigma$ is said to be \emph{$W$-discontinuous at $(0,-1)$}. 
If $p \in P$, then $(\sigma,W)$ is  \emph{continuous at $(1,p)$} iff either $\sigma(p) = \min(\prec^W)$ or $\pred_{\prec^W}(\sigma(p)) \in \ran(\sigma)$; otherwise $(\sigma,W)$ is discontinuous at $(1,p)$.

Suppose $Q,T$ are level-$m$ trees. 
A map $\pi$ is said to \emph{factor $(Q,T)$} iff $\dom(\pi) = \dom(Q)$, $q <_{BK}q' \eqiv \pi(q) <_{BK} \pi(q')$, $q \subseteq q' \eqiv \pi(q) \subseteq \pi(q')$, and for any $q \in \dom(Q)$, $Q(q) = T(\pi(q))$. 
If $\pi$ factors $(Q,T)$, $\pi$ is allowed to act on extended $Q$-descriptions as well. If $m>1$ and $\mathbf{q}= (q,P, \overrightarrow{(d,p,Z)}) \in \exdesc(Q)$ then
\begin{displaymath}
  \pi(\mathbf{q}) =
  \begin{cases}
    ({\pi}(q), P, \overrightarrow{(d,p,Z)}) & \text{if } q \text{ is of discontinuous type},\\
    ({\pi}(q^{-}) \concat (-1), P, \overrightarrow{(d,p,Z)}) & \text{otherwise.}
  \end{cases}
\end{displaymath}

Suppose $Q,T$ are level $\leq m$ trees. $\pi$ is said to \emph{factor $(Q,T)$} iff $\dom(\pi) = \dom(Q)$ and there is $(\comp{d}{\pi})_{1 \leq d \leq m}$ such that for any $d$, $\comp{d}{\pi}$ factors $(\comp{d}{Q}, \comp{d}{T})$ and $\pi(d,q) = (d, \comp{d}{\pi}(q))$ for any $q \in \dom(\comp{d}{Q})$. If $\pi$ factors $(Q,T)$, $\comp{d}{\pi}$ has this fixed meaning. 
Suppose $Q,T$ are both finite and suppose $\pi$ factors $(Q,T)$. $\pi$ is allowed to act on extended $Q$-descriptions as well. If $(d,\mathbf{q}) \in \exdesc(Q)$, then  $\pi(d,\mathbf{q}) = (d, \comp{d}{\pi}(\mathbf{q}))$. \emph{level-$m$ tree isomorphisms} and \emph{level $\leq m$ tree isomorphisms} have obvious definitions. If $Q$ is a level-$m$ or a level $\leq m$ tree, $\id_Q$ is the identity tree isomorphism between $Q$ and itself. If $\pi$ factors $(Q,T)$ and $\vec{\beta} = (\beta_{(d,t)})_{(d,t ) \in \dom(T)} $ is a tuple indexed by $\dom(T)$, then $\vec{\beta}_{\pi} = (\beta_{\pi, (d,q)})_{(d,q) \in \dom(Q)}$, where $\beta_{\pi, (d,q)} = \beta_{\pi(d,q)}$.

Suppose $Q,T$ are level $\leq m$ trees and $\pi$ factors $(Q,T)$. 
 $(\pi,T)$ is said to be \emph{discontinuous} at $(0,-1)$. Suppose $(d,\mathbf{q}) \in \exdesc(Q)$ is regular.  $(\pi,T)$ is \emph{continuous at $(d,\mathbf{q})$} iff one of the following holds:
\begin{enumerate}
\item $d=1$, either $\comp{1}{\pi}(\mathbf{q}) = \min (\prec^{\comp{1}{T}})$ or $\pred_{\prec^{\comp{1}{T}}} ( \comp{1}{\pi}(\mathbf{q})) \in \ran( \comp{1}{\pi})$.
%\item $d=2$, $\mathbf{q} = (\emptyset,\dots)$, either $\comp{1}{T} = \emptyset$ or $\max(\prec^{\comp{1}{T}}) \in \ran(\comp{1}{\pi})$.
\item $d=2i$, $\mathbf{q} = (\emptyset, \dots) \in \desc(\comp{2i}{Q})$, either $\comp{2i-1}{T}=\emptyset$ or $\max_{<_{BK}}( \dom(\comp{2i-1}{T})) \in \ran(\comp{2i-1}{\pi})$.
\item $d>1$, $\mathbf{q} = (q,\dots) \in \desc(\comp{d}{Q})$, $q \neq \emptyset$, and letting $t' = \max_{<_{BK}} \comp{d}{T} \se{ \comp{d}{\pi}(q),-}$, then either $t' = \comp{d}{\pi}(q^{-}) \concat (-1)$ or $t' \in \ran(\comp{d}{\pi})$.
\item $d>1$, $\mathbf{q} = (q, P,\dots) \notin \desc(\comp{d}{Q})$, and letting $a = \max_{<_{BK}} (\comp{d}{T}\se{\comp{d}{\pi}(q),P} \cup \se{-1})$, then either $a=-1$ or $\comp{d}{\pi}(q) \concat (a) \in \ran(\comp{d}{\pi})$.
\end{enumerate}
Otherwise, $(\pi,T)$ is \emph{discontinuous at $(d,\mathbf{q})$}. If  $(\pi,T)$  is discontinuous at $(d,\mathbf{q})$, the \emph{decomposition} of $(\pi,T)$ is $(\pi^{+},Q^{+})$ such that $Q^{+}$ is a level $\leq m$ tree extending $Q$, $\pi^{+}$ factors $(Q^{+},T)$, $\pi^{+}$ extends $\pi$, and
\begin{enumerate}
    \item if $d=1$, then $\dom(Q^{+}) \setminus \dom(Q) = \se{(1, q^{+})}$, $q =  q^{+} \res\comp{1}{Q} $,  
$\comp{1}{\pi}^{+} (q^{+}) = \suc_{\prec^{\comp{1}{T}}} (\comp{1}{\pi}(\pred_{\prec^{\comp{1}{Q}}}(q^{+})))$;
    \item if $d=2$ and $\mathbf{q} = (\emptyset,\emptyset, ((0)))$, then $\dom(Q^{+} ) \setminus \dom(Q)= \se{ (1, q^{+})}$, $\emptyset  = q^{+} \res \comp{1}{Q}$, $\comp{1}{\pi}^{+}(q^{+}) = \min_{\prec^{\comp{1}{T}}} \set{a}{\forall r \in \dom(\comp{1}{Q})~\comp{1}{\pi}(r) \prec^{\comp{1}{T}} a}$;    
\item if $d=2i>2$ and $\mathbf{q} = (\emptyset,\dots) \in \desc(\comp{2i}{Q})$, then $\dom(Q^{+} ) \setminus \dom(Q)= \se{ (2i-1, q^{+})}$, $\lh(q^{+} ) = 1$, 
$\emptyset =  q^{+} (0) \res \comp{2i-1}{Q}\se{\emptyset} $, $\comp{2i-1}{\pi}^{+}(q^{+}) = T( t^{+})$, $\lh(t^{+}) = 1$,  $t^{+}(0)= \min_{\prec^{\comp{2i-1}{T}\se{\emptyset}}} \set{a}{\forall r \in \dom(\comp{2i-1}{Q})~\comp{2i-1}{\pi}(r) \prec^{\comp{2i-1}{T}\se{\emptyset}} (a)})$;
    \item if $d>1$ and  $\mathbf{q} = (q,P, \dots) \in \desc(\comp{d}{Q})$, $q\neq \emptyset$, then $\dom(Q^{+}) \setminus \dom(Q) = \se{(d, q^{+})}$, $q^{+} = \max_{<_{BK}} \comp{d}{Q}^{+}\se{q, -}$, and $\comp{d}{\pi}^{+}(q^{+}) = \comp{d}{\pi}(q^{-}) \concat (a)$, $a=
\min_{<_{BK}} 
\set{b}{ \comp{d}{Q}_{\tree}(q \concat (a))  =  P \wedge \forall r \in \comp{d}{Q}(q,-)\setminus \se{q^{-} \concat (-1)}~  \comp{d}{\pi}(r) <_{BK} \comp{d}{\pi}(q^{-}) \concat (b)}$;
    \item if $d>1$ and  $\mathbf{q} = (q,P, \dots) \notin \desc(\comp{d}{Q})$, then $\dom(Q^{+}) \setminus \dom(Q) = \se{(d, q^{+})}$, $q^{+} = q \concat (
\max_{<_{BK}} \comp{d}{Q}^{+}\se{q})$, 
$\comp{d}{\pi}^{+}(q^{+}) = \comp{d}{\pi}(q) \concat (a)$, 
$a= \min_{<_{BK}} \set{b}{ \comp{d}{Q}_{\tree}(q \concat (a)) = P \wedge  \forall  c \in \comp{d}{Q}\se{q,P} ~ \comp{d}{\pi}(q \concat (c)) <_{BK} \comp{d}{\pi}(q) \concat (b)}$.
\end{enumerate}
If $(\pi,T)$ is discontinuous at $(d,\mathbf{q})$,
then $\pred(\pi, T, (d,\mathbf{q}))$ is a node in $\dom(T)$ defined as follows:
\begin{enumerate}
\item If $d=1$, then $\pred(\pi, T, (d,\mathbf{q})) = (1, \pred_{\prec^{\comp{1}{T}}}( \comp{1}{\pi}(\mathbf{q})))$.
% \item If $d=2$ and $\mathbf{q} = (\emptyset,\emptyset, ((0)))$, then $\pred(\pi, T, (d,\mathbf{q})) = (1, \max_{<_{BK}}\comp{1}{T})$.
\item If $d=2i$ and $\mathbf{q} = (\emptyset,\dots) \in \desc(\comp{2i}{Q})$, then $\pred(\pi,T,(d,\mathbf{q})) = (2i-1, \max_{<_{BK}} \dom(\comp{2n-1}{T}))$.
\item If $d>1$ and $\mathbf{q} = (q,\dots) \in \desc(Q)$, $q \neq \emptyset$, then $\pred(\pi,{T},(d,\mathbf{q})) = (d, \max_{<_{BK}} \comp{d}{T}\se{\comp{d}{\pi}(q),-})$.
\item If $d>1$ and $\mathbf{q} = (q, P,\dots) \notin \desc(Q)$, then $\pred(\pi, T, (d,\mathbf{q})) = (d, q \concat (a))$, $a = \max_{<_{BK}} \comp{d}{T}\se{\comp{d}{\pi}(q),P}$.
\end{enumerate}
If $(d,\mathbf{q}) = (d, (q,\dots)) \in \desc(Q)$ then put $\pred(\pi,T, (d,q)) = \pred(\pi,T, (d, \mathbf{q}))$.

Suppose $R$ is a finite level $\leq 2n+1$ tree. For $\mathbf{A} = (\mathbf{r}, \pi, T)\in \exexdesc(\comp{2n+1}{R})$, define its uniform cofinality 
\begin{displaymath}
  \ucf(\mathbf{A})
\end{displaymath}
as follows: If $\mathbf{r} = \emptyset$ then $\ucf(\mathbf{A}) = \mathbf{A}$. Suppose $\mathbf{r} = (r, Q, \overrightarrow{(d,q,P)})\neq \emptyset$, $\lh(r) = k$. 
\begin{enumerate}
\item If $r$ is of continuous type and  $(\pi,T)$ is continuous at $(d_{k-1},q_{k-1})$, then $\ucf(\mathbf{A}) = (r^{-}, R_{\tree}(r^{-}), \overrightarrow{(d,q,P)})$.
\item If $r$ is of continuous type and $(\pi,T)$ is continuous at $(d_{k-1},q_{k-1})$, then $\ucf(\mathbf{A}) = (r^{-}, Q, \overrightarrow{(d,q,P)})$.
\item If $r$ is of discontinuous type and $(\pi,T)$ is continuous  at $(d_{k-1},q_{k-1})$, then $\ucf(\mathbf{A}) = \mathbf{r}$.
\item If $r$ is of discontinuous type and $(\pi,T)$ is continuous at $(d_{k-1},q_{k-1})$, then $\ucf(\mathbf{A}) = (r, Q^{+}, \overrightarrow{(d,q,P)})$, where $(Q^{+},\pi^{+})$ is the decomposition of $(\pi,T)$.
\end{enumerate}

We fix the notation for the trivial level $\leq 2n$ tree:
\begin{itemize}
\item $Q^{(2n)}_0$ is the level $\leq 2n$ tree with domain $\set{(2i,\emptyset)}{1 \leq i \leq n}$.
% \item $R^0_{2n+1}$ is the level $\leq 2n+1$ tree
\end{itemize}

To be consistent with the higher levels, we rename some definitions in \cite{sharpIII} concerning descriptions. 

Suppose $Q$ is a level $\leq 2$ tree and $W$ is a level-1 tree. Put $W' = (W)$, a level $\leq 1$ tree. Then $\desc(Q, W',*) = \desc(Q,W)$ is the set of $(Q,W',*)$-descriptions. Suppose $\mathbf{D} \in \desc(Q,W)$. 
If $\sign(\mathbf{D}) = (w_i)_{i < k}$ viewing $\mathbf{D}$ as a $(Q,W)$-description, then the \emph{signature} of $\mathbf{D}$ is $\sign (\mathbf{D})= ((1,w_i))_{i < k}$. If $\ucf(\mathbf{D}) = -1$ viewing $\mathbf{D}$ is a $(Q,W)$-description then $\ucf(\mathbf{D}) = (0,-1)$. If $\ucf(\mathbf{D}) = w_{*}$ viewing $\mathbf{D}$ is a $(Q,W)$-description then $\ucf(\mathbf{D}) = (1, w_{*})$.

Inductively, we make the syntactical definitions of descriptions. In the rest of this section, we suppose that $n_1,n_2,n_3$ are consecutive entries of the following list:
\begin{equation}
  \label{eq:3}
    1, 2, 2, 3, 4, 4, 5, 6, 6, 7, 8, 8, 9\dots
\end{equation}
% Suppose $n_0$ is the greatest even number $\leq n_1$.

Suppose $Y$ is a level $\leq n_3$ tree and $T$ is a level $\leq n_2$ tree. Then a \emph{$(Y,T,-1)$-description} is a $(\comp{\leq 2}{Y}, \comp{\leq 1}{T}, *)$-description. The \emph{constant $(Y,T,*)$-description} is
\begin{enumerate}
\item $(n_3, (\emptyset,\emptyset))$, if $n_3$ is odd;
\item $(n_3, ( (\emptyset) \concat \comp{n_3}{Y}[\emptyset], \tau))$, if $n_3$ is even, $\tau$ factors $(\comp{n_3}{Y}_{\tree}(\emptyset), T, *)$.
\end{enumerate}
Suppose $(Q,\overrightarrow{(d,q,P)})$ is a potential partial level $\leq n_1'$ tree, $n_1' \leq n_1$. 
If $n_1'< n_1$, then a $(Y,T, (Q,\overrightarrow{(d,q,P)}))$-description is a $(\comp{\leq n_2}{Y}, \comp{\leq n_1}{T} , (Q,\overrightarrow{(d,q,P)}))$-description. Suppose now $n_1' = n_1$. Put $\overrightarrow{(d,q,P)}=(d_i,q_i,P_i)_{m_0 \leq i \leq m}$, where $m_0 =0$ if $n_1$ is odd, $m_0 = 1$ if $n_1$ is even. Suppose $m>0$. 
 A \emph{$(Y,T, {(Q,\overrightarrow{(d,q,P)})})$-description} is of the form
\begin{displaymath}
  \mathbf{B} = (b, (\mathbf{y}, \pi))
\end{displaymath}
such that
\begin{enumerate}
\item If $n_2$ is odd then $b \in \se{n_3-1, n_3}$. If $n_2$ is even then $d=n_3$.
\item $\mathbf{y} \in \desc(\comp{d}{Y})$ is not the constant $\comp{d}{Y}$-description. Put $\mathbf{y} = (y, X, \overrightarrow{(e,x,W)})$, $\lh(y) = k$. If $d$ is even, put $\overrightarrow{(e,x,W)} = (e_i,x_i,W_i)_{i < \lh(\vec{x})}$. if $d$ is odd, put  $\overrightarrow{(e,x,W)} = (e_i,x_i,W_i)_{1 \leq i \leq \lh(\vec{x})}$.
\item If $d=n_3$ then $\sigma$ factors $(X, T, Q)$. If $d=n_3-1$ then $\sigma$ factors $(X, \comp{\leq n_1}{T}, Q)$.
% \item If $d$ is even, then 
\item The contraction of $(\sign(\pi(e_i,x_i)))_{i < k}$ is the signature of $(Q,\overrightarrow{(d,q,P)})$. (When $(e_0,x_0)$ is undefined, the contraction of  $(\sign(\pi(e_i,x_i)))_{i < k}$ simply means the contraction of  $(\sign(\pi(e_i,x_i)))_{1 \leq i < k}$.)
% . If $d$ is odd, then
% the contraction of $(\sign_2(\pi(e_i,x_i)))_{1 \leq i < k}$ is the signature of $(Q,\overrightarrow{(d,q,P)})$. 
\item If $y$ is of continuous type and $(e_{k-1},x_{k-1})$ does not appear in the contraction of $(\sign(\pi(e_i,x_i)))_{ i < k} \concat (\sign(\pi(e_k,x_k)^{-}))$, then $\pi(e_{k-1}, x_{k-1})$ is of discontinuous type. 
\item Put $\ucf(X,\overrightarrow{(e,x,W)}) = (e_{*},\mathbf{x}_{*})$.
  \begin{enumerate}
  \item If $e_{*} = 0$ then $d_m = 0$.
  \item If $e_{*} = 1$ then $\ucf(\pi(1, \mathbf{x}_{*})) = \ucf (Q, \overrightarrow{(d,q,P)})$.
  \item If $e_{*} >1 $, $\mathbf{x}_{*} = (x_{*}, \dots) \in \desc(X)$, then $\ucf(\pi(e_{*}, x_{*})) = \ucf(Q, \overrightarrow{(d,q,P)})$. 
  \item If $e_{*} >1 $, $\mathbf{x}_{*} = (x_{*}, W_{*},\dots) \notin \desc(X)$, then $\ucf^{W_{*}}(\pi(e_{*}, x_{*})) = \ucf(Q, \overrightarrow{(d,q,P)})$. 
  \end{enumerate}
\end{enumerate}
We often abbreviate $(b,(\mathbf{y},\pi))$ by $(b,\mathbf{y}, \pi)$. 
If $Q'$ is a level $\leq n_1'$ tree, a $(Y,T,Q')$-description is a $(Y,T,({Q'},\overrightarrow{(d',q',P')}))$-description for some potential partial level $\leq n_1'$ tower $({Q'},\overrightarrow{(d',q',P')})$ of discontinuous type. A $(Y,T,*)$-description is either the constant $(Y,T,*)$-description or a $(Y,T,Q')$-description,  where either $Q' = -1$ or $Q'$ is a level $\leq n_1'$ tree, $n_1' \leq n_1 $. $\desc(Y,T,({Q},\overrightarrow{(d,q,P)}))$, $\desc(Y,T,Q)$, $\desc(Y,T,*)$ denote the sets of relevant descriptions. 

Suppose $(Q,\overrightarrow{(d,q,P)})$ is a potential partial level $\leq n_1$ tower of discontinuous type, $\overrightarrow{(d,q,P)} = (d_i,q_i,P_i)_{m_0 \leq i \leq m}$,  and $\mathbf{B}\in \desc(Y,T,(Q,\overrightarrow{(d,q,P)}))$, $\mathbf{y} = (y,X, \overrightarrow{(e,x,W)})$, $\lh(y) = k$,  $\overrightarrow{(e,x,W)} = (e_i,x_i,W_i)_i$. The \emph{signature} of $\mathbf{B}$ is
\begin{displaymath}
  \sign (\mathbf{B}) = \text{the contraction of } (\sign_{*}^Q (\pi(e_i,x_i)))_{ i \leq k-1}.
\end{displaymath}
$\mathbf{B}$ is \emph{of continuous type} iff $y$ is of continuous type and $\pi(e_{l-1}, x_{l-1})$ is of $*$-$Q$-continuous type. Otherwise, $\mathbf{B}$ is \emph{of discontinuous type}. The uniform cofinality of $\mathbf{B}$ is
\begin{displaymath}
  \ucf(\mathbf{B}),
\end{displaymath}
defined as follows:
\begin{enumerate}
\item  If $\ucf(X,\overrightarrow{(e,x,W)}) = (0,-1)$ then $\ucf(\mathbf{B}) = (0,-1)$.
\item If $\ucf(X,\overrightarrow{(e,x,W)}) = (1, x_{*})$ then $ \ucf(\mathbf{B}) = \ucf_{*}^Q( \pi(e_{*},x_{*}))$.
\item If $\ucf(X,\overrightarrow{(e,x,W)}) = (e_{*},\mathbf{x}_{*})$, $e_{*}>1$, $\mathbf{x}_{*} = (x_{*},\dots)$, then $ \ucf(\mathbf{B}) = \ucf_{*}^Q( \pi (e_{*},x_{*}))$.
\end{enumerate}
If $d_m>0$, then
$\mathbf{B}$ is said to be \emph{of plus-discontinuous type}, and if  $\ucf(X,\overrightarrow{(e,x,W)}) = (e_{*},\mathbf{x}_{*})$, $\mathbf{x}_{*} = x_{*}$ if $e_{*}=1$, $\mathbf{x}_{*} = (x_{*},\dots)$ if $e_{*}>1$, $Q^{+}$ is a completion of $(Q,(d_m,q_m,P_m))$, put
\begin{displaymath}
  \ucf^{Q^{+}}(\mathbf{B}) = \ucf_{*}^{Q^{+}}(  \pi ( e_{*},x_{*}) ).
\end{displaymath}
The \emph{$*$-signature} of $\mathbf{B}$ is
\begin{displaymath}
  \sign_{*}(\mathbf{B}) =
  \begin{cases}
    ((d, y \res i))_{1 \leq i \leq k-1} & \text{if $y$ is of continuous type},\\
    ((d, y \res i))_{1 \leq i \leq k} & \text{if $y$ is of discontinuous type}.
  \end{cases}
\end{displaymath}
$\mathbf{B}$ is \emph{of $*$-$T$-(dis-)continuous type} iff $\pi$ is $T \otimes Q$-(dis-)continuous at $\ucf(X,\overrightarrow{(e,x,W)})$. The \emph{$*$-$T$-uniform cofinality} of $\mathbf{B}$ is
\begin{displaymath}
  \ucf^T_{*}(\mathbf{B}),
\end{displaymath}
defined as follows. If $y$ is of continuous type,
\begin{enumerate}
\item if $\mathbf{B}$ is of $*$-$T$-continuous type, then $\ucf^T_{*}(\mathbf{B}) = (d, (y^{-}, Y_{\tree}(y^{-}), \overrightarrow{(e,x,W)}))$;
\item if $\mathbf{B}$ is of $*$-$T$-discontinuous type, then $\ucf^T_{*}(\mathbf{B}) = (d, (y^{-}, X, \overrightarrow{(e,x,W)}))$.
\end{enumerate}
If $y$ is of discontinuous type,
\begin{enumerate}
 \item if $\mathbf{B}$ is of $*$-$T$-continuous type, then $\ucf^T_{*}(\mathbf{B}) = (d, \mathbf{y})$;
\item if $\mathbf{B}$ is of $*$-$T$-discontinuous type, then $\ucf^T_{*}(\mathbf{B}) = (d, (y, X^{+}, \overrightarrow{(e,x,W)}))$, $X^{+}$ is a completion of $Y(y)$,  $X^{+}(Y_{\node}(y)) = T \otimes Q ( \pred(\pi, T \otimes Q, \ucf(Y(y))))$.
\end{enumerate}

Suppose $(\vec{Q},\overrightarrow{(d,q,P)})=(Q_i,(d_i,q_i,P_i))_{m_0 \leq i \leq m}$ is a partial level $\leq n_1$ tower and $\mathbf{B} = (b, \mathbf{y}, \pi) \in \desc(Y,T, (Q_m, \overrightarrow{(d,q,P)}) )$.  
Define $\lh(\mathbf{B}) = \lh(\mathbf{q})$. 
Define
\begin{displaymath}
  \mathbf{B} \iniseg \mathbf{B}'
\end{displaymath}
iff  $\mathbf{B} = \mathbf{B}' \res \bar{m}$ for some $\bar{m} < \lh(\mathbf{B}')$. Define $\iniseg^{Y,T} = \iniseg \res \desc(Y,T,*)$.

% $\mathbf{B} \res 0$ is the constant $(Y,T,*)$-description. 
Suppose   $\mathbf{y} = (y,X,\overrightarrow{(e,x,W)})$, $0 < \bar{m}<m$. Then
\begin{displaymath}
  \mathbf{B} \res (Y,T,Q_{\bar{m}}) \in \desc(Y,T, (Q_{\bar{m}}, (d_i,q_i,P_i)_{m_0 \leq i \leq \bar{m}}))
\end{displaymath}
is defined by the following: letting $l$ be the least such that $\pi(e_l,x_l) \notin \desc(T, Q_{\bar{m}}, *)$, letting $\ucf(X(e_l,x_l)) = (e_{*},\mathbf{x}_{*}) $, $\mathbf{x}_{*} = x_{*}$ if $e_{*} = 1$, $\mathbf{x}_{*} = (x_{*},\dots)$ if $e_{*}>1$, and letting $\mathbf{C} \in \desc(T, Q_{\bar{m}}, *)$ be such that $\mathbf{C} = \pi(e_*,x_*) \res (T, Q_{\bar{m}})$, then
\begin{enumerate}
\item if $\mathbf{C} \neq \pi(e_*,x_*)$, then $\mathbf{B} \res (Y,T,Q_{\bar{m}}) = (\mathbf{y} \res l \concat (-1,X^{+}),  \bar{\pi})$, where $\bar{\pi}$ and $\pi$ agree on $Y_{\tree}(y \res l)$, $\bar{\pi}(e_l,x_l) = \mathbf{C}$, $\bar{\pi}$ factors $(X^{+}, T, Q_{\bar{m}})$;
\item if $\mathbf{C} = \pi(e_*,x_*)$, then $\mathbf{B} \res (Y,T,Q_{\bar{m}}) =  (\mathbf{y} \res l,  \pi \res Y_{\tree}(y \res l))$. 
\end{enumerate}

Given a $(Y,T,*)$-description $\mathbf{B} = (b, \mathbf{y},\pi)$,
define
\begin{displaymath}
  \corner{\mathbf{B}} = (b, \pi \oplus  \mathbf{y}), 
\end{displaymath}
where $\emptyset \oplus \mathbf{y} = \emptyset$ and if $\mathbf{y} = (y, X, \overrightarrow{(e,x,W)})$,  $\lh(y) = k$ then
\begin{displaymath}
  \pi \oplus \mathbf{y} =
  \begin{cases}
    (\pi(e_0,x_0), y(0),  \dots, \pi(e_{k-1},x_{k-1}), y(k-1)) & \text{if } b \text{ is even,}\\
    (y(0), \pi(e_1,x_1), y(1),\dots, \pi(e_{k-1},x_{k-1}), y(k-1)) & \text{if } b \text{ is odd.}
  \end{cases}
\end{displaymath}
Define
\begin{displaymath}
  \mathbf{B} \prec \mathbf{B}'
\end{displaymath}
iff $\corner{\mathbf{B}} <_{BK} \corner{\mathbf{B}'}$, the ordering on coordinates in $\desc(T,Q,*)$ for some $T,Q$ again according to $\prec$.  Define $\prec^{Y,T} = \prec \res \desc(Y,T,*)$. 

Define
\begin{displaymath}
  n_3 \otimes n_2 =
  \begin{cases}
    n_3  & \text{ if } n_2 \text{ is even},\\
    n_2  & \text{ if } n_2 \text{ is odd}. 
  \end{cases}
\end{displaymath}
Suppose $R$ is a level $\leq n_3 \otimes n_2$ tree, $Y$ is a level $\leq n_3$ tree and $T$ is a level $\leq n_2$ tree. Suppose $\rho : \dom(R) \cup \se{(n_3 \otimes n_2, \emptyset)} \to \desc(Y,T,*) $ is a function. $\rho$ \emph{factors} $(R,Y,T)$ iff 
  \begin{enumerate}
  \item $\rho(n_3 \otimes n_2, \emptyset)$ is the constant $(Y,T,*)$-description. 
  \item For any $(d, r ) \in \dom(R)$, $\rho(d, r) \in \desc(Y,T,\comp{d}{R}[r])$. 
  \item For any $r,r' \in \comp{1}{R}$, if $r <_{BK} r'$ then $\rho(1, r) \prec \rho(1,r')$.
  \item For any $d>1$, any $r\concat (a), r \concat (b) \in \dom(\comp{d}{R})$, if $a < _{BK} b$ and $\comp{d}{R}_{\tree}(r\concat (a))  = \comp{d}{R}_{\tree}(r \concat (b))$ then $\rho(d,r\concat (a)) \prec \rho(d,r \concat (b))$. 
  \item For any $d > 1$, any $r \in \dom(\comp{d}{R})\setminus \se{\emptyset} $, $\rho(d,r^{-}) = \rho(d,r) \res (Y,T, \comp{d}{R}_{\tree}(r^{-}))$. 
\end{enumerate}
$\rho$ \emph{factors} $(R,Y,*)$ iff $\rho$ factors $(R,Y,T')$ where $T'$ is some level $\leq n_2'$ tree, $n_2' \leq n_2$. 
If $n_2 $ is even and $Y$ is a level $\leq n_3$ tree, then
\begin{displaymath}
  \id_{Y,*}
\end{displaymath}
factors $(Y, Y, *)$ where $\id_{Y,*}(d, y) =(d,(y, X, \overrightarrow{(e,x,W)}), \id_{*, X})$ for $\comp{d}{Y}[y] = (X, \overrightarrow{(e,x,W)})$.
If $n_3 = n_2$ and $T$ is a level $\leq n_2$ tree, then
\begin{displaymath}
  \id_{*,T}
\end{displaymath}
factors $(T, Q^{(n_2)}_0, T)$, defined as follows: If $n_2 = 2$ then $\id_{*,T}$ has been defined in \cite{sharpIII}. If $n_2>2$ then $\id_{*,T}$ extends $\id_{*, \comp{\leq n_2-2}{T}}$ and for $d \in \se{n_2-1,n_2}$, $\id_{*,T}(d, t) = (n_2, \mathbf{q}^d_t, \tau^d_t)$, where $\mathbf{q}^d_t = ((-1), P^d_t, (n_2-1, ((0))))$, $\tau^d_t$ factors $(P^d_t, T, \comp{d}{T}[t])$, $\tau^d_t (n_2-1, ((0))) = (d, t, \id_{\comp{d}{T}_{\tree}(t), *})$.

Suppose $Y$ is a level $\leq n_3$ tree and $T$ is a level $\leq n_2$ tree. A \emph{representation} of $Y \otimes T$ is a pair $(R,\rho)$ such that 
 $R$ is a level $\leq n_3 \otimes n_2$ tree,
 $\rho $ factors $(R,Y,T)$, and
 $\ran(\rho) = \desc(Y,T,*)$. Representations of $Y \otimes T$ are clearly mutually isomorphic. We shall regard
 \begin{displaymath}
   Y \otimes T
 \end{displaymath}
itself as a level $\leq n_3 \otimes n_2$ tree whose level-$d$ component has domain $\set{(\mathbf{y}, \pi)}{(d, (\mathbf{y}, \pi) ) \in \desc(Y,T,*)}$ and if $d>1$ then $Y \otimes T [d, (\mathbf{y}, \pi)]$ is the unique $(Q,\overrightarrow{(d,q,P)})$ for which $(d, (\mathbf{y}, \pi)) $ is a $(Y,T, (Q,\overrightarrow{(d,q,P)}))$-description. 
Suppose $m_1,n_2,n_3$ are consecutive entries of the following list:
\begin{equation}
  \label{eq:4}
    1, 2, 2, 2, 3, 4, 4, 4, 5, 6, 6, 6, 7, 8, 8, 8, 9 \dots
\end{equation}
Then
\begin{displaymath}
  (n_3 \otimes n_2 ) \otimes m_1 = n_3 \otimes (n_2 \otimes m_1) \DEF n_3 \otimes n_2 \otimes m_1.
\end{displaymath}
If
$Q$ is a level $\leq m_1$ tree, then $(Y \otimes T) \otimes Q$ is regarded as a level $\leq n_3 \otimes n_2 \otimes m_1$ tree. There is a natural isomorphism
\begin{displaymath}
  \iota_{Y,T,Q} 
\end{displaymath}
between ``level $\leq n_3 \otimes n_2 \otimes m_1$ trees'' $(Y \otimes T) \otimes Q$ and $Y \otimes (T \otimes Q)$, defined as follows: If $n_3 \leq 2$, $\iota_{Y,T,Q}$ has been defined in \cite{sharpII}. Suppose now $n_3 > 2$. Let $m_0, m_1, n_2, n_3$ be consecutive entries in the list~(\ref{eq:4}). Then $\iota_{Y,T,Q}$ extends $\iota_{\comp{\leq n_2}{Y}, \comp{\leq m_1}{T}, \comp{\leq m_0}{Q}}$ and 
\begin{enumerate}
\item if $(n_3, \mathbf{y}, \pi) \in \desc(Y,T, (Z,\overrightarrow{(d,z,N)}))$, 
$\mathbf{A} = ( n_3 \otimes n_2,  ((\mathbf{y}, \pi) , Z,  \overrightarrow{(d,z,N)})$, $  \psi) \in \desc( Y \otimes T, Q, U)$, then $\iota_{Y,T,Q} (\mathbf{A})  =  (n_3, \mathbf{y}, \iota^{-1}_{T,Q,U} \circ ( T \otimes \psi) \circ  \pi)$; 
\item if  $(n_3, \mathbf{y}, \pi) \in \desc(Y,T, (Z,\overrightarrow{(d,z,N)}))$, $\mathbf{A} = (n_3 \otimes n_2,  ((\mathbf{y},\pi) \concat (-1), Z^{+}$, $\overrightarrow{(d,z,N)}), \psi) \in \desc(Y \otimes T, Q, U)$, $\overrightarrow{(d,z,N)} = (d_i,z_i,N_i)_{l_0 \leq i \leq l}$,  $\mathbf{y} = (y,X, (e_i,x_i,W_i)_{k_0 \leq i \leq k})$, then
  \begin{enumerate}
  \item if $y$ is of discontinuous type, then $\iota_{Y,T,Q}(\mathbf{A}) = (n_3,\mathbf{y} \concat (-1,X^{+}), \psi *_0 \pi)$, where $X^{+}$ is a completion of $\comp{n_3}{Y}(y)$, $\psi*_0 \pi$ factors $(X^{+}, T \otimes Q, U)$, $\psi *_0 \pi$ extends $ \iota^{-1}_{T,Q,U} \circ ( T \otimes \psi) \circ  \pi$, $\psi *_0 \pi (e_k , x_k) = \iota_{T,Q,U}^{-1} (n_2, \mathbf{t}_0, \tau)$, $\mathbf{t}_0 = ((-1), S_0, (b_0,s_0))$, $\tau$ factors $(S_0, Q, U)$,  $\tau(b_0,s_0) = \psi (d_l,z_l)$;
  \item if $y$ is of continuous type, then $\iota_{Y,T,Q}(\mathbf{A}) = (n_3, \mathbf{y} , \psi *_1 \pi)$, where $\psi*_1 \pi$ factors $(X, T \otimes Q, U)$,  $\psi *_1 \pi$ extends $ \iota^{-1}_{T,Q,U} \circ ( T \otimes \psi) \circ \pi \res (\dom(X) \setminus \se{(e_k,x_k)})$, $\psi *_1 \pi (e_k , x_k) = \iota_{T,Q,U}^{-1} (c, \mathbf{t} \concat (-1, S^{+}), \tau^{+})$,  where $\pi(e_k,x_k) = (c, \mathbf{t}, \tau)$, $\mathbf{t} = (t, S, (b_i,s_i,p_i)_{a_0 \leq i \leq m})$, $\tau^{+}$ factors $(S^{+},Q,U)$,  $\tau^{+}$ extends $\tau$, $\tau^{+}(b_m, s_m) = \psi (d_l,z_l)$.
  \end{enumerate}
\end{enumerate}
Inductively, we can show that $\iota_{Y,T,Q}$ is a level $\leq n_3 \otimes n_2 \otimes m_1$ tree isomorphism between $(Y \otimes T) \otimes Q$ and $Y \otimes (T \otimes Q)$. The base case $(n_3,n_2,m_1) = (2,2,1)$ is in \cite{sharpII}, whose idea is easily modified to the general case. 
$\iota_{Y,T,Q}$ justifies the associativity of the $\otimes$ operator acting on level $( \leq n_3, \leq n_2, \leq m_1)$ trees.

The identity map $\id_{Y \otimes T}$ factors $(Y \otimes T, Y, T)$.  $\rho$ factors $(R,Y,T)$ iff $\rho $ factors $(R , Y \otimes T)$. If $(d,y) \in \dom(Y)$, $\mathbf{y} = (y, X, \overrightarrow{(e,x,W)}) \in \desc(\comp{d}{Y})$, 
\begin{displaymath}
  Y \otimes_{(d,y)} T
\end{displaymath}
is the level $\leq n_3$ subtree of $Y \otimes T$ whose domain is $\dom(Y \otimes Q^0)$ plus all the $(Y,T,*)$-descriptions of the form $(d,\mathbf{y}, \tau)$. 
If $\pi$ factors level $\leq n_2$ trees $(T,Q)$, then
\begin{displaymath}
  Y \otimes \pi
\end{displaymath}
factors $(Y \otimes T, Y \otimes Q)$, where $Y \otimes \pi ( \mathbf{y}, \psi )  =  (\mathbf{y}, ( \pi \otimes U ) \circ \psi)$ for $(\mathbf{y}, \psi) \in \desc(Y, T, U)$. If $\rho$ factors level $\leq n_3$ trees $(R,Y)$, then
\begin{displaymath}
  R \otimes Y
\end{displaymath}
factors $(R \otimes T, Y \otimes T)$, where $\rho \otimes T (d, \mathbf{r}, \psi) = (d, \comp{d}{\rho}(\mathbf{r}), \psi)$.

Suppose $T$ is a proper level $\leq n_2$ subtree of $T'$, both trees are finite, $(Q_i,(d_i,q_i,P_i))_{l_0 \leq  i \leq l'}$ is a partial level $\leq n_1$ tower, $l \leq l'$, $\mathbf{B} \in \desc(Y,T, (Q_l, (d_i,q_i,P_i)_{l_0 \leq i \leq l}))$ and $\mathbf{B}' \in \desc(Y,T', (Q'_l, (d_i,q_i,P_i)_{l_0 \leq i \leq l'})) \setminus \desc(Y,T, (Q_l, {(d_i, q_i, P_i)_{l_0 \leq i \leq l}}))$. Define
\begin{displaymath}
  \mathbf{B} = \mathbf{B}' \res (Y,T)
\end{displaymath}
iff $\mathbf{B}' \prec \mathbf{B}$ and $\bigcup_{l \leq m \leq l'}\set{\mathbf{B}^{*} \in \desc(Y,T,(Q_i, (d_i,q_i,P_i))_{i \leq m})}{\mathbf{B}' \prec \mathbf{B}^{*} \prec \mathbf{B}} = \emptyset$. 
Inductively, we can show that  $\mathbf{B} = \mathbf{B}' \res (Y,T)$ iff both $\mathbf{B}, \mathbf{B}'$ are of degree $n_3$ and letting $\mathbf{B} = (n_3, (y, X, \overrightarrow{(e,x,W)}), \pi)$, $\mathbf{B}' = (n_3, (y', X', \overrightarrow{(e',x',W')}), \pi')$, $\lh(y) = k$, $\overrightarrow{(e,x,W)}= (e_i,x_i,W_i)_i$, $\ucf(X, \overrightarrow{(e,x,W)}) = (e_{*}, \mathbf{x}_{*})$, $e_{*} = 1 \to \mathbf{x}_{*} = x_{*}$, $e_{*}>1 \to \mathbf{x}_{*} =  (x_{*}, \dots)$, then either
\begin{enumerate}
\item $y$ is of continuous type, $\mathbf{y} \res k-1 = \mathbf{y}' \res k-1$, $\pi \res \comp{n_3}{Y}_{\tree}(y^{-}) \subseteq \pi'$, $\pi(b_{k-1},p_{k-1}) =  \pi'(b_{k-1}, p_{k-1}) \res (T,Q_l)$, or
\item $y$ is of discontinuous type, $\mathbf{B} \iniseg \mathbf{B}'$, $\pi( e_{*},x_{*})=\pi(b_k,p_k) \res (T,Q_l)$.
\end{enumerate}

Suppose $Y$ is a proper level $\leq n_3$ subtree of $Y'$, both finite. Suppose $T$ is a level $\leq n_2$ tree.  For $\mathbf{B} \in \desc(Y,T,*)$, $\mathbf{B}' \in \desc(Y',T,*)$, define 
\begin{displaymath}
  \mathbf{B} = \mathbf{B}' \res (Y,T)
\end{displaymath}
iff $\mathbf{B}' \prec \mathbf{B}$ and $\set{\mathbf{B}^{*} \in \desc(Y,T,*)}{\mathbf{B}' \prec \mathbf{B}^{*} \prec \mathbf{B}} = \emptyset$. Putting $\mathbf{B} = (d, \mathbf{y}, \pi)$, $\mathbf{B}' = (d',\mathbf{y}', \pi')$, $\mathbf{y} = (y,X,\overrightarrow{(e,x,W)})$, $\mathbf{y}' = (y', X',\overrightarrow{(e',x',W')})$, $\lh(y)  =k$,  $\overrightarrow{(e,x,W)} = (e_i,x_i,W_i)_{k_0 \leq i \leq k}$, $\overrightarrow{(e',x',W')} = (e'_i,x'_i,W'_i)_{k'_0 \leq i \leq k'}$, $\ucf(X, \overrightarrow{(e,x,W)}) = (e_{*}, \mathbf{x}_{*})$, $e_{*} = 1 \to \mathbf{x}_{*} = x_{*}$, $e_{*}>1 \to \mathbf{x}_{*} =  (x_{*}, \dots)$,  inductively, we can show that $\mathbf{B} =\mathbf{B}' \res (Y,T)$ iff one of the following holds:
\begin{enumerate}
\item $d,d' \leq n_{*}<n_3$, $n_{*}$ is even, $\mathbf{B} = \mathbf{B}'\res  (\comp{\leq n_{*}}{Y}, \comp{\leq n_{*}}{T}) $.
\item $n_3-1\leq d = d' = n_3$, $\mathbf{B} \in \desc(Y,T,Q)$ is of continuous type, $\mathbf{y} \res k-1 = \mathbf{y}' \res k-1$, $\pi \res \comp{d}{Y}_{\tree}(y^{-}) \subseteq \pi'$, $\comp{d}{Y}\se{y'\res k,+, X'_{k}} = \se{y^{-}}$, either $e_{k-1}= 0$ or $\pi'(e_{k-1},x_{k-1})   = \pi(e_{k-1},x_{k-1}) \res (T,Q)$.
\item $n_3-1\leq d = d' = n_3$, $\mathbf{B} \in \desc(Y,T,Q)$ is of discontinuous type, $\mathbf{y} = \mathbf{y}' \res k$, $\pi  \subseteq \pi'$, $\comp{d}{Y}\se{y'\res k+1, +, X'_{k+1}} = \se{y}$,  $\pi'(e_{k},x_{k})=\pi(e_{*},x_{*}) \res (T,Q)$.
\item $d'=n_3-1$, $d=n_3$, $n_3$ is even, $\emptyset = y(0) \res  \comp{d}{Y}\se{\emptyset}$, $(\pi,T \otimes Q)$ is continuous at $(e_0,x_0)$. 
\end{enumerate}

If $n_3$ is odd, $R,Y$ are level-$n_3$ trees, $T$ is a level $\leq n_2$ tree, then $\rho$ factors $(R,Y,T)$ iff $\rho$ extends to some $\rho'$ which factors $(Q^{(n_2)}_0 \oplus R, Q^{(n_2)}_0 \oplus Y, T)$.

\section{The induction hypotheses}
\label{sec:induction-hypothesis}

From now on until the end of this paper, we assume $\boldDelta{2n}$-determinacy, where $n <\omega$.  This section lists the inductive definitions and hypotheses for any $ 0 \leq m <  n$. The base of the induction is in \cite{sharpIII}. Define $E(0) = 1$, $E(k+1) = \omega^{E(i)}$ in ordinal exponentiation. Define $u^{(1)}_i = u_i$ for $i \leq \omega$. 

$\nu_{2m+1}$ denotes the $\mathbb{L}[T_{2m+1}]$-club filter on $\bolddelta{2m+1}$, i.e., $A \in \nu_{2m+1}$ iff $A \in \mathbb{L}[T_{2m+1}]$ and there is a club $C \subseteq \bolddelta{2m+1}$ such that $C \in \mathbb{L}[T_{2m+1}]$ and $C \subseteq A$.

Assume by induction that: 
\begin{enumerate}[label=(\arabic{*}:$m$), ref=\arabic{*}, series=ind]
\item \label{item:subset_does_not_increase_0}
If $m>0$ and $A \subseteq u^{(2m-1)}_{E(2m-1)}$, then $A \in \mathbb{L}_{\bolddelta{2m+1}}[T_{2m}]$ iff $A \in \mathbb{L}[T_{2m+1}]$.
\end{enumerate}

(\ref{item:subset_does_not_increase_0}:$m$) ensures that the $\mathbb{L}_{\bolddelta{2m+1}}[T_{2m}]$-measures induced by level $\leq 2n$ trees are indeed $\mathbb{L}[T_{2m+1}]$-measures. 
 Assume we have defined by induction the level-($2m+1$) sharp operator  $x \mapsto x^{(2m+1)\#}$ for $x \in \mathbb{R}$ with the following property:
 \begin{enumerate}[resume*=ind]
 \item\label{item:odd_sharp} $x^{(2m+1)\#} $ is many-one equivalent to $M_{2m}^{\#}(x)$, the many-one reductions being independent of $x$.
 \end{enumerate}

Define
\begin{displaymath}
  A \in \mu^P
\end{displaymath}
iff $A \subseteq [(\bolddelta{2i+1})_{i \leq m}]^{P \uparrow}$, $A \in \mathbb{L}[T_{2m+1}]$ and there is $\vec{C} = (C_i)_{i \leq m} \in \prod_{i \leq m} \nu_{2i+1}$ such that  $[\vec{C}]^{P \uparrow }\subseteq A$. If $W$ is a finite level-($2m+1$) tree, let $A \in \mu^W$ iff $[(\bolddelta{2i+1})_{i < m}]^{Q^{(2m)}_0\uparrow} \oplus_{2m+1} A  \in \mu^{Q^{(2m)}_0 \oplus W}$.

% From now on, the assumption $m>0$ is dropped. 
We assume by induction that: 
\begin{enumerate}[resume*=ind]
\item\label{item:muP_is_measure} Suppose $P$ is a finite level $\leq 2m+1$ tree. Then
  \begin{enumerate}
  \item $\mu^P$ and $\mu^{\comp{2m+1}{P}}$ are both $\mathbb{L}[T_{2m+1}]$-measures.
  \item If $m>0$, then $\mu^P$ is the product $\mathbb{L}[T_{2m+1}]$-measure of $\mu^{\comp{\leq 2m}{P}}$ and $\mu^{\comp{2m+1}{P}}$, i.e., $A \in \mu^P$ iff there exist $B \in \mu^{\comp{\leq 2m}{P}}$ and $C \in \mu^{\comp{2m+1}{P}}$ such that $B \oplus_{2m+1} C \subseteq A$, where $B \oplus_{2m+1} C = \set{\vec{\alpha} \oplus_{2m+1} \vec{\beta}}{\vec{\alpha} \in B, \vec{\beta} \in C}$, $\vec{\alpha} \oplus_{2m+1}\vec{\beta} = \vec{\gamma}$ where $\comp{d}{\gamma}_p = \comp{d}{\alpha}_p$ for $(d, p ) \in \dom(P)$, $\comp{2m+1}{\gamma}_p = \beta_p$ for $(2m+1, p) \in \dom(P)$.
  \item The set of $\mathbb{L}[j^P(T_{2m+1})]$-cardinals in the interval $[\bolddelta{2m+1}, j^P(\bolddelta{2m+1})]$ is the closure of $\set{\seed^P_{(2m+1, \mathbf{A})}}{\mathbf{A} \in \exexdesc(\comp{2m+1}{P})}$.
  \end{enumerate}
\end{enumerate}
If $P$ is a finite level $\leq 2m+1$ tree or level-($2m+1$) tree,
Let
\begin{displaymath}
  j^P = j^{\mu^P}_{\mathbb{L}[T_{2m+1}]}: \mathbb{L}[T_{2m+1}] \to \mathbb{L}[j^P(T_{2m+1})]
\end{displaymath}
be the induced restricted ultrapower map. For any real $x$, $j^P$ is elementary from $L[T_{2m+1},x]$ to $L[j^P(T_{2m+1}),x]$. If $P$ is a subtree of $P'$, both finite, then $j^{P,P'}$ is the factor map from $\mathbb{L}[j^P(T_{2m+1})]$ to $\mathbb{L}[j^{P'}(T_{2m+1})]$. If $\sigma$ factors finite level $\leq 2m+1$ trees $(P,W)$, then $\sigma^W$ is the induced factor map from $\mathbb{L}[j^P(T_{2m+1})]$ to $\mathbb{L}[j^W(T_{2m+1})]$, i.e., $\sigma^W( [h]_{\mu^P} ) = [h^{*}]_{\mu^W}$, where $h^{*}(\vec{\alpha}) = h(\vec{\alpha}_{\sigma})$. 
By (\ref{item:muP_is_measure}:$m$)(\ref{item:ultrapower_bound}:$m-1$), $j^P \res \bolddelta{2m+1} = j^{\comp{\leq 2m}{P}} \res \bolddelta{2m+1}$, and similarly for $j^{P,P'}, \sigma^W$. Define $j^P_{\sup} (\alpha) = \sup (j^P)''\alpha$, and similarly for $j^{P,P'}_{\sup}$, $\sigma^W_{\sup}$. 
Assume by induction that:
\begin{enumerate}[resume*=ind]
\item \label{item:odd_wellfounded} Suppose $(P_i)_{i < \omega}$ is a level-($2m+1$) tower and $P_{\omega} = \cup_{i<\omega} P_i$. Then $P_{\omega}$ is $\Pi^1_{2n+1}$-wellfounded iff the direct limit of $(j^{P_i,P_{i'}})_{i\leq i < \omega}$ is wellfounded.
\end{enumerate}

Suppose $\mathbf{A} \in \exexdesc(\comp{2m+1}{P})$. Then
\begin{displaymath}
  \seed^P_{(2m+1,\mathbf{A})}
\end{displaymath}
is represented modulo $\mu^P$ by the function $\vec{\alpha} \mapsto \comp{2m+1}{\alpha}_{\mathbf{A}}$. Similarly define $\seed^P_{(d,\mathbf{p})}$ for $(d,\mathbf{p}) \in \exdesc(P)$ and $\seed^P_{(d, p)}$ for $(d,p) \in \dom(P)$.

If $P$ is an infinite $\Pi^1_{2m+1}$-wellfounded level $\leq 2m+1$ tree, $j^P$ is the direct limit map and $\seed^P_{(2m+1, \mathbf{A})}$, $\seed^P_{(2m+1,\mathbf{r})}$, $\seed^P_{(2m+1,{r})}$ are the images under appropriate tails of the direct limit map. 

Suppose $m>0$, $W$ is a finite level $\leq 2m+1$ tree, $Z$ is a finite level $\leq 2m$ tree and $\mathbf{B} \in \desc(W,Z,*)$. Then
\begin{displaymath}
  \seed^{W,Z}_{\mathbf{B}}
\end{displaymath}
is the element represented modulo $\mu^W$ by $\id^{W,Z}_{\mathbf{B}}$. 

Suppose $P$ is another level $\leq 2m+1$ tree and $\sigma$ factors $(P,W,Z)$. Let
\begin{displaymath}
  \seed^{W,Z}_{\sigma} = [\id^{W,Z}_{\sigma}]_{\mu^W}.
\end{displaymath}
So $\seed^{W,Z}_{\sigma} = (\seed^{W,Z}_{\sigma(d,w)})_{(d,w) \in \dom(W)}$.
 If $W,Z$ are (possibly infinite) $\Pi^1_{2m+1}$-wellfounded trees, $\seed^{W,Z}_{\mathbf{B}}$ and $\seed^{W,Z}_{\sigma}$ make sense as the images under appropriate tails of direct limit maps.

We assume by induction that:
\begin{enumerate}[resume*=ind]
\item \label{item:PW_factor}Suppose $P,W$ are finite level $\leq 2m+1$ trees, $Z$ is a finite level $\leq 2m$ tree and $\sigma$ factors $(P,W,Z)$. Then for any $A \in \mu^P$, $\seed^{W,Z}_{\sigma} \in j^W \circ j^Z(A)$.
\end{enumerate}

We can then define
\begin{displaymath}
  \sigma^{W,Z} : \mathbb{L}[j^P(T_{2m+1})] \to \mathbb{L}[j^W\circ j^Z(T_{2m+1})]
\end{displaymath}
by sending $j^P(f)(\seed^P)$ to $j^W \circ j^Z (f)(\seed^{W,Z}_{\sigma})$.

We assume by induction that:
\begin{enumerate}[resume*=ind]
\item \label{item:WZ_tensor_product} Suppose $W$ is a finite level $\leq 2m+1$ tree and $Z$ is a finite level $\leq 2m$ tree. Then $(\id_{W \otimes Z})^{W,Z}$ is the identity on $j^W \circ j^Z(\bolddelta{2m+1})+1$.
\item   \label{item:cardinals}
Suppose $P$ is a finite level $\leq 2m+1$ tree. Then the set of $\mathbb{L}[j^P(T_{2m+1})]$-regular cardinals in the interval $[\bolddelta{2m+1}, j^P(\bolddelta{2m+1})]$ is $\set{\seed^P_{\mathbf{r}}}{\mathbf{r} \in \exdesc(P) \text{ is regular}}$.
\end{enumerate}

By (\ref{item:subset_does_not_increase_0}:$k$) for $k \leq m$ and (\ref{item:subset_does_not_increase}:$k$)(\ref{item:strong_partition}:$k$)(\ref{item:ultrapower_bound}:$k$)(\ref{item:regular_cardinals_Q}:$k$) for $k < m$, if $P$ is a finite level $\leq 2m+1$ tree, then the set of  uncountable $\mathbb{L}[j^P(T_{2m+1})]$-cardinals below $j^P(\bolddelta{2m+1})$ is the closure of
  \begin{displaymath}
    \set{u^{(2k+1)}_{\xi}}{k < m, 0 < \xi \leq E(2k+1)} \cup \set{\seed^P_{(2m+1,\mathbf{A})}}{\mathbf{A} \in \exexdesc(\comp{\leq 2m+1}{P})},
  \end{displaymath}
and the set of uncountable $\mathbb{L}[j^P(T_{2m+1})]$-regular cardinals is
  \begin{displaymath}
    \set{\seed^P_{(d,\mathbf{p})}}{(d,\mathbf{p}) \in \exdesc(P) \text{ is regular}}.
  \end{displaymath}

The level-($2m+1$) uniform indiscernibles are defined:
\begin{enumerate}
\item $u^{(2m+1)}_{\xi+1} = j^P(\bolddelta{2m+1})$ when $\xi < E(2m+1)$, $P$ is a $\Pi^1_{2m+1}$-wellfounded level $\leq 2m+1$ tree and $\llbracket \emptyset \rrbracket_{\comp{2m+1}{R}} = \widehat{\xi}$. 
\item If $0 < \xi  \leq E(2m+1)$ is a limit, then $u^{(2m+1)}_{\xi} = \sup _{\eta<\xi}u^{(2m+1)}_{\eta}$.
\end{enumerate}
% In particular, $\bolddelta{2m+2} = u^{(2m+1)}_2$. 

A level-1 sharp code is a usual sharp code for an ordinal below $u_{\omega}$. 
If $m>0$, a \emph{level-$(2m+1)$ sharp code} is a pair $\corner{ \gcode{\tau}, x^{(2m+1)\#}}$ where $\tau$ is an $\mathcal{L}^{{\underline{x}},R^{(2m+1)}_{E(2m+1)}}$-Skolem term for an ordinal without free variables. For $0 < \xi \leq E(2m+1)$, $\WO^{(2m+1)}_{\xi}$ is the set of level-($2m+1$) sharp codes $\corner{ \gcode{\tau}, x^{(2m+1)\#}}$ such that $\tau$ is an $\mathcal{L}^{\underline{x},R^{(2m+1)} _{\xi}}$-Skolem term. 
The ordinal coded by $\corner{ \gcode{\tau}, x^{(2m+1)\#}}$ is
\begin{displaymath}
  \sharpcode{\corner{\gcode{\tau}, x^{(2m+1)\#}}} = \tau^{(j^{R^{(2m+1)}_{E(2m+1)}}(M^{-}_{2m,\infty}(x)); \seed^{R^{(2m+1)}_{E(2m+1)}})}.
\end{displaymath} 
Assume by induction that:
\begin{enumerate}[resume*=ind]
\item \label{item:sharp_code_complexity}
$\WO^{(2m+1)}_{\xi}$ is $\Pi^1_{2m+2}$ for $0 < \xi \leq E(2m+1)$, uniformly in $\xi$. The following relations are all $\Delta^1_{2m+1}$:
\begin{enumerate}
\item   $v,w \in \WO^{(2m+1)}_{E(2m+1)} \wedge \sharpcode{v} = \sharpcode{w}$.
\item   $v,w \in \WO^{(2m+1)}_{E(2m+1)} \wedge \sharpcode{v} < \sharpcode{w}$.
\item   $k<m \wedge v \in \WO^{(2m+1)}_{E(2m+1)} \wedge w \in \WO^{(2k+1)}_{E(2k+1)} \wedge \sharpcode{v} = \sharpcode{w}$.
\end{enumerate}
\end{enumerate}

If $\Gamma$ is a pointclass, say that $A \subseteq  u^{(2m+1)}_{E(2m+1)} \times  \mathbb{R} $ is in $\Gamma$ iff $\set{(v, x)}{v \in \WO^{(2m+1)}_{E(2m+1)}  \wedge (\sharpcode{v},x) \in A}$ is in $\Gamma$. $\Gamma $ acting on subsets of product spaces is defined in the obvious way.
Assume that we have constructed by induction the level-($2m+2$) Martin-Solovay tree $T_{2m+2}$ on $2 \times u^{(2m+1)}_{E(2m+1)}$. The construction should ensure that $T_{2m+2}$ is $\Delta^1_{2m+3}$ in the codes and projects to $\set{x^{(2m+1)\#}}{x \in \mathbb{R}}$. Let $\kappa_{2m+3}^x$ be the least $(T_{2m+2},x)$-admissible ordinal. $\kappa_{2m+3} = \kappa_{2m+3}^0$. 

Suppose $W$ is a finite level $\leq 2m+1$ tree, $\overrightarrow{(d,w)} = (d_i,w_i)_{i < k}$ is a distinct enumeration of a subset of $\dom(W)$. Suppose $f : [(\bolddelta{2i+1})_{i \leq m}]^{W \uparrow} \to \bolddelta{2m+1}$ is a function which lies in $\mathbb{L}[T_{2m+1}]$. The \emph{signature} of $f$ is $\overrightarrow{(d,w)}$ iff there is $\vec{C}  = (C_i)_{i \leq m}$ such that $C_i \in \nu_{2i+1}$ for any $i$ and 
\begin{enumerate}
\item for any $\vec{\alpha}, \vec{\beta} \in [\vec{C}]^{W \uparrow}$, if $(\comp{d_0}{\alpha}_{w_0},\ldots,\comp{d_{k-1}}{\alpha}_{w_{k-1}}) <_{BK} (\comp{d_0}{\beta}_{w_0}, \ldots,\comp{d_{k-1}}{\beta}_{w_{k-1}})$ then $f(\vec{\alpha}) < f(\vec{\beta})$;
\item for any $\vec{\alpha}, \vec{\beta} \in [\vec{C}]^{W \uparrow}$, if  $(\comp{d_0}{\alpha}_{w_0},\ldots,\comp{d_{k-1}}{\alpha}_{w_{k-1}}) = (\comp{d_0}{\beta}_{w_0}, \ldots,\comp{d_{k-1}}{\beta}_{w_{k-1}})$ then $f(\vec{\alpha}) = f(\vec{\beta})$.
\end{enumerate}

Suppose the signature of $f$ is  $\overrightarrow{(d,w)}= (d_i,w_i)_{i < k}$ and $k>0$, $d_0 = 2m+1$. 
$f$ is \emph{essentially continuous} iff for $\mu^W$-a.e.\ $\vec{\alpha}$, $f(\vec{\alpha}) = \sup \set{f(\vec{\beta})}{ (\comp{d_0}{\beta}_{w_0},\ldots,\comp{d_{k-1}}{\beta}_{w_{k-1}}) < (\comp{d_0}{\alpha}_{w_0},\ldots,\comp{d_{k-1}}{\alpha}_{w_{k-1}}) }$.
Otherwise, $f$ is \emph{essentially discontinuous}. 
Put $[\vec{B}]^{W \uparrow (0,-1)} = [\vec{B}]^{W \uparrow} \times \omega$. For $(d^{*},\mathbf{w}^{*}) \in \exdesc(W)$ regular, put $[\vec{B}]^{W \uparrow (d^{*},\mathbf{w}^{*})} =  \set{(\vec{\beta},\gamma)}{\vec{\beta} \in [\vec{B}]^{W \uparrow} , \gamma < \comp{d^{*}}{\beta}_{\mathbf{w}^{*}}}$.  
Say that the \emph{uniform cofinality} of $f$ is $(d^{*},\mathbf{w}^{*})$ iff 
 there is  $g : [(\bolddelta{2i+1})_{i \leq m}]^{W\uparrow (d^{*},\mathbf{w}^{*})} \to \bolddelta{2m+1}$ such that $g \in \mathbb{L}[T_{2m+1}]$ and for $\mu^W$-a.e.\ $\vec{\alpha}$,  $f(\vec{\alpha}) = \sup \{g(\vec{\alpha}, \beta) :  (\vec{\alpha}, \beta) \in [(\bolddelta{2i+1})_{ i \leq m}]^{W \uparrow v}\}$ and the function $\beta \mapsto g(\vec{\alpha}, \beta)$ is order preserving. 
Let $(P_i,(d_i, p_i,R_i))_{i < k} \concat (P_{k})$ be the partial level $\leq 2m+1$ tower of continuous type and let $\sigma$ factor $(P_{k}, W)$ such that $\sigma(d_i,p_i) = (d_i,w_i)$ for each $i < k$. 
Note that $(d_i,w_i) \prec^W (d_0,w_0)$ for $0 < i < k$, so each $P_i$ is indeed a regular level $\leq 2m+1$ tree.  
$\vec{P}= (P_i)_{i \leq m}$ is called the \emph{level $\leq2m+1$ tower induced by $f$}, and $\sigma$ is called the \emph{factoring map induced by $f$}. Note that $\sigma \res P_i$ factors $(P_i,W)$ for each $i$. The \emph{potential partial level $\leq 2m+1$ tower} induced by $f$ is
\begin{enumerate}
\item $(P_k, (d_i, p_i,R_i)_{i < k})$, if $f$ is essentially continuous;
\item $(P_k, (d_i,p_i,R_i)_{i < k} \concat (0,-1,\emptyset))$, if $f$ is essentially discontinuous and has uniform cofinality $(0,-1)$;
\item $(P_k, (d_i,p_i,R_i)_{i < k} \concat (d^{+}, x^{+}, R^{+}))$, if $f$ is essentially discontinuous and has uniform cofinality $ (d_{*}, \mathbf{w}_{*})$, $d_{*}>0$,  $(P_k, (d^{+},x^{+},R^{+}))$ is a partial level $\leq 2m+1$ tree with uniform cofinality $ (d_{*}, \mathbf{p}_{*})$,  $\comp{d_{*}}{\sigma}(\mathbf{p}_{*}) =  \mathbf{w}_{*}$.
\end{enumerate}
The \emph{approximation sequence} of $f$ is $( f_i)_{i \leq  k}$ where $\dom(f_i) = [(\bolddelta{2i+1})_{i \leq m}]^{P_i \uparrow}$, $f_0$ is the constant function with value $\bolddelta{2m+1}$,  $f_i(\vec{\alpha}) = \sup \set{f(\vec{\beta})}{\vec{\beta} \in [\omega_1]^{W \uparrow}, (\comp{d_0}{\beta}_{w_0},\ldots,\comp{d_{i-1}}{\beta}_{w_{i-1}})= (\comp{d_0}{\alpha}_{p_0},\ldots,\comp{d_{i-1}}{\alpha}_{p_{i-1}})}$ for $1 \leq i  \leq k$. In particular,  $f_k(\vec{\beta}_{\sigma}) = f(\vec{\beta})$ for $\mu^W$-a.e.\ $\vec{\beta}$.

Suppose $\bolddelta{2m+1} \leq \beta = [f]_{\mu^W}< j^W(\bolddelta{2m+1})$. Suppose the signature of $f$ is $(d_i,w_i)_{i < k}$, the approximation sequence of $f$ is $(f_i)_{ i \leq  k}$, the level $\leq 2m+1$ tower induced by $f$ is $(P_i)_{i \leq k}$, the factoring map induced by $f$ is $\sigma$.  
Then the \emph{$W$-signature} of $\beta$ is $({d_i,w_i})_{i < k}$, the \emph{$W$-approximation sequence} of $\beta$ is  $( [f_i]_{\mu^{P_i}})_{i \leq k}$,  $\beta$ is \emph{$W$-essentially continuous} iff $f$ is essentially continuous. The \emph{$W$-uniform cofinality} of $\beta$ is $\omega$ if $f$ has uniform cofinality $(0,-1)$, $\seed^W_{(d_{*},\mathbf{w}_{*})}$ if $f$ has uniform cofinality $(d_{*},\mathbf{w}_{*})$. 
The \emph{$W$-(potential) partial level $\leq 2m+1$ tower and $W$-factoring map induced by $\beta$} are the (potential) partial level $\leq 2m+1$ tower and factoring map induced by $f$ respectively. Assume by induction that:
\begin{enumerate}[resume*=ind]
\item \label{item:signature_Q}The $W$-partial level $\leq 2m+1$ tower induced by $\beta$ and the $W$ approximation sequence of $\beta$ are uniformly $\Delta^1_{2m+3}$ definable from $(W,\beta)$.
\end{enumerate}

Suppose $Q$ is a level-($ 2m+2$) tree. The \emph{ordinal representation} of $Q$ is the set
\begin{align*}
  \rep(Q) =& \set{\vec{\beta} \oplus_Q q}{q \in \dom(Q), \vec{\beta} \text{ respects } Q_{\tree}(q)} \\
 & \cup \set{\vec{\beta} \oplus_Q q \concat(-1)}{q \in \dom(Q), \vec{\beta} \text{ respects }Q(q)}. 
\end{align*}
Here for $q \in \bardom(Q)$ of length $k$, $\vec{\beta} \oplus_Q q = (\beta_{Q_{\node}(q \res 0)}, q(0), \dots, \beta_{Q_{\node}(q)}, q(k-1))$.
$\rep(Q)$ is endowed with the $<_{BK}$ ordering
\begin{displaymath}
  <^Q = <_{BK} \res \rep(Q).
\end{displaymath}
Assume by induction that:
\begin{enumerate}[resume*=ind]
\item \label{item:rep_Q} Suppose $Q$ is a level-($2m+2$) tree. Then $Q$ is $\Pi^1_{2m+2}$-wellfounded iff $<^Q$ is a wellordering.
\end{enumerate}

Suppose $B \in \mathbb{L}[T_{2m+1}]$. Define 
\begin{displaymath}
  f \in {B}^{Q \uparrow}
\end{displaymath}
iff $f \in \mathbb{L}[T_{2m+1}]$ is an order preserving function from $\rep(Q)$ to $ B$. If $f \in [\bolddelta{2m+1}]^{Q \uparrow}$, then for any $q \in \dom(Q)$,  ${f}_q$ is a function on $(\bolddelta{2m+1})^{Q_{\tree}(q) \uparrow}$ that sends $\vec{\alpha}$ to $f(\vec{\alpha} \oplus_Q q)$, and $f$ represents a tuple of ordinals
\begin{displaymath}
  [f]^Q = ([f]^Q_q)_{q \in \dom(P)}
\end{displaymath}
where $[f]^Q_q = [f_q]_{\mu^{Q_{\tree}(q)}}$ for $q \in \dom(Q)$. Let
\begin{displaymath}
  [B]^{Q \uparrow} = \set{ [f]^Q}{ f \in B^{Q \uparrow}}.
\end{displaymath}
A tuple of ordinals $\vec{\beta} = (\beta_q)_{q \in \dom(Q)}$ \emph{respects} $Q$ iff $\vec{\beta} \in [\bolddelta{2m+1}]^{Q \uparrow}$. $\vec{\beta}$ \emph{weakly respects $Q$} iff $\beta_{\emptyset} = \bolddelta{2m+1}$ and for any $p,p' \in \dom(Q)$, if $p$ is a proper initial segment of $p'$, then $j^{Q_{\tree}(p), Q_{\tree}(p')} (\beta_q) > \beta_{q'}$.

Suppose now $Q$ is a finite level $\leq 2m+2$ tree. Then $\rep(Q) = \cup_d \se{d} \times \rep(\comp{d}{Q})$. 
Suppose $\vec{B} = (B_i)_{i \leq m+1} \in \mathbb{L}_{\bolddelta{2m+1}}[T_{2m+2}]$. Define $  f \in \vec{B}^{Q \uparrow}$
iff $\mathbb{L}_{\bolddelta{2m+3}}[T_{2m+2}]$ is an order preserving function from $\rep(Q)$ to $\cup_i B_i$ such that for any $i$, $\ran(\comp{i}{f}) \subseteq B_i$. Define $  [\vec{B}]^{Q \uparrow} = \set{ [f]^Q}{ f \in \vec{B}^{Q \uparrow}}$. $\vec{\beta} = (\comp{d}{\beta}_q)_{q \in \dom(Q)}$ \emph{respects} $Q$ iff $\vec{\beta} \in [(\bolddelta{2i+1})_{i \leq m}]^{Q \uparrow}$. If $f \in ((\bolddelta{2i+1})_{i \leq m})^{Q \uparrow}$ and $\mathbf{q}=(d,(q,P,\dots) )\in \exdesc(Q)$, $d>1$, then $\comp{d}{f}_{\mathbf{q}}$ is the function on $[(\bolddelta{2i+1})_{i \leq m}]^{P \uparrow}$ defined as follows: $\comp{d}{f}_{\mathbf{q}} = \comp{d}{f}_q$ if $(d,\mathbf{q})  \in  \desc( Q)$; $\comp{d}{f}_{\mathbf{q}} (\vec{\alpha}) =  \comp{d}{f}_{q}(\vec{\alpha} \res \comp{d}{Q}_{\tree}(q))$ if $(d, \mathbf{q}) \notin \desc(Q)$.   
If $\vec{\beta} =  (\beta_{(d,q)})_{(d,q) \in \dom(Q)} = (\comp{d}{\beta}_q)_{(d,q) \in \dom(Q)} \in [(\bolddelta{2i+1})_{i \leq m}]^{Q \uparrow}$, we define $\beta_{(d,\mathbf{q})} = \comp{d}{\beta}_{\mathbf{q}}$ for  $(d,\mathbf{q}) \in \exdesc(Q) $: if $d>1$, $\mathbf{q} = (q,P,\dots)$, put 
$\comp{d}{\beta}_{\mathbf{q}} = [\comp{d}{f}_{\mathbf{q}}]_{\mu^P}$ where $\vec{\beta} = [f]^Q$. 
Clearly, $\comp{d}{\beta}_{\mathbf{q}} = \comp{d}{\beta}_{q}$ if $(d,\mathbf{q}) \in \desc(Q)$ of discontinuous type,  $\comp{d}{\beta}_{\mathbf{q}} = j^{\comp{d}{Q}_{\tree}(q),P}(\comp{d}{\beta}_{q})$ if $(d,\mathbf{q}) \notin \desc(Q)$.  The next induction hypothesis computes the remaining case when $\mathbf{q}\in \desc(Q)$ is  of continuous type, justifying that  $\comp{d}{\beta}_{\mathbf{q}}$ does not depend on the choice of $f$. 

\begin{enumerate}[resume*=ind]
\item \label{item:Q_ordinal_cont_type}Suppose $Q$ is a level $\leq 2m+2$ tree. 
Suppose $\vec{\beta} = (\comp{d}{\beta}_q)_{(d,q) \in \dom(Q)} \in [(\bolddelta{2i+1})_{i \leq m}]^{Q \uparrow}$, $(d,\mathbf{q})=(d,(q,P,\dots)) \in \desc({Q})$ is of continuous type, $P^{-} = Q_{\tree}(q^{-})$,  then $\comp{d}{\beta}_{\mathbf{q}} =j^{P^{-},P}_{\sup}(\comp{d}{\beta}_{q^{-}})$.
\end{enumerate}

Suppose by induction that:
\begin{enumerate}[resume*=ind]
\item \label{item:q_respecting}
  Suppose that $Q$ is a finite level $\leq 2m+2$ tree and $\vec{\beta}= (\comp{d}{\beta}_q)_{(d,q)\in \dom(Q)}$ is a tuple of ordinals in $u^{(2m+1)}_{E(2m+1)}$. Then $\vec{\beta}$ respects $Q$ iff all of the following holds:
  \begin{enumerate}
  \item $\comp{\leq 2m+1}{\vec{\beta}}$ respects $\comp{\leq 2m+1}{Q}$, where  $\comp{\leq 2m+1}{\vec{\beta}} = \vec{\beta} \res \dom(\comp{\leq 2m+1}{Q})$. 
  \item For any $q \in \dom(\comp{2m+2}{Q})$, the $\comp{2m+2}{Q}_{\tree}(q)$-potential partial level $\leq 1$ tower induced by $\beta_q$ is $\comp{2m+2}{Q}[q]$, and the  $\comp{2m+2}{Q}_{\tree}(q)$-approximation sequence of $\comp{2m+2}{\beta}_q$ is  $(\comp{2m+2}{\beta}_{q \res l})_{  l \leq \lh(q)}$.
    \item If  $\comp{2m+2}{Q} _{\tree}(q \concat (a)) = \comp{2m+2}{Q} _{\tree}(q \concat (b)) $ and $a<_{BK}b$ then $\comp{2m+2}{\beta}_{q \concat (a)} < \comp{2m+2}{\beta}_{q \concat (b)}$.
  \end{enumerate}
Moreover, if $\vec{C} = (C_i)_{i \leq m}  \in  \prod_{i \leq m}\nu_{2i+1}$ is a sequence of clubs, then $\vec{\beta} \in [\vec{C}]^{Q\uparrow}$ iff $\vec{\beta}$ respects $Q$, $\comp{\leq 2m+1}{\beta} \in [\vec{C}]^{\comp{\leq 2m+1}{Q}\uparrow}$, and letting $C'$ be the set of limit points of $C$, then  $\comp{2m+2}{\beta}_q \in j^{\comp{2m+2}{Q}_{\tree}(q)} (C')$ for each $q \in \dom(\comp{2m+2}{Q})$.
\end{enumerate}

We assume by induction the level-($2m+2$) Becker-Kechris-Martin theorem:
\begin{enumerate}[resume*=ind]
\item \label{item:bk_km} For each $A \subseteq u^{(2m+1)}_{E(2m+1)} \times \mathbb{R}$, the following are equivalent.
  \begin{enumerate}
  \item $A$ is $\Pi^1_{2m+3}$.
  \item There is a $\Sigma_1$ formula $\varphi$ such that $(\alpha,x) \in A$ iff $L_{\kappa_{2m+3}^x}[T_{2m+2},x]\models \varphi(T_{2m+2},\alpha,x)$.
  \end{enumerate}
\end{enumerate}

 If $\beta < u^{(2m+1)}_{E(2m+1)} $,  $A \subseteq \mathbb{R}$ is $\beta\text{-}\Pi^1_{2m+3}(x)$ iff there is a $\Pi^1_{2m+3}(x)$ set $B \subseteq u^{(2m+1)}_{E(2m+1)} \times \mathbb{R}$ such that $A = \Diff B$. $\beta$-$\Pi^1_{2m+1}(x)$ acting on product spaces of $\omega$ and $\mathbb{R}$ is defined in the obvious way. Lightface $\beta$-$\Pi^1_{2m+1}$ and boldface $\beta$-$\boldpi{2m+1}$ have the obvious meanings. 

Define
\begin{align*}
      \mathcal{O}^{T_{2m+2},x} & = \{ (\gcode{\varphi},\alpha) : \varphi \text{ is a $\Sigma_1$-formula}, \alpha<u^{(2m+1)}_{E(2m+1)},\\
& \qquad L_{\kappa^x_{2m+3}}[T_{2m+2},x] \models \varphi(T_{2m+2},x,\alpha)\},\\
  x^{{(2m+2)}\#}_{\xi} &= \{(\gcode{\varphi},\gcode{\psi}) :  \exists \alpha<u^{(2m+1)}_{\xi} ( (\gcode{\varphi},\alpha) \notin \mathcal{O}^{T_{2m+2},x}\wedge\\
& \qquad\forall \eta<\alpha (\gcode{\psi},\eta) \in \mathcal{O}^{T_{2m+2},x} )\},\\
  x^{{(2m+2)}\#} &= \{(k,\gcode{\varphi},\gcode{\psi}) :  k<\omega \wedge (\gcode{\varphi},\gcode{\psi}) \in x^{{(2m+2)}\#}_{F(2m+1,k)}\},
\end{align*}
where $F(1,k) = k$, $F(l+1, k) = \omega^{F(l,k)}$ in ordinal arithmetic.

Assume by induction that:
\begin{enumerate}[resume*=ind]
\item \label{item:evensharp} $x^{(2m+2)\#}$ is many-one equivalent to $M_{2m+1}^{\#}(x)$, the many-one reductions being independent of $x$.
\item \label{item:subset_does_not_increase} If $A \subseteq \bolddelta{2m+1}$, then $A \in \mathbb{L}[T_{2m+1}]$ iff $A \in \mathbb{L}_{\bolddelta{2m+3}}[T_{2m+2}]$.
\end{enumerate}
(\ref{item:subset_does_not_increase}:$m$) ensures that the $\mathbb{L}[T_{2m+1}]$-measures on $\bolddelta{2m+1}$ induced by level $\leq 2m+1$ trees are indeed $\mathbb{L}_{\bolddelta{2m+3}}[T_{2m+2}]$-measures. 

(\ref{item:bk_km}:$m$) enables the generalization of Silver's dichotomy on $\Pi^1_{2m+3}(x)$ equivalence relations: If $E$ is a thin $\Delta^1_{2m+3}(x)$ equivalence relation on $\mathbb{R}$, then $E$ is $\Delta^1_{2m+3}(x)$-reducible to $=_{u^{(2m+1)}_{E(2m+1)}}$, where $\alpha =_{u^{(2m+1)}_{E(2m+1)}} \beta$ iff $\alpha=\beta <u^{(2m+1)}_{E(2m+1)}$. As a corollary, if $\leq^{*}$ is a $\Delta^1_{2m+3}(x)$ prewellordering on $\mathbb{R}$ and $A$ is a $\Sigma^1_{2m+3}(x)$ subset of $\mathbb{R}$, then $\sharpcode{\leq^{*}}$ and $\set{\wocode{y}_{\leq^{*}}}{y \in A}$ are both $\Delta_1$-definable over $L_{\kappa_{2m+3}^x}[T_{2m+2}, M_{2m+1}^{\#}(x)]$ from parameters in $\se{T_{2m+2}, M_{2m+1}^{\#}(x)}$. This proves an effective version of the Harrington-Kechris theorem (cf.\  \cite[8G.21]{mos_dst}):
\begin{quote}
  If $\leq$ and $\leq'$ are two $\Delta^1_{2m+3}$ prewellorderings of $\mathbb{R}$, then the relation $\wocode{x}_{\leq} = \wocode{y}_{\leq'}$ is $\Delta^1_{2m+4}$ and is absolute in $M$ whenever $M$ is a $\Sigma^1_{2m+3}$-correct transitive model of ZFC and $\mathbb{R}^M$ is closed under the $M_{2m+1}^{\#}$-operator.
\end{quote}
Consequently, if $\Gamma$ is a pointclass containing $\Delta^1_{2m+4}$ and is closed under recursive preimages, then $\Gamma$ acting on spaces of the form  $(\bolddelta{2m+3})^k \times \mathbb{R}^l$ is independent of the choice of the $\Pi^1_{2m+3}$-coding of ordinals in $\bolddelta{2m+3}$. That is,  if $\varphi$ is a regular $\Pi^1_{2m+3}$-norm on a good universal $\Pi^1_{2m+3}$ set, then for any $A \subseteq \bolddelta{2m+3} \times \mathbb{R}$, $\set{(v,x)}{(\varphi(v),x) \in A}$ is in $\Gamma$ iff $\set{(v,x)}{(\sharpcode{v},x)}$ is in $\Gamma$.

% Let $X ^{\uparrow(2m+1)}$ be the set of strictly increasing functions $f : \bolddelta{2m+1} \to X$ that belong to $\mathbb{L}[T_{2m+1}]$ and satisfying $f(\bolddelta{2i+1}) = \bolddelta{2i+1}$ for any $1 \leq i < m$.
$\bolddelta{2m+1}$ is said to have \emph{the level-($2m+2$) strong partition property} iff for every finite level $\leq 2m+2$ tree $Q$, for every $A \in \mathbb{L}_{\bolddelta{2m+3}}[T_{2m+2}]$, there is are clubs $C_i \in \nu_{2i+1}$ for $i \leq m$ such that either   $[(C_i)_{i \leq m}]^{Q \uparrow} \subseteq A$ or $[(C_i)_{i \leq m}]^{Q \uparrow} \cap A = \emptyset$.
 % function $\psi : \bolddelta{2m+1}^{\uparrow (2m+1)} \to u^{(2m+1)}_{E(2m+1)}$ such that the relation ``$\corner{ \tau, x^{(2m+1)\#}}$ is a level-$(2m+1)$ sharp code for an increasing function\footnote{to be defined in Section~\ref{sec:level-2n+2-sharp}}, $\alpha = \psi (\tau^{M_{2m,\infty}^{-}(x)}(x,\cdot))$'' is $\boldDelta{2m+3}$, for every $B \in \mathbb{L}_{\bolddelta{2m+3}}[T_{2m+2}]$, there is $X \subseteq \bolddelta{2m+1}$ such that $\ot(X) = \bolddelta{2m+1}$, $X \in \mathbb{L}[T_{2m+1}]$ and either $\phi'' X^{\uparrow} \subseteq B$ or $\psi'' X^{\uparrow} \subseteq u^{(2m+1)}_{E(2m+1)}\setminus B$. 
 We assume by induction that: 
\begin{enumerate}[resume*=ind]
\item \label{item:strong_partition} $\bolddelta{2m+1}$ has the level-($2m+2$) strong partition property.
\item \label{item:Q_direct_limit_wf} Suppose $(Q_k)_{1 \leq k < \omega}$ is an infinite level-($2m+2$) tower and $Q_{\omega} = \cup_{k<\omega} Q_k$. Then $Q_{\omega}$ is $\Pi^1_{2m+2}$-wellfounded iff the direct limit of $(j^{Q_k,Q_{k'}})_{k \leq k' < \omega}$ is wellfounded.
\end{enumerate}

If $Q$ is a finite level $\leq 2m+2$ tree, define
\begin{displaymath}
  A \in \mu^Q
\end{displaymath}
iff $A \subseteq [(\bolddelta{2i+1})_{i \leq m}]^{Q \uparrow}$, $A \in \mathbb{L}_{\bolddelta{2m+3}}[T_{2m+2}]$ and there is $\vec{C} = (C_i)_{i \leq m}$ such that $\vec{C} \in \prod_{i \leq m}\nu_{2i+1}$ and $[\vec{C}]^{Q \uparrow}\subseteq A$. By~(\ref{item:strong_partition}:$m$), $\mu^Q$  is an $\mathbb{L}_{\bolddelta{2m+3}}[T_{2m+2}]$-measure and $\mu^Q$ is product measure of its level $\leq 2m$ component, its level-($2m+1$) component and its level-($2m+2$) component. Let
\begin{displaymath}
  j^Q : \mathbb{L}_{\bolddelta{2m+3}}[T_{2m+2}] \to \mathbb{L}_{\sup (j^Q)''\bolddelta{2m+3}}[j^Q(T_{2m+2})]
\end{displaymath}
be the associated $\mathbb{L}_{\bolddelta{2m+3}}[T_{2m+2}]$-ultrapower map. 
Assume by induction that:
\begin{enumerate}[resume*=ind]
\item \label{item:ultrapower_bound} Suppose $Q$ is a finite level $\leq 2m+2$ tree. Then $j^Q(\alpha)< \bolddelta{2m+3}$ for any $\alpha < \bolddelta{2m+3}$. $j^Q(T_{2m+2}) \in \mathbb{L}_{\bolddelta{2m+3}}[T_{2m+2}]$.
\end{enumerate}
So $\sup (j^Q)'' \bolddelta{2m+3} = \bolddelta{2m+3}$. 
If $Q$ is a subtree of $Q$, $j^{Q,Q'}$ is the induced factor map. If $\pi$ factors $(Q,T)$, $\pi^T$ is the induced factor map. By~(\ref{item:subset_does_not_increase}:$m$), $j^Q \res \bolddelta{2m+1} = j^{\comp{\leq 2m}{Q}} \res \bolddelta{2m+1}$, and similarly for $j^{Q,Q'}$ and $\pi^T$. 
Define $j^Q_{\sup}$, $j^{Q,Q'}_{\sup}$, $\pi^T_{\sup}$ as usual. 
As advertised in the end of \cite{sharpIII}, we need to prove that every  $\mathbb{L}_{\bolddelta{2m+3}}[T_{2m+2}]$ is equivalent to $\mu^Q$ for some finite level $\leq 2m+2$ tree $Q$ and from this, establish a $\Delta^1_{2m+3}$ coding of subsets of  $u^{(2m+1)}_{E(2m+1)}$ in $\mathbb{L}_{\bolddelta{2m+3}}[T_{2m+2}]$.
We assume by induction that 
\begin{enumerate}[resume*=ind]
\item \label{item:measure_analysis} Suppose $P$ is a finite level-$(2m+1)$ tree and $\mu$ is a nonprincipal $\mathbb{L}_{\bolddelta{2m+3}}[T_{2m+2}]$-measure on $j^P(\bolddelta{2m+1})$. Then there are functions $g,h \in \mathbb{L}[j^P(T_{2m+1})]$, a finite level $\leq 2m+2$ tree $Q$, nodes $(d_1,q_1),\dots,(d_k,q_k) \in \dom(Q)$, such that $h : j^P(\bolddelta{2m+1}) \to ( j^P(\bolddelta{2m+1}) )^k$, $g:(j^P(\bolddelta{2m+1}))^k \to j^P(\bolddelta{2m+1})$, $h$ is 1-1 a.e.\ ($\nu$), $g$ is 1-1 a.e.\ ($\mu^Q_{\overrightarrow{(d,q)}}$), $g = h^{-1}$ a.e.\ ($\nu$), and $A \in \mu^Q_{\overrightarrow{(d,q)}}$ iff $(h^{-1})''A \in \nu$.
\item \label{item:coding} There is a $\Delta^1_{2m+3}$ set $X\subseteq \mathbb{R} \times  u^{(2m+1)}_{E(2m+1)} $ such that every subset of $u^{(2m+1)}_{E(2m+1)}$ in  $\mathbb{L}_{\bolddelta{2m+3}}[T_{2m+2}]$ is equal to some $X_v \DEF \set{a}{(v,a) \in X}$.
\end{enumerate}
Renaming the second coordinate in (\ref{item:coding}:m), we fix a $\Delta^1_{2m+3}$ set
\begin{displaymath}
X^{(2m+3)} \subseteq \mathbb{R} \times ( V_{\omega} \cup u^{(2m+1)}_{E(2m+1)})^{<\omega}
\end{displaymath}
such that every subset of $(V_{\omega} \cup u^{(2m+1)}_{E(2m+1)})^{<\omega}$ in $\mathbb{L}_{\bolddelta{2m+3}}[T_{2m+2}]$ is equal to some $X^{(2m+3)}_v \DEF \set{a}{(v,a) \in X^{(2m+3)}}$. Let $v \in \LO^{(2m+3)}$ iff $X^{(2m+3)}_v$ is a linear ordering of $u^{(2m+1)}_{E(2m+1)}$, $v \in \WO^{(2m+3)}$ iff $X_v$ is a wellordering of $u^{(2m+2)}_{E(2m+1)}$. $\LO^{(2m+3)}$ is $\Delta^1_{2m+3}$, $\WO^{(2m+3)}$ is $\Pi^1_{2m+3}$. For $v \in \WO^{(2m+3)}$, put $\sharpcode{v} = \ot(X_v)$. Pointclasses are allowed to act on spaces of the form $(\bolddelta{2m+3})^k \times \mathbb{R}^l$ via this coding.

Suppose  $Q$ is a level $\leq 2m+2$ tree and $W$ is a level $\leq 2m+1$ tree. If $m=0$, the notations related to $(Q,W,*)$-descriptions have been defined in  \cite{sharpII}. Suppose $m>0$.  Suppose $\mathbf{D} = (d, \mathbf{q}, \sigma) \in \desc(Q, W, *)$. For $g \in ((\bolddelta{2i+1})_{i \leq m})^{Q \uparrow}$, let
\begin{displaymath}
  g^W_{\mathbf{D}}
\end{displaymath}
be the function on $[(\bolddelta{2i+1})_{i \leq m}]^{W \uparrow}$ as follows: 
\begin{enumerate}
% \item If $d=1$, then $g^W_{\mathbf{D}} (\vec{\alpha}) = {}^1 [g]^Q_{\mathbf{q}}$.
\item If $d \leq 2m+1$, then $g^W_{\mathbf{D}}(\vec{\alpha}) = (\comp{\leq 2m+1}{g})^{\comp{\leq 2m}{W}}_{\mathbf{D}} (\vec{\alpha})$.
\item If $d=2m+2$ and $\mathbf{D} \in \desc(Q,W,Z)$,  then $g^W_{\mathbf{D}} (\vec{\alpha}) =  j^Z (\comp{2m+2}{g}_{\mathbf{q}} )  \circ \id_{\sigma}^{W,Z} (\vec{\alpha})$, or equivalently, $g^W_{\mathbf{D}} ([f]^W) = [\comp{2m+2}{g}_{\mathbf{q}} \circ f^Z_{\sigma}]_{\mu^Z}$. Here $\id_{\sigma}^{W,Z}$ and $f^Z_{\sigma}$ have already been defined by induction and $\id_{\sigma}^{W,Z}([f]^W) = [f^Z_{\sigma}]_{\mu^Z}$.
\end{enumerate}
In particular, if $\mathbf{D}$ is the constant $(Q,*)$-description, then $g_{\mathbf{D}}^W$ is the constant function with value $\bolddelta{2m+1}$. 
Suppose additionally that $Q$ is finite. Let
\begin{displaymath}
  \id_{\mathbf{D}}^{Q,W}
\end{displaymath}
be the function  $[g]^Q\mapsto [g_{\mathbf{D}}^W]_{\mu^W}$, or equivalently,  $\vec{\beta} \mapsto \id_{\mathbf{D}}^{\comp{\leq 2m+1}{Q}, \comp{\leq 2m}{W}} ( \comp{\leq 2m+1}{\vec{\beta}} )$ if $d\leq 2m+1$, 
$\vec{\beta}\mapsto \sigma^{W , Z} (  \comp{d}{\beta}_{\mathbf{q}})$ otherwise.
\begin{displaymath}
  \seed^{Q,W}_{\mathbf{D}}
\end{displaymath}
is the element represented modulo $\mu^Q$ by $\id^{Q,W}_{\mathbf{D}}$.  In particular, if $d=1$ then $\seed^{Q,W}_{\mathbf{D}} = \seed^Q_{(1,\mathbf{q})}$; if $d>1$, $\mathbf{q} = (q, W, \id_W)$, then $\seed^{Q,W}_{\mathbf{D}} = \seed^Q_{(d, \mathbf{q})}$.
 We assume by induction that:
 \begin{enumerate}[resume*=ind]
 \item\label{item:QW_desc_order} Suppose $Q$ is a finite level $\leq 2m+2$ tree, $W$ is a finite level $\leq 2m+1$ tree, $Z$ is a finite level $\leq 2m$ tree, and suppose $\mathbf{D} , \mathbf{D}' \in \desc(Q,W,Z)$, $\mathbf{D} \prec^{Q,W} \mathbf{D}'$,  $\mathbf{D} =  (d,\mathbf{q},\sigma) $, $\mathbf{D}' = (d',\mathbf{q}', \sigma') $. Then for any  $g \in ((\bolddelta{2i+1})_{i \leq m})^{Q \uparrow}$, for any $f \in ((\bolddelta{2i+1})_{i \leq m})^{W \uparrow}$, for any $\vec{\eta} \in [((\bolddelta{2i+1})_{i < m})]^{Z \uparrow}$,
   \begin{displaymath}
\comp{d}{g}_{\mathbf{q}} \circ f^Z_{\sigma} (\vec{\eta}) <  \comp{d'}{g}_{\mathbf{q}'} \circ f^Z_{\sigma'} (\vec{\eta}) .
\end{displaymath}
 \item\label{item:QW_desc_target_extension} Suppose $Q$ is a finite level $\leq 2m+2$ tree, $W, W'$ are finite level $\leq 2m+1$ trees, $W$ is a proper subtree of $W'$, $\mathbf{D} = (d,\mathbf{q}, \sigma) \in \desc(Q,W,Z)$, $\mathbf{D}' = (d',\mathbf{q}',\sigma') \in \desc(Q',W', Z')$, $\mathbf{D} = \mathbf{D}' \res (Q,W)$. Suppose $\vec{E} = (E_i)_{i \leq m} \in \prod_{i \leq m} \nu_{2i+1}$, each $E_i$ is a club, $\eta \in E_i'$ iff $E_i \cap \eta$ has order type $\eta$, $\vec{E}' = (E_i)_{i \leq m}$. Then for any $g \in ((\bolddelta{2i+1})_{i \leq m})^{Q \uparrow}$, for any $\vec{\alpha} \in [\vec{E}']^{W \uparrow}$,
   \begin{displaymath}
     j^{Z,Z'} \circ g^W_{\mathbf{D}} (\vec{\alpha}) = \sup \set{ g^{W'}_{\mathbf{D}'} (\vec{\beta}) }{\vec{\beta} \in [E]^{W' \uparrow}, \vec{\beta} \text{ extends $\vec{\alpha}$}}.
   \end{displaymath}
 \item\label{item:QW_desc_origin_extension}  Suppose $Q, Q'$ are finite level $\leq 2m+2$ trees, $Q$ is a proper subtree of $Q'$, $W$ is a finite level $\leq 2m+1$ tree and $\mathbf{D} \in \desc(Q,W,*)$, $\mathbf{D}' \in \desc(Q',W, *)$, $\mathbf{D}= \mathbf{D}' \res (Q,W)$. Suppose $\vec{E} = (E_i)_{i \leq m} \in \prod_{i \leq m} \nu_{2i+1}$, each $E_i$ is a club, $\eta \in E_i'$ iff $E_i \cap \eta$ has order type $\eta$, $\vec{E}' = (E_i)_{i \leq m}$. Then for any $\vec{\beta} \in [\vec{E}']^{Q \uparrow}$,
   \begin{displaymath}
     \id^{Q,W}_{\mathbf{D}}(\vec{\beta})  =  \sup \set{ \id^{Q',W}_{\mathbf{D}'} (\vec{\gamma})}{ \vec{\gamma} \in [\vec{E}]^{Q' \uparrow}, \vec{\gamma} \text{ extends } \vec{\beta}}. 
   \end{displaymath}
\end{enumerate}

Suppose $S$ is a level $\leq 2m+1$ tree and $\tau$ factors $(S, Q, *)$. 
For  $g \in (\bolddelta{2m+1})^{Q \uparrow}$, let
\begin{displaymath}
  g_{\tau}^W 
\end{displaymath}
be the function sending $\vec{\alpha}$ to $ (g_{\tau(d,s)}^W(\vec{\alpha}))_{(d,s) \in \dom(S) } $.    
If $W$ is finite, let
\begin{displaymath}
  \id^{Q,W}_{\tau} 
\end{displaymath}
is the map sending $[g]^Q$ to $[g^W_{\tau}]_{\mu^W}$. So $\id^{Q,W}_{\tau} (\vec{\beta}) = (\id^{Q,W}_{\tau (d,s )}(\vec{\beta}))_{(d,s) \in \dom(S)}$. If $Q,W$ are both finite, put 
\begin{displaymath}
  \seed_{\tau}^{Q,W} =  [\id^{Q,W}_{\tau}]_{\mu^Q}.
\end{displaymath}
% We assume 
%
% For  $g \in \omega_1^{Q \uparrow}$, let
% % \begin{displaymath}
% %   g_{\tau}^W : [\omega_1]^{W \uparrow} \to [\omega_1]^{S \uparrow}
% % \end{displaymath}
% be the function sending $\vec{\alpha}$ to $ (g_{\tau(s)}^W(\vec{\alpha}))_{s \in \dom(S) } $.  Lemma~\ref{lem:level_2_desc_order} ensures that $g^W_{\tau}$ is indeed a function into $[\omega_1]^{S \uparrow}$. 
We assume by induction that:
\begin{enumerate}[resume*=ind]
\item \label{item:SQ_factor}Suppose $Q$ is a finite level $\leq 2m+2$ tree, $S,W$ are finite level $\leq 2m+1$ trees, $\tau$ factors $(S,Q,W)$. Then for any $A \in \mu^S$, $\seed^{Q,W}_{\tau} \in j^Q \circ j^W(A)$. 
\end{enumerate}

We define
\begin{displaymath}
  \tau^{Q,W} : \mathbb{L}_{\bolddelta{2m+3}}[j^S(T_{2m+2})] \to \mathbb{L}_{\bolddelta{2m+3}}[j^Q \circ j^W(T_{2m+2})]
\end{displaymath}
by sending $j^S(h)(\seed^S) $ to $ j^Q\circ j^W (h)  ( \seed_{\tau}^{Q,W} )$. % For any $z \in \mathbb{R}$, $\tau^{Q,W}$ is elementary from $L_{\kappa_{2m+3}^z}[j^S(T_{2m+2}),z]$ to $L_{\kappa_{2m+3}^z}$
Assume by induction that:
\begin{enumerate}[resume*=ind]
\item \label{item:tensor_product_QW}
Suppose $Q$ is a finite level $\leq 2m+2$ tree and $W$ is a finite level $\leq 2m+1$ tree. Then
\begin{enumerate}
\item $(\id_{Q \otimes W})^{Q,W}$ is the identity on $j^Q \circ j^W(\bolddelta{2m+1}+1)$. 
\item Suppose $W'$ is another finite level $\leq 2m+1$ tree and $\sigma$ factors $(W,W')$. Then $j^Q ( \sigma^{W'} \res j^W(\bolddelta{2m+1}+1)) = (Q \otimes \sigma)^{Q \otimes W'} \res (j^{Q \otimes W} (\bolddelta{2m+1}+1))$.
\item Suppose $Q'$ is another finite level $\leq 2m+2$ tree and $\pi$ factors $(Q,Q')$. Then $\pi^{Q'} \res (j^Q \circ j^W (\bolddelta{2m+1}+1)) = (\pi \otimes W)^{Q' \otimes W} \res j^{Q \otimes W}(\bolddelta{2m+1}+1)$.
\end{enumerate}
\item \label{item:local_definability_jQ} Suppose $x \in \mathbb{R}$ and $Q$ is a finite level $\leq 2m+2$ tree. Then $j^Q(\kappa_{2m+3}^x) = \kappa_{2m+3}^x$ and $j^Q \res \kappa_{2m+3}^x$ is $\Delta_1$-definable over $L_{\kappa_{2m+3}^x}[T_{2m+2},x]$ from $\se{T_{2m+2},x}$, uniformly in $(Q,x)$.
\item \label{item:regular_cardinals_Q} Suppose $Q$ is a finite level $\leq 2m+2$ tree. 
 Then the set of uncountable $\mathbb{L}_{\bolddelta{2m+3}}[j^Q(T_{2m+2})]$-cardinals in the interval $[\bolddelta{2m+1}, \bolddelta{2m+3})$ is
  \begin{displaymath}
    \set{u^{(2m+1)}_{\xi}}{0 < \xi \leq E(2m+1)},
  \end{displaymath}
and the set of uncountable $\mathbb{L}_{\bolddelta{2m+3}}[j^Q(T_{2m+2})]$-regular cardinals in the interval $[\bolddelta{2m+1}, \bolddelta{2m+3})$ is
  \begin{displaymath}
    \set{\seed^Q_{(d,\mathbf{q})}}{2m+1 \leq d \leq 2m+2, (d,\mathbf{q}) \in \exdesc(Q) \text{ is regular}}.
  \end{displaymath}
In particular, the set of uncountable $\mathbb{L}_{\bolddelta{2m+3}}[T_{2m+2}]$-regular cardinals  in the interval $[\bolddelta{2m+1}, \bolddelta{2m+3})$ is
\begin{displaymath}
          \{j^P(\bolddelta{2m+1}):  P \text{ level-($2k+1$) tree}, \card(P) \leq 1\}.
\end{displaymath}
\end{enumerate}

By (\ref{item:subset_does_not_increase_0}:$k$)(\ref{item:subset_does_not_increase}:$k$)(\ref{item:strong_partition}:$k$)(\ref{item:ultrapower_bound}:$k$)(\ref{item:regular_cardinals_Q}:$k$) for $k \leq m$, if $Q$ is a finite level $\leq 2m+2$ tree, then the set of  uncountable $\mathbb{L}_{\bolddelta{2m+3}}[j^Q(T_{2m+2})]$-cardinals is
  \begin{displaymath}
    \set{u^{(2k+1)}_{\xi}}{k \leq m, 0 < \xi \leq E(2k+1)},
  \end{displaymath}
and the set of uncountable $\mathbb{L}_{\bolddelta{2m+3}}[j^Q(T_{2m+2})]$-regular cardinals is
  \begin{displaymath}
    \set{\seed^Q_{(d,\mathbf{q})}}{(d,\mathbf{q}) \in \exdesc(Q) \text{ is regular}}.
  \end{displaymath}

  Suppose $Q$ is a finite level $\leq 2m+2$ tree, $\overrightarrow{(d,q)} = (d_i, q_i)_{1 \leq i < k}$ is a distinct enumeration of a subset of $Q$ and such that for each $d>1$, $\set{q_i}{d_i = d}\cup \se{\emptyset}$ forms a tree on $\omega^{<\omega}$. Suppose $F : [(\bolddelta{2i+1})_{i \leq m}]^{Q \uparrow} \to \bolddelta{2m+3}$ is a function which lies is $\mathbb{L}_{\bolddelta{3}}[T_{2m+2}]$. The \emph{signature} of $F$ is $\overrightarrow{(d,q)}$ iff 
there is $\vec{E} \in \prod_{i \leq m } \nu_{2i+1}$ such that 
\begin{enumerate}
\item for any $\vec{\beta}, \vec{\gamma} \in [E]^{Q \uparrow}$, if $(\comp{d_0}{\gamma}_{q_0}, \ldots, \comp{d_{k-1}}{\gamma}_{q_{k-1}}) <_{BK} (\comp{d_0}{\beta}_{q_0}, \ldots, \comp{d_{k-1}}{\beta}_{q_{k-1}})$ then $f(\vec{\beta}) < f(\vec{\gamma})$;
\item for any $\vec{\beta}, \vec{\gamma} \in [E]^{Q \uparrow}$, if $(\comp{d_0}{\gamma}_{q_0}, \ldots, \comp{d_{k-1}}{\gamma}_{q_{k-1}}) = (\comp{d_0}{\beta}_{q_0}, \ldots, \comp{d_{k-1}}{\beta}_{q_{k-1}})$ then $f(\vec{\beta}) = f(\vec{\gamma})$.
\end{enumerate}
Clearly the signature of $F$ exists and is unique. In particular, $F$ is constant on a $\mu^Q$-measure one set iff the signature of $F$ is $\emptyset$. 

Suppose the signature of $F$ is  $\overrightarrow{(d,q)} = ((d_i, q_i))_{1 \leq i < k}$. 
$F$ is \emph{essentially continuous} iff for $\mu^Q$-a.e.\ $\vec{\beta}$, $F(\vec{\beta}) = \sup\set{F(\vec{\gamma})}{\vec{\gamma} \in [(\delta^1_{2i+1})_{i \leq m}]^{Q \uparrow}, (\comp{d_0}{\gamma}_{q_0}, \ldots, \comp{d_{k-1}}{\gamma}_{q_{k-1}}) <_{BK} (\comp{d_0}{\beta}_{q_0}, \ldots, \comp{d_{k-1}}{\beta}_{q_{k-1}})}$. Otherwise, $F$ is \emph{essentially discontinuous}. 
Put $[\vec{B}]^{Q \uparrow (0,-1)} =  [\vec{B}]^{Q \uparrow} \times \omega $. For 
 $(d,\mathbf{q}) \in \exdesc(Q)$ regular, put $[\vec{B}]^{Q \uparrow(d, \mathbf{q})}=
 \set{(\vec{\beta},\gamma)}{\vec{\beta} \in [\vec{B}]^{Q \uparrow} , \gamma < \comp{d}{\beta}_{\mathbf{q}}}$.
For $(d,\mathbf{q}) $ either $(0,-1)$ or in $ \exdesc(Q)$ regular, 
say that the \emph{uniform cofinality} of $F$ is $\ucf(F)  =  (d,\mathbf{q})$ iff 
 there is  $G : [(\delta^1_{2i+1})_{i \leq m}]^{Q\uparrow (d,\mathbf{q}) } \to \bolddelta{2m+3}$ such that $G \in \mathbb{L}_{\bolddelta{2m+3}}[T_{2m+2}]$ and for any for $\mu^Q$-a.e.\ $\vec{\beta}$,   $F(\vec{\beta}) = \sup \{G(\vec{\beta}, \gamma) :  (\vec{\beta}, \gamma) \in [(\delta^1_{2i+1})_{i \leq m}]^{Q \uparrow (d,\mathbf{q})}\}$ and the function $\gamma \mapsto G(\vec{\beta}, \gamma)$ is order preserving. 
Let $(X_i, (d_i,x_i, W_i)) \concat (X_k)$ be the partial level $\leq 2m+2$ tower of continuous type and let $\pi$ factor $(X_k, Q)$ such that $\pi(d_i, x_i) = (d_i, q_i)$ for each $1 \leq i < k$. The \emph{potential partial level $\leq 2m+2$ tower} induced by $F$ is
\begin{enumerate}
\item $(X_k, (d_i,x_i, W_i)_{1 \leq i < k})$, if $F$ is essentially continuous;
\item $(X_k, (d_i,x_i,W_i)_{1 \leq i < k} \concat (0,-1,\emptyset))$, if $F$ is essentially discontinuous and has uniform cofinality $(0,-1)$;
\item $(X_k, (d_i,x_i,W_i)_{1 \leq i < k} \concat (d^{+}, x^{+},W^{+}))$, if $F$ is essentially discontinuous and has uniform cofinality $(d_{*},\mathbf{q}_{*})$, $d_{*}>0$, $(X_k, (d^{+},x^{+},W^{+}))$ is a partial level $\leq 2m+2$ tree with uniform cofinality $(d_{*}, \mathbf{x}_{*})$, $\comp{d_{*}}{\pi}(\mathbf{x}_{*}) = \mathbf{q}_{*}$.
\end{enumerate}
The \emph{approximation sequence} of $F$ is $( F_i)_{1 \leq i \leq  k}$ where $F_i $ is a function on $[(\delta^1_{2i+1})_{i \leq m}]^{X_i\uparrow}$, $F_i(\vec{\beta}) = \sup \{F(\vec{\gamma}):$ $ \vec{\gamma} \in [(\delta^1_{2i+1})_{i \leq m}]^{Q \uparrow},  (\comp{d_1}{\gamma}_{q_1}, \ldots, \comp{d_{i-1}}{\gamma}_{q_{i-1}}) = (\comp{d_1}{\beta}_{x_1}, \ldots, \comp{d_{i-1}}{\beta}_{x_{i-1}}) \}$ for $1 \leq i \leq k$.

  Suppose $\bolddelta{2m+1} \leq \gamma < \bolddelta{2m+3}$ is a limit ordinal. Suppose $Q$ is a finite level $\leq 2m+2$ tree,  $\gamma = [F]_{\mu^Q}$,  the signature of $F$ is $(d_i,q_i)_{1 \leq i < k}$, the approximation sequence of $F$ is $(F_i)_{1 \leq i \leq k}$. Then the \emph{$Q$-signature} of  $\beta$ is $({d_i,q_i})_{1 \leq i < k}$, the \emph{$Q$-approximation sequence} of $\gamma$ is $([F_i]_{\mu^Q})_{1 \leq i \leq k}$,  $\gamma$ is \emph{$Q$-essentially continuous} iff $F$ is essentially continuous. The \emph{$Q$-uniform cofinality} of $\gamma$ is $\omega$ if $F$ has uniform cofinality $(0,-1)$, $\seed^Q_{(d,\mathbf{q})}$ if $f$ has uniform cofinality $(d,\mathbf{q}) \in \exdesc(Q)$. 
The \emph{$Q$-(potential) partial level $\leq 2m+2$ tower induced by $\gamma$} and the \emph{$Q$-factoring map} are the potential partial level $\leq 2m+2$ tower induced by $F$ and the factoring map induced by $F$ respectively. Assume by induction that:
\begin{enumerate}[resume*=ind]
\item \label{item:Q_signature}
If $x \in \mathbb{R}$ and $\gamma < \kappa_{2m+3}^x$, then the $Q$-potential partial level $\leq 2m+2$ tower induced by $\gamma$ and the $Q$-approximation sequence of $\gamma$ are uniformly $\Delta_1$-definable over $L_{\kappa_{2m+3}^x}[T_{2m+2},x]$ from $(T_{2m+2},x, Q, \gamma)$.
\end{enumerate}

Suppose $T,Q$ are level $\leq 2m+2$ trees and $\mathbf{C}  =  (d, \mathbf{t}, \tau) \in \desc(T,Q,*)$. For $h \in ((\delta^1_{2i+1})_{i \leq m})^{T \uparrow}$,
\begin{displaymath}
  h^Q_{\mathbf{C}}
\end{displaymath}
is the function on $[(\delta^1_{2i+1})_{i \leq m}]^{Q \uparrow}$ defined as follows:
\begin{enumerate}
\item If $d \leq 2m+1$, then $  h^Q_{\mathbf{C}} (\vec{\beta}) = ( \comp{\leq 2m+1}{h})^{\comp{\leq 2m}{Q}}_{\mathbf{C}} (\vec{\beta})$.
\item If $d=2m+2$, then $h^Q_{\mathbf{C}} (\vec{\beta}) = j^W(\comp{2m+2}{h}_{\mathbf{t}}) \circ \id^{Q,W}_{\tau} (\vec{\beta})$, or equivalently, $h^Q_{\mathbf{C}}([g]^Q) = [\comp{2m+2}{h}_{\mathbf{t}} \circ g^W_{\tau}]_{\mu^W}$. 
\end{enumerate}
Suppose additionally that $T$ is finite. Let
\begin{displaymath}
  \id_{\mathbf{C}}^{T,Q}
\end{displaymath}
be the function  $[g]^T\mapsto [g_{\mathbf{C}}^Q]_{\mu^Q}$, or equivalently,  $\vec{\xi} \mapsto \id_{\mathbf{C}}^{\comp{\leq 2m+1}{T}, \comp{\leq 2m}{Q}} ( \comp{\leq 2m+1}{\vec{\xi}} )$ if $d\leq 2m+1$, 
$\vec{\xi}\mapsto \tau^{Q , W} (  \comp{d}{\xi}_{\mathbf{q}})$ otherwise.
\begin{displaymath}
  \seed^{T,Q}_{\mathbf{C}}
\end{displaymath}
is the element represented modulo $\mu^T$ by $\id^{T,Q}_{\mathbf{C}}$.  By (\ref{item:tensor_product_QW}:$m$) and induction hypotheses at lower levels, we have:
\begin{enumerate}
\item If $d=1$, then $\seed^{T,Q}_{\mathbf{C}} = \seed^T_{\mathbf{C}}$.
\item If $d>1$ and $\mathbf{C} \in \desc(T,Q, W)$, then $\seed^{T,Q}_{\mathbf{C}} = \seed^{T, Q \otimes W}_{\mathbf{C}} = \seed^{T \otimes (Q \otimes W)}_{\mathbf{C}}$.
\end{enumerate}

We assume by induction that:
\begin{enumerate}[resume*=ind]
 \item\label{item:TQ_desc_order} Suppose $T,Q$  are finite level $\leq 2m+2$ trees,  $W$ is a finite level $\leq 2m+1$ tree, and suppose $\mathbf{C} , \mathbf{C}' \in \desc(T,Q,W)$, $\mathbf{C} \prec^{T,Q} \mathbf{C}'$,  $\mathbf{C} =  (d,\mathbf{t},\tau) $, $\mathbf{C}' = (d',\mathbf{t}', \tau') $. Then for any  $h \in ((\bolddelta{2i+1})_{i \leq m})^{T \uparrow}$, for any $f \in ((\bolddelta{2i+1})_{i \leq m})^{Q \uparrow}$, for any $\vec{\alpha} \in [((\bolddelta{2i+1})_{i  \leq m})]^{W \uparrow}$,
   \begin{displaymath}
\comp{d}{g}_{\mathbf{t}} \circ f^W_{\tau} (\vec{\alpha}) <  \comp{d'}{g}_{\mathbf{t}'} \circ f^W_{\tau'} (\vec{\alpha}) .
\end{displaymath}
 \item\label{item:TQ_desc_target_extension} Suppose $T, Q,Q'$ are finite level $\leq 2m+2$ trees, $Q$ is a proper subtree of $Q'$, $\mathbf{C} = (d,\mathbf{t}, \tau) \in \desc(T,Q,W)$, $\mathbf{C}' = (d',\mathbf{t}',\tau') \in \desc(T',Q', W')$, $\mathbf{C} \mathbf{C}' \res (T,Q)$. Suppose $\vec{E} = (E_i)_{i \leq m} \in \prod_{i \leq m} \nu_{2i+1}$, each $E_i$ is a club, $\eta \in E_i'$ iff $E_i \cap \eta$ has order type $\eta$, $\vec{E}' = (E_i)_{i \leq m}$. Then for any $h \in ((\bolddelta{2i+1})_{i \leq m})^{T \uparrow}$, for any $\vec{\alpha} \in [\vec{E}']^{Q \uparrow}$,
   \begin{displaymath}
     j^{W,W'} \circ h^Q_{\mathbf{C}} (\vec{\alpha}) = \sup \set{ h^{Q'}_{\mathbf{C}'} (\vec{\beta}) }{\vec{\beta} \in [E]^{Q' \uparrow}, \vec{\beta} \text{ extends $\vec{\alpha}$}}.
   \end{displaymath}
 \item\label{item:TQ_desc_origin_extension}  Suppose $T, T', Q$ are finite level $\leq 2m+2$ trees, $T$ is a proper subtree of $T'$ and $\mathbf{C} \in \desc(T,Q,*)$, $\mathbf{C}' \in \desc(T',Q, *)$, $\mathbf{C} =\mathbf{C}' \res (T,Q)$. Suppose $\vec{E} = (E_i)_{i \leq m} \in \prod_{i \leq m} \nu_{2i+1}$, each $E_i$ is a club, $\eta \in E_i'$ iff $E_i \cap \eta$ has order type $\eta$, $\vec{E}' = (E_i)_{i \leq m}$. Then for any $\vec{\xi} \in [\vec{E}']^{T' \uparrow}$,
   \begin{displaymath}
     \id^{T,Q}_{\mathbf{C}}(\vec{\xi})  =  \sup \set{ \id^{T',Q}_{\mathbf{C}'} (\vec{\eta})}{ \vec{\eta} \in [\vec{E}]^{T \uparrow}, \vec{\eta} \text{ extends } \vec{\xi}}. 
   \end{displaymath}
\end{enumerate}

Suppose $X$ is a level $\leq 2m+2$ tree and $\pi$ factors $(X,T,Q)$. For $h \in ((\delta^1_{2i+1})_{i \leq m})^{T \uparrow}$, let
\begin{displaymath}
  h_{\pi}^Q
\end{displaymath}
be the function sending $\vec{\beta}$ to $(h^Q_{\pi(d,x)}(\vec{\beta}))_{(d,x) \in \dom(X)}$. If $Q$ is finite, let
\begin{displaymath}
  \id^{T,Q}_{\pi}
\end{displaymath}
be the map sending $[h]^T$ to $[h^Q_{\pi}]_{\mu^Q}$. So $\id^{T,Q}_{\pi}(\vec{\xi}) = (\id^{T,Q}_{\pi(d,x)}(\vec{\xi}))_{(d,x) \in \dom(X)}$. If $T,Q$ are both finite, put
\begin{displaymath}
  \seed^{T,Q}_{\pi} = [\id^{T,Q}_{\pi}]_{\mu^T}.
\end{displaymath}
We assume by induction that:
\begin{enumerate}[resume*=ind]
\item \label{item:XT_factor}Suppose $X,T,Q$ are finite level $\leq 2m+2$ trees. Then for any $A \in \mu^X$, $\seed^{T,Q}_{\pi} \in j^T \circ j^Q(A)$.
\end{enumerate}

We define
\begin{displaymath}
  \pi^{T,Q} :  \mathbb{L}_{\bolddelta{2m+3}}[j^X(T_{2m+2})] \to \mathbb{L}_{\bolddelta{2m+3}}[j^T \circ j^Q(T_{2m+2})]
\end{displaymath}
by sending $j^X(h)(\seed^X) = j^T \circ j^Q(h) (\seed^{T,Q}_{\pi})$.

We assume by induction that: 
\begin{enumerate}[resume*=ind]
% \item \label{item:regular_cardinals_2}
% Suppose $Q$ is a finite level $\leq 2m+2$ tree. Then the set of $\mathbb{L}_{\bolddelta{2m+3}}[j^Q(T_{2m+2})]$-regular cardinals is $\set{\seed^Q_{\mathbf{q}}}{\mathbf{q} \in \exdesc(Q) \text{ is regular}}$.
\item \label{item:tensor_product_TQ} 
Suppose $X,T,Q$ are finite level $\leq 2m+2$ trees and $\pi$ factors $(X,T)$. Then
\begin{enumerate}
\item $j^T \circ j^Q = j^{T \otimes Q}$.
    \item $j^Q ( \pi ^T \res a )  =  (Q \otimes \pi)^{Q \otimes T} \res j^Q(a)$ for any $a \in \mathbb{L}_{\bolddelta{2m+3}}[T_{2m+2}]$;
    \item  $\pi^T  \res \mathbb{L}_{\bolddelta{2m+3}}[j^{X \otimes Q}(T_{2m+2})] = (\pi \otimes Q)^{T \otimes Q}$.
\end{enumerate}
\end{enumerate}

Suppose $Q$ is a finite level $\leq 2m+2$ tree. Suppose $(d,\mathbf{q})\in \exdesc(Q)$, and if $d>1$ then $\mathbf{q} = (q,P, \vec{p})$. Put
\begin{displaymath}
  \llbracket(d,\mathbf{q}) \rrbracket _Q =
  \begin{cases}
    \wocode{(1, (\mathbf{q}))}_{<^Q} & \text{ if } d=1, \\
    [\vec{\alpha} \mapsto \wocode{ (d, \vec{\alpha} \oplus_{\comp{d}{Q}} q)}_{<^Q}]_{\mu^P} & \text{ if }d>1.
  \end{cases}
\end{displaymath}
To save ink, put $\llbracket d, \mathbf{q} \rrbracket_Q  = \llbracket (d,\mathbf{q}) \rrbracket_Q$. 
If in addition, $d>1$ and $\mathbf{q}=(q,\dots)\in \desc(\comp{d}{Q})$ of discontinuous type, put  $\llbracket d,q \rrbracket_Q = \llbracket d, \mathbf{q} \rrbracket_Q$. If $\pi$ factors level $\leq 2m+2$ trees $(Q,T)$, then $\pi$ is said to \emph{minimally factor} $(Q,T)$ iff $X,T$ are both $\Pi^1_{2m+2}$-wellfounded and for any $(d,q) \in \dom(Q)$, $\llbracket d,q \rrbracket_Q = \llbracket \pi(d,q) \rrbracket_T$. Assume by induction that:
\begin{enumerate}[resume*=ind]
\item \label{item:order_type_embed} Suppose $X,T$ are $\Pi^1_{2m+2}$-wellfounded level $\leq 2m+2$ trees. Then there exist a $\Pi^1_{2m+2}$-wellfounded level $\leq 2n$ tree $Q$ and a map $\pi$ minimally factoring $(X,T \otimes Q)$.
\end{enumerate}

% In the discussions about level-($2m+3$) trees, the case $m=0$ is completely analyzed in \cite{sharpI}. Suppose now $m>0$ and we define the following related to level-($2m+3$) trees and level $\leq 2m+3$ trees. 
Suppose $R$ is a level-($2m+3$) tree. The \emph{ordinal representation} of $R$ is the set
\begin{align*}
  \rep(R) =& \set{\vec{\alpha} \oplus_R r}{r \in \dom(R), \vec{\alpha} \text{ respects } R_{\tree}(r)} \\
 & \cup \set{\vec{\alpha} \oplus_R r \concat(-1)}{r \in \dom(R), \vec{\alpha} \text{ respects }R(r)}. 
\end{align*}
Here for $r \in \bardom(R)$ of length $k$, $\vec{\alpha} \oplus_R r = (r(0), \alpha_{R_{\node}(r \res 1)}, r(1), \dots, \alpha_{R_{\node}(r)}, r(k-1))$.
$\rep(R)$ is endowed with the $<_{BK}$ ordering
\begin{displaymath}
  <^R = <_{BK} \res \rep(R).
\end{displaymath}
Assume by induction that:
\begin{enumerate}[resume*=ind]
\item \label{item:rep_R} A level-$(2m+3)$ tree $ R$ is $\Pi^1_{2m+3}$-wellfounded iff $<^R$ is a wellordering.
\end{enumerate}

Suppose $B \in \mathbb{L}_{\bolddelta{3}}[T_{2m+2}]$. Define 
\begin{displaymath}
  F \in {B}^{R \uparrow}
\end{displaymath}
iff $F \in \mathbb{L}_{\bolddelta{2m+3}}[T_{2m+2}]$ is an order preserving function from $\rep(R)$ to $ B$. If $F \in (\bolddelta{2m+3})^{R \uparrow}$, then for any $r \in \dom(R)$,  ${F}_r$ is a function on $[(\bolddelta{2i+1})_{i \leq m}]^{R_{\tree}(r) \uparrow}$ that sends $\vec{\beta}$ to $F(\vec{\beta} \oplus_R r)$, and $F$ represents a tuple of ordinals
\begin{displaymath}
  [F]^R = ([F]^R_r)_{r \in \dom(R)}
\end{displaymath}
where $[F]^R_r = [F_r]_{\mu^{R_{\tree}(r)}}$ for $r \in \dom(R)$. Let
\begin{displaymath}
  [B]^{R \uparrow} = \set{ [F]^R}{ F \in B^{R \uparrow}}.
\end{displaymath}
A tuple of ordinals $\vec{\gamma} = (\gamma_r)_{r \in \dom(R)}$ \emph{respects} $R$ iff $\vec{\gamma} \in [\bolddelta{2m+3}]^{R \uparrow}$. $\vec{\gamma}$ \emph{weakly respects $R$} iff for any $r,r' \in \dom(R)$, if $r$ is a proper initial segment of $r'$, then $j^{R_{\tree}(r), R_{\tree}(r')} (\gamma_r) > \gamma_{r'}$. If $\mathbf{r}=(r, Q, \dots) \in \exdesc(R)$ and $F \in (\bolddelta{2m+3})^{R \uparrow}$, define $F_{\mathbf{r}}$ to be a function on $[(\delta^1_{2i+1})_{i \leq m}]^{Q \uparrow}$: if  $\mathbf{r} \in \desc(R)$, then $F_{\mathbf{r}} = F_r$; if $\mathbf{r} \notin \desc(R)$, then  $F_{\mathbf{r}} ( \vec{\beta} ) = F_r(\vec{\beta}\res R_{\tree}(r))$.  
If $\vec{\gamma} =(\gamma_r)_{r \in \dom(R)}\in [\bolddelta{2m+3}]^{R \uparrow}$, put $\gamma_{\mathbf{r}} = [F_{\mathbf{r}}]_{\mu^Q}$. If $\mathbf{r} \in \desc(R)$ and $\mathbf{A} = (\mathbf{r}, \pi, T) \in \exexdesc(R)$, put $\gamma_{\mathbf{A}} = \pi^T(\gamma_{\mathbf{r}})$. Put $\gamma_{\emptyset} = \gamma_{(\emptyset,\emptyset,\emptyset)} = \bolddelta{2m+3}$.  Thus, if $\mathbf{r} \in \desc(R)$ is of discontinuous type, then $\gamma_{\mathbf{r}} = \gamma_r$; if $\mathbf{r}\notin \desc(R)$, then $\gamma_{\mathbf{r}} = j^{R_{\tree}(r), Q} (\gamma_r) = \gamma_{(\mathbf{r}, \id_{R_{\tree}(r)}, Q)}$. 
The next induction hypothesis computes the remaining case when $\mathbf{r}\in \desc(R)$ is of continuous type, justifying that $\gamma_{\mathbf{r}}$ does not depend on the choice of $F$. 
\begin{enumerate}[resume*=ind]
\item 
    \label{item:gamma_r_continuous_type}
    Suppose $R$ is a level-($2m+3$) tree, $\vec{\gamma} \in [\bolddelta{2m+3}]^{R \uparrow}$, $\mathbf{r} = (r, Q, \overrightarrow{(d,q, P)}) \in \desc(R)$ is of continuous type. Then $\gamma_{\mathbf{r}} = j_{\sup}^{R_{\tree}(r^{-}), Q} (\gamma_{r^{-}})$.
\end{enumerate}

 We assume by induction that: 
\begin{enumerate}[resume*=ind]
\item\label{item:respecting_R} 
  Suppose $R$ is a level-($2m+3$) tree and $\vec{\gamma} = (\gamma_r)_{r \in \dom(R)}$ is a tuple of ordinals in $\bolddelta{2m+3}$.  Then $\vec{\gamma}$ respects $R$ iff the following holds:
  \begin{enumerate}
  \item For any $r \in \dom(R)$, the $R_{\tree}(r)$-potential partial level $\leq 2m+2$ tower induced by $\gamma_r$ is $R[r]$, and the $R_{\tree}(r)$-approximation sequence of $\gamma_r$ is $(\gamma_{p \res l})_{1 \leq  l \leq  \lh(r)}$.
    \item If $R_{\tree}(r \concat (a)) = R_{\tree} (r \concat (b))$ and $a<_{BK}b$ then $\gamma_{p \concat (a)} <\gamma_{p \concat (b)}$.
  \end{enumerate}
Moreover, if $B \in \mathbb{L}_{\bolddelta{2m+3}}[T_{2m+2}]$ is a closed set, $B'$ is the set of limit points of $B$, then $\vec{\gamma} \in [B]^{R \uparrow}$ iff $\vec{\gamma}$ respects $R$ and for each $r \in \dom(R)$, $\gamma_r \in j^{R_{\tree}(r)}(B')$.
\end{enumerate}

Define $C^{*}_{2m+3} = \set{\xi < \bolddelta{2m+3}}{\text{for any finite level $\leq 2n+2$ tree }Q, j^Q_{\sup}(\xi) = \xi}$. By (\ref{item:local_definability_jQ}:$m$),  $C^{*}\cap \kappa_{2m+3}^x$ has order type $\kappa_{2m+3}^x$, and hence $C^{*}$ has order type $\bolddelta{2m+3}$. Suppose $R$ is a level-($2m+3$) tree. $\vec{\gamma}$ \emph{strongly respects} $R$ iff $\vec{\gamma}\in [C^{*}]^{R \uparrow}$. 
The function $\mathbf{A} \mapsto \corner{\mathbf{A}}$ is defined exactly as in  \cite{sharpII}. So are the relations $\mathbf{A} \prec^R \mathbf{A}'$, $\mathbf{r} \prec^R \mathbf{A}$, $r \prec^R \mathbf{A}$, etc.\  for $\mathbf{r} \in \exdesc(R)$,  $r \in \dom(R)$. Define $\prec^R_{*} = \prec \res \exexdesc(R)$, $\sim^R_{*} = \sim \res \exexdesc(R)$.  If $\mathbf{A} = (\mathbf{r}, \pi, T)$ and $\vec{\gamma}$ respects $R$, let $\gamma_{\mathbf{A}} = \pi^T(\gamma_{\mathbf{r}})$. 
Assume by induction that: 
\begin{enumerate}[resume*=ind]
\item\label{item:R_description} Suppose $R$ is a level-($2m+3$) tree,  $\mathbf{A}, \mathbf{A}' \in \exdesc(R)$,  $\vec{\gamma}$ strongly respects $R$.  Then ${\mathbf{A}} \prec^R{\mathbf{A}}'$ iff $\gamma_{\mathbf{A}} < \gamma_{\mathbf{A}'}$; ${\mathbf{A}} \sim^R{\mathbf{A}}'$ iff $\gamma_{\mathbf{A}} = \gamma_{\mathbf{A}'}$.
\end{enumerate}
% If $\rho$ factors.....

Suppose now $R$ is a finite level $\leq 2m+3$ tree. Then $\rep(R) = \cup_d \se{d} \times \rep(\comp{d}{R})$. 
Suppose $\vec{B} = (B_i)_{i \leq m+1} \in \mathbb{L}[T_{2m+3}]$. Define $  f \in \vec{B}^{R \uparrow}$
iff $f \in \mathbb{L}[T_{2m+3}]$ is an order preserving function from $\rep(R)$ to $\cup_i B_i$ such that for any $i$, $\ran(\comp{i}{f}) \subseteq B_i$. Define $  [B]^{R \uparrow} = \set{ [f]^R}{ f \in B^{R \uparrow}}.
$ $\vec{\alpha} =(\alpha_{(d,r)})_{(d,r) \in \dom(R)} =  (\comp{d}{\alpha}_r)_{r \in \dom(R)}$ \emph{respects} $R$ iff $\vec{\alpha} \in [(\bolddelta{2i+1})_{i \leq n}]^{R \uparrow}$. 

Suppose $Y$ is a level $\leq 2m+3$ tree, $T$ is a level $\leq 2m+2$ tree, and $\mathbf{B} =(d, \mathbf{y},\pi)\in \desc(Y,T,*)$, $F \in ((\bolddelta{2i+1})_{i \leq m+1})^{Y \uparrow}$. Then 
\begin{displaymath}
  F^T_{\mathbf{B}} : [\omega_1]^{T \uparrow} \to \bolddelta{3}
\end{displaymath}
is the function that sends $\vec{\xi}$ to $j^Q(\comp{d}{F}_{\mathbf{y}}) \circ \id^{T,Q}_{\pi} (\vec{\xi})$, or equivalently, sends
$[h]^T$ to $[\comp{d}{F}_{\mathbf{y}} \circ h_{\pi}^{Q}]_{\mu^{Q}}$. % Note that $F_{\mathbf{y}}\circ h^Q_{\pi}$ has signature $\sign(Q, \overrightarrow{(d,q,P)})$, is essentially discontinuous, and has uniform cofinality $\ucf(Q, \overrightarrow{(d,q,P)})$. % 
% Of course, $F^T_{\mathbf{B}}$ is meaningful only when $T$ is $\Pi^1_2$-wellfounded. 
\begin{displaymath}
  \id^{Y,T}_{\mathbf{B}}
\end{displaymath}
is the function $[F]^Y \mapsto [F^T_{\mathbf{B}}]_{\mu^T}$, or equivalently, $\vec{\gamma} \mapsto \pi^{T,Q}(\gamma_{\mathbf{y}})$. 
\begin{displaymath}
  \seed^{Y,T}_{\mathbf{B}}
\end{displaymath}
is the element represented modulo $\mu^Y$ by $\id^{Y,T}_{\mathbf{B}}$. 
 We assume by induction that:
 \begin{enumerate}[resume*=ind]
 \item\label{item:YT_desc_order} Suppose $Y$ is a finite level $\leq 2m+3$ tree, $T,Q$  are finite level $\leq 2m+2$ trees, and suppose $\mathbf{B} , \mathbf{B}' \in \desc(Y,T,Q)$, $\mathbf{B} \prec^{Y,T} \mathbf{B}'$,  $\mathbf{B} =  (d,\mathbf{y},\pi) $, $\mathbf{B}' = (d', \mathbf{y}', \pi') $. Then for any  $F \in ((\delta^1_{2i+1})_{i \leq m+1})^{Y \uparrow}$, for any $h \in ((\bolddelta{2i+1})_{i \leq m})^{T \uparrow}$, for any $\vec{\beta} \in [((\bolddelta{2i+1})_{i < m})]^{Q \uparrow}$,
   \begin{displaymath}
{F}_{\mathbf{y}} \circ h^Q_{\pi} (\vec{\beta}) <  \comp{d'}{F}_{\mathbf{y}'} \circ h^Q_{\pi'} (\vec{\beta}) .
\end{displaymath}
 \item\label{item:YT_desc_target_extension} Suppose $Y$ is a finite level $\leq 2m+3$ tree, $T, T'$ are finite level $\leq 2m+2$ trees, $T$ is a proper subtree of $T'$, $\mathbf{B} = (\mathbf{y}, \pi) \in \desc(Y,T,Q)$, $\mathbf{B}' = (\mathbf{y}',\pi') \in \desc(Y',T', Q')$, $\mathbf{B}= \mathbf{B}' \res (Y,T)$. Suppose $\vec{E} = (E_i)_{i \leq m} \in \prod_{i \leq m} \nu_{2i+1}$, each $E_i$ is a club, $\eta \in E_i'$ iff $E_i \cap \eta$ has order type $\eta$, $\vec{E}' = (E_i)_{i \leq m}$. Then for any $F \in ((\delta^1_{2i+1})_{i \leq m+1})^{Y \uparrow}$, for any $\vec{\xi} \in [\vec{E}']^{T \uparrow}$,
   \begin{displaymath}
     j^{Q,Q'} \circ F^T_{\mathbf{B}} (\vec{\xi}) = \sup \set{ F^{T'}_{\mathbf{B}'} (\vec{\eta}) }{\vec{\eta} \in [E]^{T' \uparrow}, \vec{\eta} \text{ extends $\vec{\xi}$}}.
   \end{displaymath}
 \item\label{item:YT_desc_origin_extension}  Suppose $Y, Y'$ are finite level $\leq 2m+3$ trees, $Y$ is a proper subtree of $Y'$, $T$ is a finite level $\leq 2m+2$ tree and $\mathbf{B} \in \desc(Y,T,*)$, $\mathbf{B}' \in \desc(Y',T, *)$, $\mathbf{B} = \mathbf{B}' \res (Y,T)$. Suppose $\vec{E} = (E_i)_{i \leq m+1} \in \prod_{i \leq m+1} \nu_{2i+1}$, each $E_i$ is a club, $\eta \in E'_i$ iff $E_i \cap \eta$ has order type $\eta$, $\vec{E}' = (E'_i)_{i \leq m+1}$. Then for any $\vec{\gamma} \in [{\vec{E}}']^{Y \uparrow}$,
   \begin{displaymath}
     \id^{Y,T}_{\mathbf{B}}(\vec{\gamma})  =  \sup \set{ \id^{Y',T}_{\mathbf{B}'} (\vec{\delta})}{ \vec{\delta} \in [\vec{E}]^{Y' \uparrow}, \vec{\delta} \text{ extends } \vec{\gamma}}. 
   \end{displaymath}
\end{enumerate}

Suppose $R$ is a level $\leq 2m+3$ tree and $\rho$ factors $(R,Y,T)$. For $F \in ((\delta^1_{2i+1})_{i \leq m+1})^{Y \uparrow}$, let
\begin{displaymath}
  F_{\rho}^Y
\end{displaymath}
be the function sending $\vec{\gamma}$ to $(F^T_{\rho(d,r)}(\vec{\gamma}))_{(d,r) \in \dom(R)}$. If $T$ is finite, let
\begin{displaymath}
  \id^{Y,T}_{\rho}
\end{displaymath}
be the map sending $[F]^Y$ to $[F^T_{\rho}]_{\mu^T}$. So $\id^{Y,T}_{\rho}(\vec{\gamma}) = (\id^{Y,T}_{\rho(d,r)}(\vec{\gamma}))_{(d,r) \in \dom(R)}$. 
% If  $Y,T$ are both finite, put
% \begin{displaymath}
%   \seed^{Y,T}_{\rho} = [\id^{Y,T}_{\rho}]_{\mu^Y}
% \end{displaymath}

Suppose $R$ is $\Pi^1_{2m+3}$-wellfounded. 
Put $\llbracket 2m+3, \emptyset \rrbracket_R = \ot(<^R)$.
For $(d,\mathbf{r}) \in \exexdesc(R)$,  $\mathbf{r}=(r,Q, \overrightarrow{(d,q,P)})$, put
\begin{displaymath}
 \llbracket d, \mathbf{r} \rrbracket_{R} = [\vec{\beta} \mapsto \wocode{d, \vec{\beta}\oplus_R r}_{<^R}]_{\mu^Q}.
\end{displaymath}
If $(d,\mathbf{r}) \in \desc(R)$ is of discontinuous type, put $\llbracket d,r \rrbracket_{R} = \llbracket d, \mathbf{r} \rrbracket_{R}$. We say that $\rho$ \emph{minimally factors} $(R,Y)$ iff $\rho$ factors $(R,Y)$, $R,Y$ are both $\Pi^1_{2m+3}$-wellfounded and $\llbracket d, r \rrbracket_R = \llbracket \rho(d,r) \rrbracket_Y$ for any $r \in \dom(R)$.
Assume the induction hypothesis:
\begin{enumerate}[resume*=ind]
\item \label{item:order_type_embed_another} Suppose $R,Y$ are $\Pi^1_{2m+3}$-wellfounded level $\leq 2m+3$ trees and $\llbracket 2m+3, \emptyset \rrbracket_R \leq \llbracket 2m+3, \emptyset \rrbracket_Y$. Then there exist a $\Pi^1_{2m+2}$-wellfounded level $\leq 2m+2$ tree $T$ and a map $\rho$ minimally factoring $(R,Y \otimes T)$. If $\llbracket 2m+3, \emptyset \rrbracket_R < \llbracket 2m+3, \emptyset \rrbracket_Y$, we can further assume that for some $\mathbf{B} \in \desc(Y,T,*)$ we have $\lh(\mathbf{B}) = 1$ and $\llbracket 2m+3,  \emptyset \rrbracket_R = \llbracket \mathbf{B} \rrbracket_Y$. 
\end{enumerate}

If $R$ is a level-($2m+3$) tree, then $\llbracket \mathbf{r} \rrbracket_R = \llbracket 2m+3 \rrbracket_{R'}$ and $\llbracket r \rrbracket_R = \llbracket 2m+3 \rrbracket_{R'}$, where $R' = Q^{(2m+2)}_0 \oplus R$. If $R,Y$ are level-($2m+3$) trees, then $\rho$ \emph{minimally factors} $(R,Y)$ iff $\rho$ extends to $\rho'$ which minimally factors $(Q^{(2m+2)}_0 \oplus R, Q^{(2m+2)}_0 \oplus Y)$. 

Suppose $R$ is a level-$(2m+3)$ tree and $\dom(R) = \se{((0))}$, $R_{\node}(((0))) = (d, q)$, $Q$ is a completion of $R(((0)))$. We say that  \emph{$\epsilon$ is $Q$-represented by $T$} iff $Q$ is a subtree of $T$ and $\llbracket d,q \rrbracket_T = \epsilon$. 
Suppose $Q$ is a finite level $\leq 2m+2$ tree and $\vec{\epsilon} = (\comp{d}{\epsilon}_t)_{(d,t) \in \dom(Q)}$ is a tuple of ordinals indexed by $\dom(Q)$. We say that \emph{$\vec{\epsilon}$ is represented by $Q'$} iff $Q$ is a subtree of $Q'$, $Q'$ is $\Pi^1_{2m+2}$-wellfounded and $\vec{\epsilon} =( \llbracket d,t \rrbracket_{Q'})_{(d,t) \in \dom(Q')}$. Similarly define a tuple $\vec{\epsilon}  = (\epsilon_r)_{r \in \dom(R)}$ being represented by a level-($2m+3$) tree $R$.
Assume by induction that:
\begin{enumerate}[resume*=ind]
\item \label{item:represent_cofinal} Suppose $R$ is a level-$(2m+3)$ tree and $\dom(R) = \se{((0))}$, $R(((0))) =(Q^{(2m+2)}_0,(d, q, P))$, $2k+1\leq d<2k+3$, $Q$ is a completion of $R(((0)))$. Then cofinally many ordinals in $j^P (\bolddelta{2k+1})$ are $Q$-represented by some level $\leq 2m+2$ tree. 
\item\label{item:represent} Suppose $Q$ is a finite level $\leq 2m+2$ tree and $\vec{\epsilon}$ respects $Q$. Then there is $Q'$ extending $Q$ such that $\vec{\epsilon}$ is represented by $Q'$.
 \item \label{item:represent_another} Suppose $R$ is a finite level-$(2m+3)$ tree and $\vec{\epsilon}$ respects $R$. Then there is $R'$ extending $R$ such that $\vec{\epsilon}$ is represented by $R'$. 
\end{enumerate}

To every ordinal $\xi < E(2m+3)$, we assign $\widehat{\xi}$ as follows:
\begin{enumerate}
\item If $\xi < E(2m+1)$, then $\widehat{\xi}$ has been defined by induction.
  \item If $0 < \eta = \omega^{E(2m)+\eta_1}+ \dots + \omega^{E(2m)+\eta_k} + \chi  < E(2m+2)$, $E(2m+1)>\eta_1 \geq \dots \geq \eta_k $, $E(2m+1)>\chi$, then $\widehat{\omega^{\eta}} = u^{(2m+1)}_{\eta_1+1} \cdot \dots \cdot u^{(2m+1)}_{\eta_k+1} \cdot \widehat{\xi}$.
  \item If $0 < \xi = \omega^{\eta_1}+ \dots + \omega^{\eta_k}$, $E(2m+2) > \eta_1\geq \dots \geq \eta_k$, then $\widehat{\xi} = \widehat{\omega^{\eta_1}}+\dots+ \widehat{\omega^{\eta_k}}$.
\end{enumerate}

Let $R^{(2m+3)}_{E(2m+3)}$ be the unique (up to an isomorphism) level-($2m+3$) tree such that
\begin{enumerate}
\item  for any finite level-($2m+3$) tree $Y$, there exists $\rho$ which minimally factors $(Y,R^{(2m+3)}_{E(2m+3)})$;
\item if $r \in \dom(R^{(2m+3)}_{E(2m+3)})$ then there exist a finite $Y$ and $\rho$ which minimally factors $(Y,R)$ such that $r \in \dom(\rho)$.
\end{enumerate}
We fix the following representation of $R^{(2m+3)}_{E(2m+3)}$, whose domain consists of finite tuples of ordinals in $E(2m+3)$:
\begin{enumerate}
\item $(\xi_1) \in \dom(R^{(2m+3)}_{E(2m+3)})$ iff $0<\xi_1< E(2m+3)$. $R^{(2m+3)}_{E(2m+3)}((\xi_1))$ is the $Q^{(2m)}_0$-partial level $\leq 2m+2$ tree induced by $\widehat{\xi_1}$. 
\item If $r = (\xi_1,\dots,\xi_{k-1}) \in \dom(R^{(2m+3)}_{E(2m+3)})$, then $r \concat (\xi_k) \in \dom(R^{(2m+3)}_{E(2m+3)})$ iff ${\xi_k} < E(2m+3)$ and there exists a completion $Q^{+}$ of $R^{(2m+3)}_{E(2m+3)}(r)$ such that 
the $Q^{+}$-approximation sequence of $\widehat{\xi_k}$ is $(\widehat{\xi_i})_{1 \leq i \leq k}$;  if $r \concat(\xi_k) \in \dom(R^{(2m+3)}_{E(2m+3)})$ and $Q^{+}$ is the unique such completion, then $R^{(2m+3)}_{E(2m+3)}(r \concat (\xi_k))$ is the $Q^{+}$-partial level $\leq 2m+2$ tree induced by $\widehat{\xi_k}$. 
\end{enumerate}
Therefore, $\llbracket \emptyset \rrbracket_R = u^{(2m+1)}_{E(2m+1)}$ and if $r=(\xi_1,\dots,\xi_k) \in \dom(R^{(2m+3)}_{E(2m+3)})$, then $\llbracket r \rrbracket_R = \widehat{\xi_k}$. If $Y$ is a finite level-3 tree, then the map $y \mapsto r_y$ minimally factors $(Y,R)$, where if $(\llbracket y \res i \rrbracket_Y) _{1 \leq i \leq \lh(y)}=  (\widehat{\xi_1},\dots, \widehat{\xi_{\lh(y)}})$ then $r_y = (\xi_1,\dots,\xi_{\lh(y)})$. For $0 < \xi < E(2m+3)$, let $R^{(2m+3)}_{\xi} = R^{(2m+3)}_{E(2m+3)} \res (\xi) $.

 \section{The level-($2n+1$) sharp}
\label{sec:induction-under-pi_det}

From now on, we assume $\boldpi{2n+1}$-determinacy. This is equivalent to ``$\forall x \in \mathbb{R}~$ there is an $(\omega_1,\omega_1)$-iterable $M_{2n}^{\#}$ by Neeman \cite{nee_opt_I,nee_opt_II} and Woodin \cite{hod_as_a_core_model,woodin-handbook,SUW}.'' We shall show the induction hypotheses (\ref{item:subset_does_not_increase_0}:$n$)-(\ref{item:q_respecting}:$n$).

(\ref{item:subset_does_not_increase_0}:$n$) follows from Steel's computation of $L_{\bolddelta{2n+1}}[T_{2n+1}] = M_{2n,\infty}^{-}$.
We proceed with the definition of the operator $x \mapsto x^{(2n+1)\#}$. It will be basically copying the arguments in  \cite{sharpII}.

\subsection{The equivalence of $M_{2n}^{\#}$ and $0^{(2n+1)\#}$}
\label{sec:sharp_existence}
Suppose
$C \subseteq \bolddelta{2n+1}$. $C$ is \emph{firm} iff every member of $C$ is additively closed, the set $\set{\xi }{\xi = \ot(C \cap \xi)}$ has order type $\bolddelta{2n+1}$ and $C \cap \xi \in \mathbb{L}_{\bolddelta{2n+1}}[T_{2n}]$ for all $\xi < \bolddelta{2n+1}$. $C$ is a set of \emph{potential level-($2n+1$) indiscernibles for $M_{2n,\infty}^{-}$} iff for any level-($2n+1$) tree $R$, for any $F, G \in C^{R\uparrow} \cap \mathbb{L}_{\bolddelta{2n+1}}[T_{2n}]$,
\begin{displaymath}
  (M_{2n}^{-}; [F]^R) \equiv   (M_{2n}^{-}; [G]^R) .
\end{displaymath}

Say that $\delta$ is an $M_{2n-2}$-Woodin cardinal iff $M_{2n-2}(V_{\delta})\models ``\delta$ is Woodin''. By a theorem of Woodin, putting $\kappa = u^{(2n-1)}_{E(2n-1)}$, then
$M_{2n-2,\infty}^{-} (x)\models ``\kappa$ is the least $M_{2n-2}$-Woodin cardinal''.
By $\boldpi{2n+1}$-determinacy, if $C$ is the set of $M_{2n-2}$-Woodin cardinals in $M_{2n}^{-}(M_{2n}^{\#})$ and their limits, then $C$ is a firm set of potential level-($2n+1$) indiscernibles for $M^{-}_{2n,\infty}$. We define
\begin{displaymath}
  0^{(2n+1)\#}
\end{displaymath}
to be a map sending a level-($2n+1$) tree $R$ to $0^{(2n+1)\#}(R)$, where $\gcode{\varphi} \in 0^{(2n+1)\#}(R)$ iff $\varphi$ is an $\mathcal{L}^R$-formula and for all $\vec{\gamma} \in [C]^{R\uparrow}$,
\begin{displaymath}
  (M_{2n,\infty}^{-};\vec{\gamma}) \models \varphi.
\end{displaymath}
 If $R$ is a finite level-($2n+1$) tree, we have:
  \begin{enumerate}
  \item $0^{(2n+1)\#}(R)$ is a $\game((\llbracket \emptyset \rrbracket_R+\omega)\text{-}\Pi^1_{2n+1})$ real, and
  \item the universal $\game(\llbracket \emptyset \rrbracket_R\text{-}\Pi^1_{2n+1})$ real is many-one reducible to $0^{(2n+1)\#}(R)$, uniformly in $R$. 
  \end{enumerate}
Hence, by induction hypotheses (\ref{item:bk_km}:$n-1$)(\ref{item:evensharp}:$n-1$), $0^{(2n+1)\#} \equiv_m M_{2n}^{\#}$. Relativizing to any real,
\begin{displaymath}
  x^{(2n+1)\#} \equiv_m M_{2n}^{\#}(x), \text{ uniformly in } x.
\end{displaymath}
This verifies (\ref{item:odd_sharp}:$n$).

\subsection{Syntactical properties of $0^{(2n+1)\#}$}
\label{sec:synt-prop}

By (\ref{item:bk_km}:$n-1$)(\ref{item:local_definability_jQ}:$n-1$), if $Q,T$ are finite level $\leq 2n$ trees and $\pi$ factors $(Q,T)$, then $j^Q \res \bolddelta{2n+1}$ and $\pi^T \res \bolddelta{2n+1}$ are $\Sigma^1_{2n+2}$ in the codes. We make the following informal symbols that will occur in a level-$(2n+1)$ EM blueprint:
\begin{enumerate}
\item If $Q$ is a finite level $\leq 2n$ tree,  $\underline{j^Q}( a ) = b$ iff for any $\xi $ cardinal cutpoint such that $\se{a,b} \in K|\xi$, the $\coll(\omega,\xi)$-generic extension satisfies $j^Q(\pi_{K|\xi, \infty} ( a) )= \pi_{K|\xi, \infty}(b) $.
\item If $\pi$ factors finite level $\leq 2n$ trees $(X, T)$,  $\underline{\pi^T}(a) = b$ iff for any $\xi $ cardinal cutpoint such that $\se{a,b} \in K|\xi$, the $\coll(\omega,\xi)$-generic extension satisfies $\pi^T(\pi_{K|\xi, \infty} ( a) )=  \pi_{K|\xi, \infty}(b)  $.
\item If $Q$ is a level $\leq 2n$ subtree of $Q'$, $Q'$ is finite, then $\underline{j^{Q,Q'}} = \underline{(\id_Q)^{Q,Q'}}$.
\item Corresponding to items 1-3, $  \underline{j^Q_{\sup}}, \underline{\pi^{T}_{\sup}}, \underline{j^{Q,Q'}_{\sup}}$ stand for functions on ordinals that send $\alpha$ to $\sup (\underline{j^Q})''\alpha$, $\sup (\underline{\pi^T})'' \alpha$, $\sup (\underline{j^{Q,Q'}})'' \alpha$ respectively.
% \item $\underline{S_{2n+1}}$ is the informal symbol such that  $(\emptyset,\emptyset) \in \underline{S_{2n+1}}$ and $ ( (R_i)_{ i\leq n}, (\alpha_i)_{i \leq n}  ) \in \underline{S_{2n+1}}$ iff $\vec{R}$ is a finite regular level-($2n+1$) tower and letting 
% $r_i \in \dom(R_{i+1}) \setminus \dom(R_i)$, then $r_k = (r_l)^{-} \to \alpha_k < \underline{j^{(R_n)_{\tree}(r_k), (R_n)_{\tree}(r_l)}} (\alpha_l)$.
% \item For $m<n$ and $0 < \xi \leq E(2m+1)$, 
% $\underline{u^{(2m+1)}_{\xi}}$ is the symbol so that for any $\xi > \underline{u^{(2m+1)}_{\xi}}$ cardinal cutpoint, the $\coll(\omega,\xi)$-generic extension satisfies $\pi_{K|\xi, \infty} (\underline{u^{(2m+1)}_{\xi}}) = {u^{(2m+1)}_{\xi}}$.
% \item  Suppose $T$ is a finite level $\leq 2n$ tree. If $\mathbf{D} \in \desc(T,U,*)$, $\wocode{ \mathbf{D}}_{\prec^{T\otimes U}_{*}} = \xi$, then $\underline{\seed^{T,U}_{\mathbf{D}}}=\underline{u^{(2n+1)}_{\xi+1}}$. If $(1,t) \in \dom(T)$, then $\underline{\seed^T_{(1,t)}} = \underline{\seed^{T,\emptyset}_{(1,t,\emptyset)}}$. If $(d,t) \in \dom(T)$, and $\comp{d}{T}_{\tree}(t) = S$,  then $\underline{\seed^T_{(2,t)}} = \underline{\seed^{T,S}_{(d, (t) \concat \comp{d}{T}[t], \id_{S,*})}}$. $\underline{\seed^T} = (\underline{\seed^T_{(d,t)}})_{(d,t) \in \dom(T)}$. 
\item If $k$ is a definable class function and $W$ is a definable class, then $k(W) = \bigcup \set{k(W \cap V_{\alpha})}{\alpha \in \ord}$. 
\item If  $X , T, T'$ are finite level $\leq 2n$ trees,  $T$ is a subtree of $T'$, $a \in \underline{j^X}(V)$, $R$ is a level-($2n+1$) tree, $\dom(R) = \se{((0))}$,   then 
  \begin{enumerate}
  \item $\underline{B^T_{X, a}} = \{\underline{\pi^{T \otimes Q}} (a) : Q$ finite level $\leq 2n$ tree,  $\pi $ factors $(X,T \otimes Q)\}$;
  \item $\underline{H^T_{X, a}}$ is the transitive collapse of the Skolem hull of $\underline{B^T_{X, a}} \cup \ran(\underline{j^T})$ in $\underline{j^T}(V)$ and  $\underline{\phi^T_{X, a}}: \underline{H^T_{X, a}} \to \underline{j^T}(V)$ is the associated elementary embedding;
  \item $\underline{j^T_{X, a}} = (\underline{\phi^T_{X, a}})^{-1} \circ \underline{j^T}$;
  \item $\underline{j^{T,T'}_{X, a}} = (\underline{\phi^{T'}_{X, a}})^{-1} \circ \underline{j^{T,T'}} \circ \underline{\phi^T_{X, a}}$;
  \item $\underline{B^T_{R,a}} = \underline{B^T_{Q^{(2n)}_0,a}} \cup \big( \bigcup \{ \underline{B^T_Q}: Q \text{ is a completion of } R(((0))))\}\big)$.
  \item  $\underline{H^T_{R,a}}$ is the transitive collapse of the Skolem hull of $\underline{B^T_{R,a}} \cup \ran( \underline{j^T} )$ in $\underline{j^T} (V) $ and $\underline{\phi^T_{R, a}}: \underline{H^T_{R, a}} \to \underline{j^T}(V)$ is the associated elementary embedding;
  \item $\underline{j^T_{R, a}} = (\underline{\phi^T_{R, a}})^{-1} \circ \underline{j^T}$;
  \item $\underline{j^{T,T'}_{R, a}} = (\underline{\phi^{T'}_{R, a}})^{-1} \circ \underline{j^{T,T'}} \circ \underline{\phi^T_{R, a}}$.
 \end{enumerate}
\item Suppose $R$ is a level-($2n+1$) tree. 
  \begin{enumerate}
  \item 
If $\mathbf{r} = (r,Q, \overrightarrow{(d,q,P)}) \in \exdesc(R)$, $\underline{c_{\mathbf{r}}}$ is the informal $\mathcal{L}^{R}$-symbol whose interpretation is
 \begin{displaymath}
   \underline{c_{\mathbf{r}}} =
   \begin{cases}
\underline{j_{\sup}^{R_{\tree}(r^{-}),Q}} (\underline{c_{r^{-}}}) & \text{if }  \mathbf{r} \in \desc(R) \text{ of continuous type} ,\\
\underline{c_r} & \text{if } \mathbf{r} \in \desc(R) \text{ of discontinuous type} ,\\
\underline{j^{R_{\tree}(r), Q}}(\underline{c_r}) & \text{if } \mathbf{r} \notin \desc(R).
   \end{cases}
 \end{displaymath}
\item If  $T,U$ are finite level $\leq 2n$ trees and $\mathbf{B} = (\mathbf{r}, \pi) \in \desc(R, T, U)$, $\mathbf{r} \neq \emptyset$, then $  \underline{c_{\mathbf{B}}^T}$ is the informal $\mathcal{L}^{R}$-symbol which stands for $\underline{\pi^{T,U}} ( \underline{c_{\mathbf{r}}})$.
\item If $\mathbf{A} = (\mathbf{r}, \pi, T) \in \exexdesc(R)$, $\mathbf{r} \neq \emptyset$, then $\underline{c_{\mathbf{A}}}$ is the informal $\mathcal{L}^R$-symbol which stands for $\underline{\pi^T}(\underline{c_{\mathbf{r}}})$.
  \end{enumerate}
\end{enumerate}

\begin{definition}\label{def:EM}
  A \emph{pre-level-($2n+1$)} blueprint is a function $\Gamma$ sending any finite level-($2n+1$) tree $Y$ to a complete consistent $\mathcal{L}^Y$-theory $\Gamma(Y)$ which contains all of the following axioms:
  \begin{enumerate}
  \item\label{item:EM_ZFC_n} $\ZFC+$ there is no inner model with $2n$ Woodin cardinals $+ V=K +$ there is no strong cardinal $+ V$ is closed under the $M_{2n-1}^{\#}$-operator.
  \item \label{item:EM_XTQ_factor_n}
    \label{item:EM_commutativity_n} %(Commutativity of embeddings)
    Suppose $X,T, Q,Z$ are finite level $\leq 2n$ trees, $\pi$ factors $(X,T)$, $\psi$ factors $(T,Z)$.
    \begin{enumerate}
    \item $\underline{j^T}: V \to \underline{j^T}(V)$ is $\mathcal{L}$-elementary. $\underline{j^{Q^{(2n)}_0}} $ is the identity map on $V$.
    \item $\underline{\pi^T}: \underline{j^X}(V) \to \underline{j^T}(V)$ is $\mathcal{L}$-elementary. $\underline{j^{Q^{(2n)}_0,T}} = \underline{j^T}$. $\underline{j^{T,T}}$ is the identity map on $\underline{j^T}(V)$.
    \item $\underline{(\psi \circ \pi)^Z} = \underline{\psi^Z} \circ \underline{\pi^T} $.
    \item $\underline{j^T} \circ \underline{j^Q} = \underline{j^{T \otimes Q}}$.
    \item $ \underline{j^Q} (\underline{\pi^T}) = \underline{(Q \otimes \pi)^{Q\otimes T}}$.
    \item $\underline{\pi^T} \res \underline{j^{X \otimes Q}}(V) = \underline{ (\pi \otimes Q)^{T \otimes Q}}$.
    \end{enumerate}
  \item \label{item:level-2-embedding_invariance_n} If $\xi$ is a cardinal and strong cutpoint, then $V^{\coll(\omega,\xi)}$ satisfies the following: If ${U}$ is a $\Pi^1_{2n}$-wellfounded level $\leq 2n$ tree, then $K|\xi$ and $(\underline{j^{{U}}})^K(K|\xi) $ are countable $\Pi^1_{2n+1}$-iterable mice and $(\underline{j^{{U}}})^K\res (K|\xi) $ is essentially an iteration map from $K|\xi$ to $(\underline{j^{{U}}})^K(K|\xi)$. Here $(\underline{j^{{U}}})^K$ stands for the direct limit map of $(\underline{j^{Z,Z'}})^K$ for $Z,Z'$ finite subtrees of $U$, $Z$ a finite subtree of $Z'$.
  \item \label{item:EM_c_t_are_ordinals_n} For any $y\in \dom(Y)$, ``$\underline{c_y} \in \ord$'' is an axiom.
  \item\label{item:EM_order_n} If $\mathbf{y} \prec^Y \mathbf{y}'$, then ``$\underline{c_{\mathbf{y}}} < \underline{c_{\mathbf{y}'}}$'' is an axiom; if $\mathbf{y} \sim^Y \mathbf{y}'$, then ``$\underline{c_{\mathbf{y}}} = \underline{c_{\mathbf{y}'}}$'' is an axiom.
  \end{enumerate}
  A level-($2n+1$) EM blueprint is a pre-level-($2n+1$) EM blueprint satisfying the \emph{coherency property}: if $R$ is a finite level-($2n+1$) tree, $Y,T$ are finite level $\leq 2n$ trees, $\rho$ factors $(R,Y,T)$, then for each $\mathcal{L}$-formula $\varphi(v_1,\ldots,v_n)$, for each $r_1,\ldots,r_n \in \dom(R)$,
  \begin{displaymath}
    \gcode{\varphi(\underline{c_{r_1}},\ldots,\underline{c_{r_n}})} \in \Gamma(R)
  \end{displaymath}
  iff
  \begin{displaymath}
    \gcode{\underline{j^T}(V) \models  \varphi (\underline{c_{\rho(r_1)}^T},\ldots,\underline{c_{\rho(r_n)}^T})}\in \Gamma(Y).
  \end{displaymath}
\end{definition}

If $\Gamma$ is a level-($2n+1$) EM blueprint and $Y$ is a level-($2n+1$) tree, the EM model $\mathcal{M}_{\Gamma,Y}$ is defined. If $\rho$ factors $(R,Y)$, then $\rho^Y_{\Gamma}$ is defined. If $T$ is a level $\leq 2n$ tree and $\rho$ factors $(R,Y,T)$, the notations $\mathcal{M}_{\Gamma,Y}^T$, $j_{\Gamma,Y}^T$, $\rho^{Y,T}_{\Gamma}$, etc.\ are defined.

A level-($2n+1$) EM blueprint is \emph{iterable} iff for any $\Pi^1_{2n+1}$-wellfounded level-($2n+1$) tree $Y$, $\mathcal{M}_{\Gamma,Y}$ is a $\Pi^1_{2n+1}$-iterable mouse. Unboundedness, weak remarkability and level $\leq 2n$ correctness are defined as in  \cite{sharpII}, only replacing every occurrence of ``level-3 tree'', ``level $\leq 2$ tree'', ``$\Pi^1_2$-wellfounded'' by ``level-($2n+1$) tree'', ``level $\leq 2n$ tree'', ``$\Pi^1_{2n}$-wellfounded'' respectively. 
If $\Gamma$ is a weakly remarkable level-($2n+1$) EM blueprint, for a level-($2n+1$) tree, the full model $\mathcal{M}^{*}_{\Gamma,Y}$ is defined. The notations $\mathcal{M}^{*,T}_{\Gamma,Y}$, $\rho^{*,Y}_{\Gamma}$, etc.\  carry over. For $\gamma$ respecting a level $\leq 2n$ tree $Q$, $c_{\Gamma, Q, \gamma}$ is defined. Define $c_{\Gamma,\gamma} = c_{\Gamma,Q^{(2n)}_0, \gamma}$ for any limit ordinal $\gamma < \bolddelta{2n+1}$. 
$\Gamma$ is \emph{remarkable} iff $\Gamma$ is weakly remarkable and
\begin{enumerate}
\item If $R$ is a level-($2n+1$) tree, $\dom(R) = \se{((0))}$, $R(((0)))$ is of degree 0, then 
$\Gamma(R)$ contains the axiom ``$\underline{c_{((0))}}$ is not measurable''.
  \item  If $R$ is a level-($2n+1$) tree, $\dom(R) = \se{((0))}$, $R(((0))) = (Q^{(2n)}_0, (d,q,P))$, then
$\Gamma(R)$ contains the following axiom:  
      if $\xi$ is a cardinal and strong cutpoint, $c = \underline{c_{((0))}}$, $b^Q = (\underline{\phi^Q_{R , c}})^{-1}(c)$ for $Q$ a completion of $R(((0)))$, then $V^{\coll(\omega,\xi)}$ satisfies the following: 
     \begin{enumerate}
\item  If $Q$ is a completion of $R(((0)))$ and 
 $\alpha$ is $Q$-represented by both $T$ and $T'$, then 
$(  (\underline{j^T_{R,c}})^K(K|\xi) ,  (\underline{j^{Q,T}_{R,c}})^K (b^Q)  ) \sim_{DJ} (  (\underline{j^{T'}_{R,c}})^K(K|\xi) , ( \underline{j^{Q,{T'}}_{R,c}})^K (b^Q)  ) $. 
\item 
 Let $F({\alpha}) = \pi_{(\underline{j^T_{R,c}})^K(K|\xi), \infty}( ( \underline{j^{Q,T}_{R,c}})^K  (b^Q)) $ for ${\alpha}$ $Q$-represented by $T$ and $Q$ a completion of $R(((0)))$. Then $\sup_{\alpha < j^P ( \bolddelta{2k+1} )} F(\alpha) = \pi_{K|\xi , \infty }(c)$ where $2k+1 \leq d < 2k+3$.  
 \end{enumerate}
\end{enumerate}

By appealing to (\ref{item:represent_another}:$n-1$), if $\Gamma$ is an iterable, weakly remarkable level-($2n+1$) EM blueprint, then:
\begin{enumerate}
\item The following are equivalent:
  \begin{enumerate}
  \item $\Gamma$ is remarkable.
  \item The map $\gamma \mapsto c_{\Gamma,\gamma}$ is continuous.
  \item If $R$ is a level-$(2n+1)$ tree with domain $\se{((0))}$, $R((0)) = (Q^{(2n)}_0,(d,q,P))$, $2k+1 \leq d < 2k+3$, then there is $\gamma_R$ such that $\cf^{\mathbb{L}_{\bolddelta{2n+3}}[T_{2n+2}]}(\gamma_R) = j^P(\bolddelta{2k+1})$ and
    \begin{displaymath}
      c_{\Gamma,\gamma_R} = \set{c_{\Gamma,\beta}} {\beta < \gamma_R}.
    \end{displaymath}
  \end{enumerate}
\item The following are equivalent:
\begin{enumerate}
\item $\Gamma$ is level $\leq 2n$ correct.
\item For any potential partial level $\leq 2n$ tower $(X,\overrightarrow{(e,x,W)})$ of continuous type, if $F \in ((\bolddelta{2i+1})_{i \leq n})^{(X, \overrightarrow{(e,x,W)}) \uparrow}$, then
  \begin{displaymath}
    c_{\Gamma,X,[F]_{\mu^X}} = [  \vec{\alpha} \mapsto c_{\Gamma,F(\vec{\alpha})}  ]_{\mu^X}.
  \end{displaymath}
\item For any potential partial level $\leq 2n$ tower $(X,\overrightarrow{(e,x,W)})$ of continuous type, there exists $F \in ((\bolddelta{2i+1})_{i \leq n})^{(X, \overrightarrow{(e,x,W)}) \uparrow}$ satisfying
  \begin{displaymath}
    c_{\Gamma,X,[F]_{\mu^X}} = [  \vec{\alpha} \mapsto c_{\Gamma,F(\vec{\alpha})}  ]_{\mu^X}.
  \end{displaymath}
\end{enumerate}
\end{enumerate}

$0^{(2n+1)\#}$ is the unique iterable, remarkable, level $\leq 2n$ correct level-($2n+1$) EM blueprint. $c^{(2n+1)}_{Q,\gamma}$, $c^{(2n+1)}_{\gamma}$ are defined as usual. $I^{(2n+3)}$ is the set of level-($2n+1$) indiscernibles for $M_{2n,\infty}^{-}$.  $0^{(2n+1)\#}$ contains the \emph{universality of level $\leq 2n$ ultrapowers axiom}: 
\begin{quote}
    If $\alpha$ is a limit ordinal and $\xi > \alpha$ is a cardinal and cutpoint, then $V^{\coll(\omega,\xi)}$ satisfies $\pi_{K | \xi, \infty} (\alpha) = \sup \{ \pi_{(\underline{j^T})^K (K | \xi),\infty} (\beta): T $ is $\Pi^1_{2n}$-wellfounded, $\beta < (\underline{j^T})^K (\alpha)\}$. 
\end{quote}

If $\mathcal{N}$ is a structure that satisfies Axioms~\ref{item:EM_ZFC_n}-\ref{item:level-2-embedding_invariance_n} in Definition~\ref{def:EM} and the universality of level $\leq 2n$ ultrapowers axiom, then
  \begin{displaymath}
    \mathcal{G}_{\mathcal{N}}
  \end{displaymath}
is the direct system consisting of models $\mathcal{N}^T$ for which $T$ is a $\Pi^1_{2n}$-wellfounded level $\leq 2n$ tree and maps $\pi^{T,T'}_{\mathcal{N}} : \mathcal{N}^T \to \mathcal{N}^{T'}$ for $\pi$ minimally factoring $T,T'$. Define
\begin{align*}
  \mathcal{N}_{\infty}& = \dirlim \mathcal{G}_{\mathcal{N}},\\
\pi_{\mathcal{N}^T,\mathcal{N}_{\infty}}&: \mathcal{N}^T \to \mathcal{N}_{\infty} \text{ is tail of the direct limit map.}
\end{align*}
If in addition, $\mathcal{N}$ is countable $\Pi^1_{2n+1}$-iterable mouse, then  $\mathcal{G}_{\mathcal{N}}$ is a dense subsystem of $\mathcal{I}_{\mathcal{N}}$, so there is no ambiguity in the notation $\mathcal{N}_{\infty}$.
If $Q$ is a finite level $\leq 2n$ tree, $a \in{\mathcal{N}}$, $\vec{\beta} = (\comp{d}{\beta}_x)_{(d,x) \in \dom(Q)}$ is represented by both $T$ and $T'$, then $\pi_{\mathcal{N}^T,\infty}  \circ j^{Q,T}_{\mathcal{N}}(a) =\pi_{\mathcal{N}^{T'},\infty}  \circ j^{Q,T'}_{\mathcal{N}}(a) $.  We can define 
\begin{displaymath}
  \pi_{\mathcal{N}, Q, \vec{\beta}, \infty} (a) = \pi_{\mathcal{N}^T, \infty} \circ j^{Q,T}_{\mathcal{N}} (a)
\end{displaymath}
for $\vec{\beta}$ represented by $T$. So
\begin{displaymath}
  \mathcal{N}_{\infty} = \set{\pi_{\mathcal{N},Q,\vec{\beta},\infty}(a)}{a \in \mathcal{N},Q \text{ finite level $\leq 2n$ tree}, \vec{\beta} \text{ respects } Q} .
\end{displaymath}

If $\vec{\gamma}$ respects a level-($2n+1$) tree $R$, define
\begin{displaymath}
  c_{\vec{\gamma}} = (c_{R_{\tree}(r),\gamma_r}^{(2n+1)})_{r \in \dom(R)}
\end{displaymath}
which strongly respects $R$. The general remarkability of $0^{(2n+1)\#}$ is as follows: 

\begin{quote}
  Suppose $\vec{\gamma}$ and $\vec{\gamma}'$ both respect a finite level-($2n+1$) tree $R$. Suppose $r \in \dom(R)$ and for any $s \prec^R r$, $\gamma_s = \gamma'_s$. Then for any $\mathcal{L}$-Skolem term $\tau$,
  \begin{displaymath}
    M_{2n,\infty}^{-} \models \tau(c_{\vec{\gamma}}) < c_{R_{\tree}(r), \gamma_r}^{(2n+1)} \to \tau(c_{\vec{\gamma}}) = \tau(c_{\vec{\gamma}'}).
  \end{displaymath}
\end{quote}

For any $c_{\omega}^{(2n+1)} <\xi \in I^{(2n+1)}$, there is an $\mathcal{L}$-Skolem term $\tau$ such that 
 $M_{2n,\infty}^{-}(0^{(2n+1)\#}) \models `` \tau(\sup(I^{(2n+1)} \cap \xi) , \cdot)$ is a surjection from $ \sup (I^{(2n+1)} \cap \xi)$ onto $\xi$''. For any $u^{(2n-1)}_{E(2n-1)} < \xi < c^{(2n+1)}_{\omega}$, there  is an $\mathcal{L}$-Skolem term $\tau$ such that 
 $M_{2n,\infty}^{-}(0^{(2n+1)\#}) \models `` \tau( u^{(2n-1)}_{E(2n-1)} , \cdot)$ is a surjection from $ u^{(2n-1)}_{E(2n-1)} $ onto $\xi$''. 
For notational convenience, if $X$ is a finite level $\leq 2n$ tree and $\gamma = [F]_{\mu^X}$ is a limit ordinal, define $ c^{(2n+1)} _{X, \gamma} = [\vec{\alpha} \mapsto c^{(2n+1)}_{F(\vec{\alpha})}]_{\mu^X}$; define $c^{(2n+1)}_{\emptyset, \bolddelta{2n+1}} = \bolddelta{2n+1}$. Ordinals of the form  $c^{(2n+1)}_{X, \gamma}$ when $ X \neq \emptyset$  are definable from elements in $ I^{(2n+1)}$ over $M_{2,\infty}^{-}$.
Define $\bar{I}^{(2n+1)}=$the closure of $I^{(2n+1)}$ under the order topology.
 Every ordinal in $\bar{I}^{(2n+1)}$ is of the form $c^{(2n+1)}_{X, \gamma}$ where $X$ is finite and $\gamma < \bolddelta{2n+1}$ is a limit. 
Thus, if $\mathbf{A} = (\mathbf{r}, \pi,T) \in \exexdesc(R)$ and $\vec{\gamma} $ strongly respects $R$, then $c^{(2n+1)}_{T, \gamma_{\mathbf{A}}} \in \bar{I}^{(2n+1)}$ and is a limit point of $I^{(2n+1)}$.

We define the level $\leq 2n$ indiscernibles below $u^{(2n-1)}_{E(2n-1)}$ as in the last section of \cite{sharpIII}. This leads to a closed-below-$u^{(2n-1)}_{E(2n-1)}$ set $I^{(\leq 2n)} \subseteq u^{(2n-1)}_{E(2n-1)}$, enumerated as $b^{(\leq 2n)}_{\beta}$ for $\beta <^{(2n-1)}_{E(2n-1)} $ in the increasing order. We get clubs $\vec{C} \in \prod_{i \leq n} \nu_i$ such that if $P$ is a finite level $\leq 2n-1$ tree and $\beta = [f]_{\mu^P} < u^{(2n-1)}_{E(2n-1)}$, then $b^{(\leq 2n)}_{\beta} = [\vec{\alpha} \mapsto b^{(\leq 2n)}_{f(\alpha)} ]_{\mu^P}$. As in the last section of \cite{sharpIII}, for any $\min(I^{\leq 2n}) <\xi \in I^{(\leq 2n)}$, there is an $\mathcal{L}$-Skolem term $\tau$ such that 
 $M_{2n,\infty}^{-}(0^{(2n+1)\#}) \models `` \tau(\sup(I^{(\leq 2n)} \cap \xi) , \cdot)$ is a surjection from $ \sup (I^{(\leq 2n)} \cap \xi)$ onto $\xi$''; for any $\xi < \min(I^{\leq 2n})$, there is an $\mathcal{L}$-Skolem term $\tau$ such that 
 $M_{2n,\infty}^{-}(0^{(2n+1)\#}) \models `` \tau( \cdot)$ is a surjection from $ \omega$ onto $\xi$''.

All the above notions relativize to an arbitrary real. The relevant notations are $c^{(2n+1)}_{x,\gamma}$, $c^{(2n+1)}_{x,Q, \gamma}$, $I^{(2n+1)}_x$, etc. 
If $x \in \mathbb{R}$, a \emph{level $\leq 2n+1$ code for an ordinal in $\bolddelta{2n+1}$ relative to $x$} is of the form
 \begin{displaymath}
  (x,R, \vec{\gamma}, Q, \vec{\beta}, % (\gcode{ \tau_{(d,x)}})_{(d,x) \in \dom(X)} 
  \gcode{\sigma})
\end{displaymath}
such that $R$ is a finite level-($2n+1$) tree, $\vec{\gamma}$ respects $R$, $Q$ is a finite level $\leq 2n$ tree, $\vec{\beta}$
respects $Q$, and  $\sigma$ is an $\mathcal{L}^{\underline{x}, R}$-Skolem term for an ordinal. It codes the ordinal
\begin{displaymath}
      \sharpcode{ (x,R, \vec{\gamma}, Q, \vec{\beta}, \gcode{\sigma})}  =
  \pi_{\mathcal{M}^{*}_{x^{(2n+1)\#},R}, Q, \vec{\beta}, \infty}    
(
 \sigma^{\mathcal{M}^{*}_{x^{(2n+1)\#},R}} ( (\underline{c_r})_{r \in \dom(R)})).
\end{displaymath}
For any $x$, every ordinal in $\bolddelta{2n+1}$ has a level $\leq 2n+1$ code relative to $x$.

\subsection{Level-($2n+1$) uniform indiscernibles}
\label{sec:level-2n+1-indisc}

Suppose $R$ is a finite level-($2n+1$) tree,  $\tau$ is an $\mathcal{L}$-Skolem term, $\vec{\gamma} $ strongly respects $R$. Suppose $\mathbf{A} = (\mathbf{r}, \pi, T)\in \exexdesc(R)$. Then
\begin{displaymath}
        \tau^{M_{2,\infty}^{-}} (c_{\vec{\gamma}}) <  c^{(2n+1)}_{T, \gamma_{\mathbf{A}}} \to  \tau^{M_{2,\infty}^{-}} (c^{(2n+1)}_{\vec{\gamma}}) < \min (  I^{(2n+1)}  \setminus \sup
\set{c^{(2n+1)}_{T', \gamma_{\mathbf{A}'}}}{\mathbf{A}' \prec^R_{*} \mathbf{A}}).
\end{displaymath}
  Suppose $R$ is a finite level-($2n+1$) tree and $\mathbf{A}  = (\mathbf{r}, \pi, T)\in \exexdesc(R)$, $\mathbf{r} \neq \emptyset$, $\vec{\gamma}$ strongly respects $R$. Then $c^{(2n+1)}_{T,\gamma_{\mathbf{A}}}$ is a cardinal in $M_{2n,\infty}^{-}$.

The theory of level-($2n+1$) sharps imply (\ref{item:muP_is_measure}:$n$). (\ref{item:odd_wellfounded}:$n$) follows from (\ref{item:muP_is_measure}:$n$).  (\ref{item:PW_factor}:$n$) follows from (\ref{item:respecting_R}:$n-1$)(\ref{item:YT_desc_order}:$n-1$)(\ref{item:YT_desc_target_extension}$n-1$) and \Los{}. (\ref{item:WZ_tensor_product}:$n$) follows from the theory of level-($2n+1$) sharps.

% If $R$ is a finite level $\leq 2n+1$ tree, then the set of $\mathbb{L}[j^R(T_{2n+1})]$-cardinals is the closure of
% \begin{displaymath}
% \set{\seed^R_{\mathbf{A}}}{\mathbf{A} \in \exexdesc(R)},
% \end{displaymath}
% and $\seed^R_{\mathbf{A}} = u^{(2n+1)}_{ \wocode{\mathbf{A}}_{\prec^R_{*}}+1}$ for $\mathbf{A} \in \exexdesc(R)$.

Recall the lemma in \cite{sharpIII} on the well-definedness of $\gamma_{\mathbf{r}}$ for $\mathbf{r} $ an $R$-description when $R$ is a level-3 tree. 
\begin{lemma}
  \label{lem:gamma_r_continuous_type}
  Suppose $R$ is a level-3 tree, $\vec{\gamma}  \in [\bolddelta{3}]^{R \uparrow}$, $\mathbf{r} = (r, Q, \overrightarrow{(d,q, P)}) \in \desc(R)$ is of continuous type. Then $\gamma_{\mathbf{r}} = j_{\sup}^{R_{\tree}(r^{-}), Q} (\gamma_{r^{-}})$. % Suppose $\mathbf{A} = (\mathbf{r}, \pi, T) \in \exexdesc(R)$. Then $\gamma_{\mathbf{A}} = (\pi\res R_{\tree}(r^{-}))^T_{\sup} (\gamma_{r^{-}})$.
\end{lemma}
It has an obvious generalization to the higher levels. 

\begin{lemma}
  \label{lem:regular_cardinals}
  Suppose $R$ is a finite level-($2n+1$) tree, $\mathbf{A} \in \exexdesc(R)$, $\mathbf{A} \neq (\emptyset,\emptyset,\emptyset)$. Then $\cf^{\mathbb{L}[j^R(T_{2n+1})]}( \seed^R_{\mathbf{A}}) = \seed^R_{\ucf(\mathbf{A})}$.
\end{lemma}
\begin{proof}
  By the higher analog of proof of Lemma~\ref{lem:gamma_r_continuous_type},  for any real $x$, $x^{3\#}(R)$ contains the axiom ``$\cf(\underline{c_{\mathbf{A}}}) = \cf( \underline{c_{\ucf(\mathbf{A})}})$''. So $\cf^{\mathbb{L}[j^R(T_{2n+1})]}( \seed^R_{\mathbf{A}}) = \cf^{\mathbb{L}[j^R(T_{2n+1})]}(\seed^R_{\ucf(\mathbf{A})})$. So the $\mathbb{L}[j^R(T_{2n+1})]$-regular cardinals in the interval $[\bolddelta{2n+1}, j^R(\bolddelta{2n+1})]$ is a subset of $\{\seed^R_{\mathbf{r}}:\mathbf{r} \in \exdesc(R)$ is regular$\}$.  It remains to show that for any $\mathbf{r} = (r, Q, \overrightarrow{(d,q,P)})\in \exdesc(R)$ regular, $\seed^R_{\mathbf{r}}$ is regular in $\mathbb{L}[j^R(T_{2n+1})]$. Suppose towards a contraction that $\seed^R_{\mathbf{r}}$ is the least counterexample and for some real $x$, 
$\cf^{M^{-}_{2n,\infty}(x)}(\seed^R_{\mathbf{r}})  =  \delta < \seed^R_{\mathbf{r}}$ and $\delta$ is $\mathbb{L}[j^R(T_{2n+1})]$-regular.

Case 1: $\delta = \seed^R_{\mathbf{r}'}$, $\mathbf{r}' \prec^R \mathbf{r}$. 

So $x^{(2n+1)\#}(R)$ contains the axiom ``$\cf(\underline{c_{\mathbf{r}}}) = \underline{c_{\mathbf{r}'}}$''.  However, we can easily construct a one-node extension $S$ of $R$ such that $s \in \dom(S) \setminus \dom(R)$, $s \prec^S{\mathbf{r}}$ and $\mathbf{A}\prec^S s$ for any $\mathbf{A} \in \exexdesc(R)$. Let $\tau$ be an $\mathcal{L}^R$-Skolem term such that $x^{(2n+1)\#}(R)$ contains the formula ``$\tau(\cdot)$ is a cofinal map from $\underline{c_{\mathbf{r}'}}$ to $\underline{c_{\mathbf{r}}}$''. By general remarkability, $x^{(2n+1)\#}(S)$ contains the formula ``$\tau'' \underline{c_{\mathbf{r}'}} \subseteq \underline{c_s}$''. This is a contradiction. 

Case 2: $\delta < \bolddelta{2n+1}$. 

By (\ref{item:subset_does_not_increase_0}:$k$)(\ref{item:muP_is_measure}:$k$) for $k \leq n$, (\ref{item:subset_does_not_increase}:$k$)(\ref{item:regular_cardinals_Q}:$k$) for $k\leq n-1$, either $\delta=\omega$ or $\delta = j^P(\bolddelta{2k+1})$ satisfying $k < n$, $P$ is a level-($2k+1$) tree, $\card(P) \leq 1$. $\delta$ is definable in $M_{2n}^{\#}(x)$, allowing us to proceed with the proof of Case 1.
\end{proof}

Lemma~\ref{lem:regular_cardinals} implies~(\ref{item:cardinals}:$n$). The proof of~(\ref{item:sharp_code_complexity}:$n$) is basically in \cite{sharpIII}.

We prove (\ref{item:signature_Q}:$n$). Recall that a level-($2n+1$) sharp code is based on the tree $R^{(2n+1)}_{E(2n+1)}$.  One can define the ``$R^{(2n+1)}_{E(2n+1)}$-signature'', ``$R^{(2n+1)}_{E(2n+1)}$-approximation sequence'', etc.\ in a completely parallel way. Appealing to (\ref{item:cardinals}:$n$) for uniform cofinality when necessary, the $R^{(2n+1)}_{E(2n+1)}$-partial level $\leq 2n+1$ tower induced by $\beta$ and the $R^{(2n+1)}_{E(2n+1)}$ approximation sequence of $\beta$ are uniformly $\Delta^1_{2n+3}$ definable from $\beta$. The issue is reduce $R^{(2n+1)}_{E(2n+1)}$ to finite trees. 
If $\beta = [f]_{\mu^W}$, $W$ is a finite level-$(2n+1)$ tree, and $f(\vec{\alpha}) = \tau^{M_{2n,\infty}^{-}(x)}(x, \vec{\alpha})$ for $\mu^W$-a.e.\ $\vec{\alpha}$, $\rho$ factors $(W, R^{(2n+1)}_{E(2n+1)})$, then $\rho^{R^{(2n+1)}_{E(2n+1)}}(\beta) = \sharpcode{\corner{\gcode{\tau(\underline{x}, (\underline{c_{\rho(w)}})_{w \in \dom(W)})}, x^{(2n+1)\#}}}$. 
It suffices to show that $\rho^{R^{(2n+1)}_{E(2n+1)}} \res j^W(\bolddelta{2n+1})$  is $\Delta^1_{2n+3}$, uniformly in $(W, \rho)$. By looking at the $R^{(2n+1)}_{E(2n+1)}$-signature, $(\rho^{R^{(2n+1)}_{E(2n+1)}})''  j^W(\bolddelta{2n+1})$ is uniformly $\Delta^1_{2n+3}$. Hence,
\begin{displaymath}
\rho^{E^{(2n+1)}}_{E(2n+1)}(\sharpcode{\corner{\gcode{\sigma}, y^{(2n+1)\#}}}) = \sharpcode{\corner{\gcode{\chi}, z^{(2n+1)\#}}}
\end{displaymath}
iff for some (or equivalently, for any) countable transitive model $M$ of ZFC containing $y,z$ such that $\mathbb{R}^M$ is closed under $M_{2n}^{\#}$, $M$ satisfies that $\rho^{R^{(2n+1)}_{E(2n+1)}}(\sharpcode{\corner{\gcode{\sigma}, y^{(2n+1)\#}}}) = \sharpcode{\corner{\gcode{\chi}, z^{(2n+1)\#}}}$. This is a  $\Delta^1_{2n+3}$ definition of $\rho^{R^{(2n+1)}_{E(2n+1)}} \res j^W(\bolddelta{2n+1})$. 

The level-($2n+2$) Martin-Solovay tree $T_{2n+2}$ is defined as follows. Let $T^{(2n+1)}$ be 
 a recursive tree so that $z \in [T]$ iff $z$ is a remarkable level-($2n+1$) EM blueprint over a real. Let $(r_i)_{1 \leq i < \omega}$ be an effective enumeration of $\dom(R^{(2n+1)}_{E(2n+1)})$ and let $(\tau_{k})_{k<\omega}$ be an effective enumeration of all the $\mathcal{L}$-Skolem terms for an ordinal, where $\tau_k$ is ($f(k)+1$)-ary.   
 $T_{2n+2}$ is the tree on $2 \times u^{(2n+1)}_{E(2n+1)}$ where
\begin{displaymath}
  (t,\vec{\alpha}) \in T_{2n+2}
\end{displaymath}
iff $t \in T^{(2n+1)}$ and 
\begin{enumerate}
\item if $\xi \leq \eta < E(2n+1)$, $r_1,\dots,r_{f(k)} \in \dom(R^{(2n+1)} _{\xi})$, $r_1,\dots,r_{f(l)} \in \dom(R^{E(2n+1)}_{\eta})$,  $\rho$ factors $(R^{(2n+1)} _{\xi}, R^{(2n+1)} _{\eta})$,
  \begin{enumerate}
  \item if
``$\tau_k ( \underline{x},\underline{c_{\rho(r_1)}},\dots, \underline{c_{\rho(r_{f(k)})}}) = \tau_l (\underline{x}, \underline{c_{r_1}},\dots, \underline{c_{r_{f(l)}}})$'' is true in $t$,  then $\rho^{R^{(2n+1)}_{\eta}}(\alpha_k) = \alpha_l$;
  \item if
``$\tau_k ( \underline{x},\underline{c_{\rho(r_1)}},\dots, \underline{c_{\rho(r_{f(k)})}}) < \tau_l (\underline{x}, \underline{c_{r_1}},\dots, \underline{c_{r_{f(l)}}})$'' is true in $t$,  then $\rho^{R^{(2n+1)}_{\eta}}(\alpha_k) < \alpha_l$;
\end{enumerate}
\item if $\xi < E(2n+1)$, $r_1,\dots,r_{\max(f(k),f(l))} \in \dom(R^{(2n+1)}_{\xi})$, 
$Q,Q'$ are finite level $\leq 2n+1$ trees, $Q$ is a subtree of $Q'$, ``$\underline{j^{Q,Q'}}(\tau_{k}(\underline{x},\underline{c_{r_1}},\dots, \underline{c_{r_{f(k)}}})) = \tau_{l}(\underline{x}, \underline{c_{r_1}},\dots, \underline{c_{r_{f(l)}}})$'' is true in $t$, then $j^{R^{(2n+1)}_{\xi} \otimes Q, R^{(2n+1)}_{\xi} \otimes Q'} (\alpha_k)  = \alpha_l$. 
\end{enumerate}

We have:
\begin{itemize}
\item $p[T_{2n+2}] = \set{x^{(2n+1)\#}}{x \in \mathbb{R}}$.
\item $T_{2n+2}$ is $\Delta^1_{2n+3}$ in the level-($2n+1$) sharp codes.
\item For any $x \in \mathbb{R}$, $(T_{2n+2})_{x^{(2n+1)\#}}$ has an honest leftmost branch
  \begin{displaymath}
(\tau_{k}^{(j^{R^{(2n+1)}_{E(2n+1)}}(M_{2n,\infty}^{-}(x)))}(x, \seed^{R^{(2n+1)}_{E(2n+1)}}_{r_1},\dots,\seed^{R^{(2n+1)}_{E(2n+1)}}_{r_{f(k)}}))_{k<\omega}.
\end{displaymath}

\end{itemize}

By analyzing the representative functions of ordinals below $u^{(2n+1)}_{E(2n+1)}$, we conclude the following in parallel to the level-1 scenario:
\begin{lemma}\label{lem:W_signature}
 Suppose $\bolddelta{2n+1} \leq \beta < j^W(\bolddelta{2n+1})$ is a limit ordinal. Then:
  \begin{enumerate}
  \item The $W$-uniform cofinality of $\beta$ is equal to $\ucf^{\mathbb{L}[j^W(T_{2n+1})]}(\beta)$.
  \item Suppose the $W$-signature of $\beta$ is ${(d_i,\mathbf{w}_i)_{i < k}}$, the $W$-approximation sequence of $\beta$ is $(\gamma_i)_{i \leq k}$, the $W$-partial level $\leq 2n+1$ tower induced by $\beta$ is $(P_i, \dots)_i$. Then:
    \begin{enumerate}
    \item For $i < l < k$, $j_{\sup}^{P_i, P_l} (\gamma_i) < \gamma_l < j^{P_i,P_l}(\gamma_i)$.
    \item For $i < k$, $(\sigma \res \dom(P_i))^W_{\sup} (\gamma_i) \leq \gamma_k <(\sigma \res \dom(P_i))^W(\gamma_i)$.
    \item For $i < k$, $(\sigma \res \dom(P_i))^W_{\sup} (\gamma_i) = \gamma_k$ iff $i=k-1$ and $\beta$ is $W$-essentially continuous.
    \item $\beta = \sigma^W (\gamma_k)$.
    \end{enumerate}
  \end{enumerate}
\end{lemma}

\begin{lemma}\label{lem:W_signature_reverse}
 Suppose $\pi$ factors finite level $\leq 2n+1$ trees $(W,W')$. Suppose $\gamma < j^P(\bolddelta{2n+1})$ and $\pi^{W'}_{\sup}(\gamma) < \gamma' < \pi^{W'}(\gamma)$. Let $({(d_k, {w}_k)})_{k<v}$, $(\gamma_k)_{k\leq v}$, $(P,\overrightarrow{(d,p,R)})$ be the $W$-signature, $W$-approximation sequence and $W$-potential partial level $\leq 2n+1$ tower induced by $\gamma$ respectively. Let $({(d_k', {w}_k')})_{k<v'}$, $(\gamma_k')_{k\leq v'}$, $(P',\overrightarrow{(d',p',R')})$ be the $W'$-signature, $W'$-approximation sequence and $W'$-potential partial level $\leq 2n+1$ tower induced by $\gamma'$ respectively.  Let $\cf^{\mathbb{L}[j^W(T_{2n+1})]}(\gamma) = \seed^W_{(d_{*},\mathbf{w}_{*})}$.  Then
  \begin{enumerate}
  \item $v < v'$, $\pi(d_k,{w}_k) = (d_k',{w}_k')$ and $\gamma_k = \gamma_k'$ for any $k < v$. $\gamma$ is essentially discontinuous $\to \gamma_v = \gamma_{v'}$.  $\gamma$ is essentially continuous$\to\gamma_v < \gamma_{v'}$.
  \item $l'_{k} \notin \ran( \pi)$ for $v\leq k < v'$.
  \item $P$ is a proper subtree of $P'$ and $\vec{p}$ is an initial segment of $\vec{p}'$.
  \end{enumerate}
  Moreover, if $ \gamma' < \gamma'' < \pi^{W'}(\gamma)$ and $(\gamma_k'')_{k \leq v''}$ is the $W'$-approximation sequence of $\gamma''$, then $\gamma_{v}'<\gamma_{v}''$.
\end{lemma}
    
\begin{lemma}
  \label{lem:continuity}
  Suppose $P,W$ are finite level $\leq 2n+1$ trees and $\sigma$ factors $(P,W)$. Suppose $\gamma < j^P(\bolddelta{2n+1})$ and $\cf^{\mathbb{L}[j^P(T_{2n+1})]}(\gamma) = \seed^P_{(d,\mathbf{p})}$, $(d,\mathbf{p}) \in \exdesc(P)$ is regular.  Then
  \begin{enumerate}
  \item $\sigma^W$ is continuous at $\gamma$ iff $(\sigma,W)$ is continuous at $(d,\mathbf{p})$.
  \item Suppose $(\sigma,W)$ is discontinuous at $(d,\mathbf{p})$. Let $(P^{+}, \sigma^{+})$ be the $W$-decomposition of $\pi$. Then $\sigma^W_{\sup}(\gamma) = (\sigma^{+})^W \circ j^{P,P^{+}}_{\sup} (\gamma)$.
  \end{enumerate}
\end{lemma}

\begin{lemma}
  \label{lem:discontinuity_1}
  Suppose $(P^{-},{(d,p,R)})$ is a partial level $\leq 2n+1$ tree and $P$ is a completion of $(P^{-}, {(d,p,R)})$. Suppose $W$ is a level $\leq 2n+1$ tree and $\sigma,\sigma'$ both factor $(P,W)$, $\sigma$ and $\sigma'$ agree on $\dom(P^{-})$, $\sigma'(d,p) = \pred(\sigma, W, (d,p))$. Then for any $\beta < j^{P^{-}}(\bolddelta{2n+1})$ such that $\cf^{\mathbb{L}[j^{P^{-}}(T_{2n+1})]}(\beta) = \seed^{P^{-}}_{\ucf(P^{-},(d,p,R))}$, we have
  \begin{displaymath}
    \sigma^W \circ j^{P^{-},P}_{\sup} (\beta) = (\sigma')^W_{\sup} \circ j^{P^{-},P} (\beta).
  \end{displaymath}
\end{lemma}

\begin{lemma}
 \label{lem:discontinuity_2} Suppose $(P,(d,p,R))$ is a partial level $\leq 2n+1$ tree, $\ucf(P,(d,p,R)) = (d^{*}, \mathbf{p}^{*})$ and $\sigma$ factors $(P,W)$. Suppose $\beta < j^P(\bolddelta{2n+1})$ and either
  \begin{enumerate}
  \item $d=0$, $P^{+} = P$, $\sigma' = \sigma$, $\cf^{\mathbb{L}[j^P(T_{2n+1})]}(\beta) = \omega$, or
  \item $d > 0$, $P^{+}$ is a completion of $P$, $\sigma'$ factors $(P^{+}, W)$, $\sigma = \sigma' \res \dom(P)$, $\sigma'(d,p) = \pred(\sigma, T, (d^{*},\mathbf{p}^{*})) $, $\cf^{\mathbb{L}[j^P(T_{2n+1})]}(\beta) = \seed^P_{(d^{*}, \mathbf{p}^{*})}$.
  \end{enumerate}
  Then
  \begin{displaymath}
    \sigma^W(\beta) = (\sigma')^W_{\sup} \circ j^{P,P^{+}}(\beta).
  \end{displaymath}
\end{lemma}

(\ref{item:rep_Q}:$n$) follows from (\ref{item:odd_wellfounded}:$n$). 
The proofs of (\ref{item:Q_ordinal_cont_type}:$n$)(\ref{item:q_respecting}:$n$)  generalize their level-1 versions, appealing to the general remarkability of level-($2n+1$) sharps when necessary.

 % (\ref{item:q_respecting}:$n$) is proved by generalizing the arguments of Lemmas 3.15-3.18 in \cite{sharpI}, using (\ref{item:cardinals}:$n-1$) and remarkability of level-$(2n+1)$ sharps. 

\section{The level-($2n+2$) sharp}
\label{sec:level-2n+2-sharp}

From now on, we assume $\boldDelta{2n+2}$-determinacy. We will prove (\ref{item:bk_km}:$n$)-(\ref{item:represent_another}:$n$).

For $x \in \mathbb{R}$, a putative $x$-$(2n+1)$-sharp is a remarkable, level $\leq 2n$ correct level-($2n+1$) EM blueprint over $x$ that satisfies the universality of level $\leq 2n$ ultrapowers axiom. Suppose $x^{*}$ is a putative $x$-3-sharp.  
For any limit ordinal $\alpha < \bolddelta{2n+1}$, we can build an EM model
\begin{displaymath}
\mathcal{M}^{*}_{x^{*},\alpha}
\end{displaymath}
as follows. Let $R$ be a level-($2n+1$) tree such that $\llbracket \emptyset \rrbracket_R = \alpha$. Then $\mathcal{M}^{*}_{x^{*},\alpha} = (\mathcal{M}^{*}_{x^{*},R})_{\infty}$. This definition is independent of the choice of $R$. We say that $x^{*}$ is \emph{$\alpha$-iterable} iff $\alpha$ is in the wellfounded part of $\mathcal{M}^{*}_{x^{*},\alpha}$. 

A \emph{putative level-(2n+1) sharp code for an increasing function} is $w = \corner{\gcode{\tau}, x^{*}}$ such that $x^{*}$ is a putative $x$-(2n+1)-sharp, $\tau$ is a unary $\mathcal{L}^{\underline{x}}$-Skolem term and
\begin{displaymath}
``\forall v,v'( (v , v'\in \ord \wedge v < v' )\to  (\tau(v) \in \ord  \wedge \tau(v) < \tau(v')))\text{''}
\end{displaymath}
is true in $x^{*}(\emptyset)$.  The statement `` $\corner{\gcode{\tau}, x^{*}}$ is a putative level-(2n+1) sharp code for an increasing function, $x^{*}$ is $\alpha$-iterable, $r$ codes the order type of $\tau^{\mathcal{M}_{x^{*}, \alpha}}(\alpha)$'' about  $(\corner{\gcode{\tau}, x^{*}},r)$  is $\boldsigma{2n+1}$ in the code of $\alpha$. 
In addition, when $x^{*} = x^{(2n+1)\#}$, $\corner{\gcode{\tau}, x^{*}}$ is called a \emph{(true) level-(2n+1) sharp code for an increasing function}.

The proof of (\ref{item:bk_km}:$n$)-(\ref{item:evensharp}:$n$) is basically a copy of the arguments in \cite{sharpIII}.

By (\ref{item:bk_km}:$n$), every subset of $u^{(2n+1)}_{E(2n+1)}$ in $\mathbb{L}_{\bolddelta{2n+3}}[T_{2n+2}]$ is $\boldDelta{2n+3}$. 
We use this and Moschovakis Coding Lemma \cite{mos_dst} to prove 
(\ref{item:subset_does_not_increase}:$n$). 
  Suppose  $A \subseteq \bolddelta{2n+1}$ is in $\mathbb{L}_{\bolddelta{2n+3}}[T_{2n+2}]$. Suppose $B, C \subseteq \mathbb{R}$ are $\Pi^1_{2n+2}(x)$ subsets of $\mathbb{R}^2$ such that ($w \in \WO^{(2n+1)}\wedge \sharpcode{w} \in A$) iff $\exists z ((w,z) \in B)$ iff $\neg \exists z ((w,z) \in C)$. By Moschovakis Coding Lemma, there is a real $y$ and a $\Sigma^1_{2n+2}(y)$ set $D \subseteq \mathbb{R}^2$ satisfying:
  \begin{itemize}
  \item If $(w,z) \in D$ then $w \in \WO^{(2n+1)}$ and there is $w' \in \WO^{(2n+1)}$ such that $\wocode{w} = \wocode{w'}$ and $(w',z) \in B \cup C$,
  \item If $w  \in \WO^{(2n+1)}$ then there is $w'  \in \WO^{(2n+1)}$ and $z$ such that $(w',z) \in B \cup C$. 
  \end{itemize}
Then $A$ is $\Sigma^1_{2n+2}(x,y)$ and hence $A \in L[T_{2n+1},x,y]$.
% \begin{displaymath}
% v \in \WO_n \wedge   \sharpcode{v} \in A \eqiv L[v,x] \models \exists w \in \WO_n 
% \end{displaymath}

% This game is definable over $L[M,N,y]$ uniformly from $\se{\gcode{\tau},M,\gcode{\sigma},N,y}$, hence 

% (\ref{item:subset_does_not_increase}:$n$) is a consequence of Lemma~\ref{lem:complexity} and (\ref{item:bk_km}:$n$). 

(\ref{item:strong_partition}:$n$) is a simple generalization of the level-2 partition property of $\omega_1$ in \cite{sharpII}. The idea of partially iterable level-$(2n-1)$ sharps is used in the proof. 
(\ref{item:Q_direct_limit_wf}:$n$) follows from (\ref{item:rep_Q}:$n$). (\ref{item:ultrapower_bound}:$n$) is a simple generalization of the $n=0$ case, using (\ref{item:q_respecting}:$n$) when necessary. 

We now prove (\ref{item:measure_analysis}:$n$). To save notations, we prove the case $n=1$. The statement is: 
\begin{quote}
  Suppose $R$ is a $\Pi^1_3$-wellfounded level-3 tree. Suppose $\mu$ is a nonprincipal $\mathbb{L}_{\bolddelta{5}}[T_4]$-measure on $j^R(\bolddelta{3})$. Then there
are functions $g,h \in \mathbb{L}[j^R(T_{3})]$, a finite level $\leq 4$ tree $Q$, nodes $(d_1,q_1),\dots,(d_k,q_k) \in \dom(Q)$, such that $h : j^P(\bolddelta{3}) \to ( j^P(\bolddelta{3}) )^k$, $g:(j^P(\bolddelta{3}))^k \to j^P(\bolddelta{3})$, $h$ is 1-1 a.e.\ ($\nu$), $g$ is 1-1 a.e.\ ($\mu^Q_{\overrightarrow{(d,q)}}$), $g = h^{-1}$ a.e.\ ($\nu$), and $A \in \mu^Q_{\overrightarrow{(d,q)}}$ iff $(h^{-1})''A \in \nu$.
\end{quote}
Take the restricted ultrapower
\begin{displaymath}
  j : \mathbb{L}[j^R(T_3)] \to \Ult (\mathbb{L}[j^R(T_3)], \mu) =  \mathbb{L}[j\circ j^R(T_3)].
\end{displaymath}
$\beta< j \circ j^R (\bolddelta{3})$ is called (only in this proof) a uniform indiscernible iff $\beta $ is represented by some $f$ in this ultrapower such that for any $x \in \mathbb{R}$, for $\nu$-a.e.\ $\xi$,  $f(\xi) \in j^R(I^{(\leq 3)}_x)$.
\begin{claim}
  \label{claim:uniform_indis_closed}
  The set of uniform indiscernibles is closed below $j \circ j^R (\bolddelta{3})$.
\end{claim}
\begin{proof}
  Suppose that $ \beta < j^R(\bolddelta{3})$ is not a uniform indiscernible and $\beta = [f]_{\nu}$. Pick $x$ such that for $\nu$-a.e.\ $\xi$, $f(\xi) \notin j^R(I^{(\leq 3)}_x)$. Let $g(\xi)=\max(f(\xi) \cap j^R(I^{(\leq 3)}_x))$. Then $[g]_{\nu} < \beta$ and  any $\gamma \in ([g]_{\nu}, \beta)$ is not a uniform indiscernible. So $\beta$ cannot be a limit of uniform indiscernibles. 
\end{proof}

\begin{claim}
  \label{claim:next_uniform_indis}
  If $[f]_{\nu} < j \circ j^R (\bolddelta{3})$ is not a uniform indiscernible, then there is a real $x$ such that putting $g(\xi)=\max(f(\xi) \cap j^R(I^{(\leq 3)}_x))$, either $[g]_{\nu} = 0$ or $[g]_{\nu}$ is a uniform indiscernible.
\end{claim}
\begin{proof}
  Suppose not. Pick $x_0$ such that for $\nu$-a.e.\ $\xi$, $f(\xi) \notin j^R(I^{(\leq 3)}_{x_0})$. Let $f_0(\xi)=\max(f(\xi) \cap j^R(I^{(\leq 3)}_{x_0}))$. Then $[f_0]_{\nu}$ is not a uniform indiscernible and $0< [f_0]_{\nu} < [f]_{\nu}$. Continuing this way, we obtain a descending chain of ordinals $[f]_{\nu} > [f_0]_{\nu} > [f_1]_{\nu} > \dots$. 
\end{proof}

\begin{claim}
  \label{claim:uniform_indis_generate}
  If $\alpha < j  \circ j^R(\bolddelta{3})$ then there is a real $z$, an $\mathcal{L}$-Skolem term $\tau$ and uniform indiscernibles $\beta_1,\dots,\beta_k \leq \alpha$ such that 
  \begin{displaymath}
    \alpha = \tau^{L[j \circ j^R(T_3), z]} (z, \beta_1,\dots,\beta_k)  .
  \end{displaymath}
\end{claim}
\begin{proof}
Suppose without loss of generality that $\alpha$ is not a uniform indiscernible.  We show 
  by induction on $\alpha$.  Let $\alpha = [f]_{\nu}$. If $\alpha<$the smallest uniform indiscernible, by Claim~\ref{claim:next_uniform_indis}, 
there is $x $ such that for $\nu$-a.e.\  $\xi$, $f(\xi) < \min( I^{(\leq 3)}_x)$. 
By the analysis of level $\leq 2$ indiscernibles, there is a term $\tau$ such that
  \begin{displaymath}
    \tau^{L[j^R(T_3),x^{3\#}]}  ( x^{3\#}, \cdot)
  \end{displaymath}
defines a surjection from $\omega$ onto $I^{(\leq 3)}_x$. So there is $l < \omega$ such that
\begin{displaymath}
  \alpha = \tau^{L[j \circ j^R(T_3), x^{3\#}]}(x^{3\#}, l)
\end{displaymath}
and we are done.

If $\beta$ is the largest uniform indiscernible below $\alpha$, by Claims~\ref{claim:uniform_indis_closed}-\ref{claim:next_uniform_indis}, there is $x$ and $g(\xi) = \max(f(\xi) \cap j^R(I^{(\leq 3)}_x))$ such that $[g]_{\nu}$ is the largest uniform indiscernible below $\alpha$. By the analysis of level $\leq 2$ indiscernibles and level-3 indiscernibles, there is a term $\tau$ such that 
  \begin{displaymath}
    \tau^{L[j^R(T_3),x^{3\#}]}  ( x^{3\#}, g(\xi),  \cdot)
  \end{displaymath}
defines a surjection from $g(\xi)$ onto $\min(j^R(I^{\leq 3}_x) \setminus f(\xi))$. So there is $\gamma < [g]_{\nu}$ such that
\begin{displaymath}
  \alpha = \tau^{L[j \circ j^R(T_3), x^{3\#}]}(x^{3\#}, g(\xi),\gamma).
\end{displaymath}
By induction, $\gamma$ can be represented as
\begin{displaymath}
  \gamma = \sigma^{L[j \circ j^R (T_3),z]} (z, \beta_1,\dots,\beta_k)
\end{displaymath}
for uniform indiscernibles $\beta_1,\dots,\beta_k \leq \gamma$.
Now combine the last two formulas together. 
\end{proof}

Let $\beta_1,\dots,\beta_k \leq [id]_{\nu}$ be uniform indiscernibles and $z, \tau$ be given by Claim~\ref{claim:uniform_indis_generate} such that
\begin{displaymath}
  [\id]_{\nu} = \tau^{L[j\circ j^R(T_3), z]}(z, \beta_1, \dots, \beta_k).
\end{displaymath}
Let $\beta_i = [f_i]_{\nu}$. Let 
$h: j^R(\bolddelta{3}) \to (j^R(\bolddelta{3}))^k$ be 
\begin{displaymath}
  h(\xi) = (f_1(\xi),\dots,f_k(\xi)). 
\end{displaymath}
Let $g(\gamma_1,\dots,\gamma_k) = \tau^{L[j^R(T_3),z]}(z, \gamma_1,\dots,\gamma_k)$. Clearly $g \circ h= \id$ a.e.\ ($\nu$).  
There is a unique finite level $\leq 4$ tree $Q$ and nodes $(d_1,q_1)$, $\dots$, $(d_k,q_k) \in \dom(Q)$ such that $\dom(Q)$ is the upward closure of $\se{(d_1,q_1),\dots,(d_k,q_k)}$ and for $\nu$-a.e. $\xi$, there is $\vec{\delta}$ respecting $Q$ such that $\delta_{(d_i,q_i)} = f_i(\xi)$. 
If $\mu^Q_{\overrightarrow{(d,q)}} (A) = 1$, take clubs $C \subseteq \bolddelta{3}$ and $E \subseteq \omega_1$, both in $\mathbb{L}[T_3]$ such that $[E,C]^{Q \uparrow}_{\overrightarrow{(d,q)}} \subseteq A$. Since $[f_i]$ is a uniform indiscernible, we have $f_i (\xi) \in (\cup_{P \text{ finite}} j^P(E)) \cup \se{u_{\omega}} \cup j^R(C)$ for $\nu$-a.e. $\xi$. Thus $(h^{-1})'' A \in \nu$. Thus $h \circ g = id$ a.e. ($\mu^Q_{\overrightarrow{(d,q)}}$) from \cite[Fact 3.4]{sol_delta13coding}. This finishes the proof of (\ref{item:measure_analysis}:$n$) for $n=1$.

We then use the proof of (\ref{item:measure_analysis}:$n$) to show (\ref{item:coding}:$n$). Again we assume $n=1$. A set $A \subseteq j^R(\bolddelta{3})$ is \emph{$R$-simple} iff there are clubs $E \subseteq \omega_1$, $C \subseteq \bolddelta{3}$, both in $\mathbb{L}[T_3]$, a finite level $\leq 4$ tree $Q$, nodes $(d_1,q_1)$,$\dots$,$(d_k,q_k)$ and an $F : (j^R(\bolddelta{3}))^m \to j^R(\bolddelta{3})$ such that 
\begin{enumerate}
\item $F$ is 1-1 on $[E,C]^{Q \uparrow}_{\overrightarrow{(d,q)}}$,
\item $A = F'' [E,C]^{Q \uparrow}_{\overrightarrow{(d,q)}}$,
\item $F \in \mathbb{L}[j^R(T_3)]$. 
\end{enumerate}
Every subset of $j^R(\bolddelta{3})$ in $\mathbb{L}_{\bolddelta{5}}[T_4]$ is $\boldDelta{5}$. Let $G \subseteq \mathbb{R}^2 $ be a universal $\boldpi{5}$ set. Apply the proof of \cite[Section 3.3]{sol_delta13coding} but change $h : \mathbb{R} \twoheadrightarrow \power(\lambda)$ in the proof of Lemma 3.7 to $h(x) = G_x$. We see that every $A \in \power(j^R(\bolddelta{3})) \cap \mathbb{L}_{\bolddelta{5}}[T_4]$ is a countable union of $R$-simple sets. Applying everything above to $R = R^{(3)}$, we get a $\Delta^1_5$ coding  of subsets of $u \ooo$ in $\mathbb{L}_{\bolddelta{5}}[T_4]$.

(\ref{item:QW_desc_order}:$n$) simply follows from definitions. (\ref{item:QW_desc_target_extension}:$n$) follows from (\ref{item:YT_desc_origin_extension}:$n-1$) and (\ref{item:respecting_R}:$k$) for $k < n$, (\ref{item:q_respecting}:$k$) for $k \leq n$. (\ref{item:QW_desc_origin_extension}:$n$) follows from Lemmas~\ref{lem:discontinuity_1}, \ref{lem:discontinuity_2} and (\ref{item:respecting_R}:$k$) for $k < n$, (\ref{item:q_respecting}:$k$) for $k \leq n$.
(\ref{item:SQ_factor}:$n$) follows from \Los{}, (\ref{item:QW_desc_order}:$n$)(\ref{item:QW_desc_target_extension}:$n-1$) and (\ref{item:respecting_R}:$k$) for $k < n$, (\ref{item:q_respecting}:$k$) for $k \leq n$.

The proof of (\ref{item:tensor_product_QW}:$n$) generalizes the $n=0$ case in \cite{sharpIII}. We explain some of the details for the $n>0$ case. Parts (b)(c) follow from part (a) and \Los{}. We prove part (a). Let $Q,W$ be as given. By (\ref{item:tensor_product_TQ}:$n-1$), $(\id_{\comp{\leq 2n-2}{Q} \otimes  \comp{\leq 2n-2}{W}})^{\comp{\leq 2n-2}{Q}, \comp{\leq 2n-2}{W}}$ is the identity on $\bolddelta{2n+1}$ and agrees with $(\id_{Q \otimes W})^{Q,W}$. By (\ref{item:muP_is_measure}:$n$), the set of $\mathbb{L}[j^{Q \otimes W}(T_{2n+1})]$-cardinals in the interval $[\bolddelta{2n+1}, j^{Q \otimes W}(\bolddelta{2n+1})]$ is the closure of $\set{\seed^{ (Q \otimes W)}_{(2n+1,\mathbf{A})}}{\mathbf{A} \in \exexdesc( \comp{2n+1}{Q \otimes W})}$. We prove by induction on $\wocode{\mathbf{A}}_{\prec^{Q \otimes W}_{*}}$ that $(\id_{Q \otimes W})^{Q,W} \res (\seed^{Q \otimes W}_{(2n+1,\mathbf{A})}+1)$ is the identity. By elementarity, it suffices to prove that $(\id_{Q \otimes W})^{Q,W}$ is continuous at $\seed^{Q \otimes W}_{(2n+1,\mathbf{A})}$. We prove the typical case when $\comp{\leq 2n+1}{Q} = \emptyset$ and $\comp{\leq 2n+1}{W} \neq \emptyset$. 

Case 1: $\wocode{\mathbf{A}}_{\prec^{Q \otimes W}_{*}} = 0$. 

Then $\mathbf{A} = ( ( \mathbf{q}, \sigma),  \id_T, T )$ where putting $\mathbf{D} = (2n+2, \mathbf{q}, \sigma) \in \desc(Q,W,*)$, we have $\lh(\mathbf{D}) = 1$, $Q \otimes W (\mathbf{D}) = (T, (c,t,S))$,   $\mathbf{D} \prec^{Q \otimes W} \mathbf{D}'$ whenever $\lh(\mathbf{D}') = 1$ and $\mathbf{D} \neq \mathbf{D}'$. Let $Q'$ be the extension of $Q$ by adding the node $(2n+1, ((0)))$ into its domain such that $Q'(2n+1, ((0))) = (T, (c,t,S))$. Given any $g$ such that $[g]_{\mu^Q} < (\id_{Q \otimes W})^{Q , W}(\bolddelta{2n+1})$, we partition functions $f \in ((\bolddelta{2i+1})_{i \leq m})^{Q' \uparrow}$ according to whether or not ${}^{2n+1}[f]^{Q'}_{((0))} \leq g([f \res \rep(Q)]^Q)$. We obtain, by (\ref{item:strong_partition}:$n$) and the assumption on $g$, clubs $\vec{E} = (E_i)_{i \leq m} \in \prod_{i \leq m} \nu_{2i+1}$ such that for any $f \in \vec{E}^{Q'\uparrow}$,   ${}^{2n+1}[f]^{Q'}_{((0))} > g([f \res \rep(Q)]^Q)$. This implies that $[g]_{\mu^Q}<$the $u^{(2n-1)}_{E(2n-1)}$-th element of $E_m$. So $(\id_{Q \otimes W})^{Q,W}$ is continuous at $\bolddelta{2n+1} = \seed^{Q \otimes W}_{(2n+1,\mathbf{A})}$. 

Case 2: $\wocode{\mathbf{A}}_{\prec^{Q \otimes W}_{*}} > 0$. 

We prove the case when $\mathbf{A} = ((\mathbf{q}, \sigma), \pi, T)$, where putting $\mathbf{D} = (2n+2, \mathbf{q}, \sigma)$, we have $\mathbf{D} \in \desc(Q,W,*)$ and $\mathbf{q}$ is of discontinuous type. The other cases when $\mathbf{A} = ((\mathbf{q}, \sigma) \concat (-1), \pi, T)$ or $\mathbf{q}$ is of continuous type are similar. Put $\mathbf{q} = (q,P, (d_i,p_i,R_i)_{i \leq k+1})$. If $\wocode{\mathbf{A}}_{\prec^{Q \otimes W}_{*}}  = \eta +1$ is a successor, we must have that $d_{k+1} = 0$ and $(\pi,T)$ is discontinuous at $(d_{k},p_k)$. Let $\mathbf{A}' = ((\mathbf{q}, \sigma), \pi',T)$ where $\pi'$ and $\pi$ agree on $\dom(Q) \setminus\se{(d_{k},p_{k})}$, $\pi' (d_k,p_k) = \pred(\pi, T, (d_k,p_k))$. Then $\wocode{\mathbf{A}'}_{\prec^{Q \otimes W}_{*} } = \eta$. Given any $g$ such that $[g]_{\mu^Q} < (\id_{Q \otimes W})^{Q, W} (\seed^{Q \otimes  W}_{(2n+1, \mathbf{A})})$, we partition function  $f \in ((\bolddelta{2i+1})_{i \leq m})^{T \uparrow}$ according to whether or not $  [f]^T_{\pi'(d_k,p_k)} \leq g ( ([f ] ^T)_{\pi})  $. The homogeneous side must satisfy $>$, yielding that $[g]_{\mu^Q} < \seed^{Q \otimes W}_{(2n+1, \mathbf{A})}$. If $\wocode{\mathbf{A}}_{\prec^{Q \otimes W}_{*}} =\eta$ is a limit, we must have $d_{k+1}>0$ and we obtain % $\mathbf{D}' = (2n+2, (\mathbf{q}', \sigma')) \in \desc(Q \otimes W)$ and
$(\mathbf{A}_i,\pi_i,T_i)_{i < \omega}$ such that $\sup_{i<\omega} \wocode{\mathbf{A}_i}_{\prec^{Q \otimes W}_{*}} = \eta$ and for any $i$, $\mathbf{A}_i =( (\mathbf{q},\sigma) \concat (-1), \pi_i,T_i) \in \exexdesc(Q \otimes W)$.
% % $ \lh(\mathbf{D}') = \lh(\mathbf{D}) +1$,  $\mathbf{D} \iniseg ^{Q,W} \mathbf{D}'$ and either
% \begin{itemize}
% \item for any $i$, $\mathbf{A}_i =( (\mathbf{q}',\sigma'), \pi_i,T_i) \in \exexdesc(Q \otimes W)$, or 
% \item for any $i$, $\mathbf{A}_i =( (\mathbf{q}',\sigma') \concat (-1), \pi_i,T_i) \in \exexdesc(Q \otimes W)$.
% \end{itemize}
We may further assume that: $T_0$ is a one-node extension of $T$, $T_{i+1}$ is a one-node extension of $T_i$, $t_0 \in \dom(T_0) \setminus \dom(T)$, $t_{i+1} \in \dom(T_{i+1}) \setminus \dom(T_i)$,   $ \comp{d_k}{\pi}_i(p_k) = t_i$ and  $d_k>1$ implies that $t_i^{-} = t_0^{-}$ and $\comp{d_k}{T}_i[t_i] = \comp{d_k}{T}_i[t_0]$. Let $T_{\omega} = \cup_i T_i$ and work with partition arguments based on $T_{\omega}$.

(\ref{item:local_definability_jQ}:$n$) follows from (\ref{item:tensor_product_QW}:$n$) and \Los{}.

(\ref{item:regular_cardinals_Q}:$n$) is proved as follows. Suppose $Q$ is a finite level $\leq 2n+2$ tree. By (\ref{item:ultrapower_bound}:$n$), $j^P (T_{2n+1}) \in \mathbb{L}_{\bolddelta{2n+3}}[T_{2n+2}]$ for any finite level-($2n+1$)-tree $P$, and hence by (\ref{item:muP_is_measure}:$n$),  
the set of $\mathbb{L}_{\bolddelta{2n+3}}[j^Q(T_{2n+2})]$-cardinals in the interval $[\bolddelta{2n+1}, \bolddelta{2n+3})$ is a subset of $\set{u^{(2n+1)}_{\xi}}{0 < \xi \leq E(2n+1)}$. But every $u^{(2n+1)}_{\xi}$ is an $\mathbb{L}_{\bolddelta{2n+3}}[j^Q(T_{2n+2})]$-cardinal by an easy adaption of Martin's proof that under AD,  if $\kappa$ has the strong partition property and $\mu$ is an ultrafilter on $\kappa$, then $j^{\mu}(\kappa)$ is a cardinal. The part on $\mathbb{L}_{\bolddelta{2n+3}}[j^Q(T_{2n+2})]$-regular cardinals is an easy generalization of the $n=0$ case in \cite{sharpIII}.

(\ref{item:Q_signature}:$n$) is a simple computation, using (\ref{item:regular_cardinals_Q}:$n$) for the part concerning uniform cofinality. 

(\ref{item:TQ_desc_order}:$n$) simply follows from definitions. (\ref{item:TQ_desc_target_extension}:$n$) follows from (\ref{item:QW_desc_origin_extension}:$n$) and (\ref{item:respecting_R}:$k$) for $k < n$, (\ref{item:q_respecting}:$k$) for $k \leq n$. (\ref{item:TQ_desc_origin_extension}:$n$) follows from Lemmas~\ref{lem:discontinuity_1}, \ref{lem:discontinuity_2} and (\ref{item:respecting_R}:$k$) for $k < n$, (\ref{item:q_respecting}:$k$) for $k \leq n$.
(\ref{item:XT_factor}:$n$) follows from \Los{},  (\ref{item:TQ_desc_order}:$n$)-(\ref{item:TQ_desc_target_extension}:$n$) and (\ref{item:respecting_R}:$k$) for $k < n$, (\ref{item:q_respecting}:$k$) for $k \leq n$.

(\ref{item:tensor_product_TQ}:$n$) follows from (\ref{item:tensor_product_QW}:$n$), (\ref{item:regular_cardinals_Q}:$n$) and the associativity of the $\otimes$-operator acting on level ($\leq 2n+2, \leq 2n+2, \leq 2n+1$) trees. 

We outline the proof of (\ref{item:order_type_embed}:$n$). Let $\theta : \rep(X) \to \rep(T)$ be an isomorphism. For $(e,x) \in \dom(X)$ and $e>1$, let $X_{\tree}(e,x) = W_{(e,x)}$. Let $\vec{E} = (E_i)_{i \leq n} \in \prod_{i \leq n} \nu_{2i+1}$, $(d_{e,x}, \mathbf{t}_{e,x}) \in \desc(T)$, $\mathbf{t}_{(e,x)} = (t_{e,x}, S_{e,x},\dots)$ when $d_{(e,x)}>1$, 
and let $\theta_{(e,x)} \in \mathbb{L}_{\bolddelta{2n+3}}[T_{2n+2}]$ be such that 
\begin{itemize}
\item $d_{(e,x)} = 1$ implies $e=1$ and $\theta(1, (x)) = (d_{(1,x)}, (\mathbf{t}_{(1,x)}))$, and
\item $d_{(e,x)}>1$ implies that for any $\vec{\alpha} \in [\vec{E}]^{W_{(e,x)}\uparrow}$,  $\theta(e, \vec{\alpha}\oplus_{\comp{e}{Q}} x) = (d_{(e,x)}, \theta_{(e,x)}(\vec{\alpha}) \oplus _{\comp{d_{(e,x)}}{T}} t_{(e,x)})$. 
\end{itemize}
If $d_{(e,x)}>1$, let $\vec{\beta}_{(e,x)} = (\beta_{(e,x),(a,s)})_{(a,s) \in \dom(S_{(e,x)})} = [\theta_{(e,x)}]_{\mu^{W_{(e,x)}}}$, 
for $k \leq n$, let
\begin{align*}
  B^{(2k+1)}_{(e,x)} &= \set{(a,s) \in \dom(S_{(e,x)})}{a=2k+1, \beta_{(e,x),(a,s)} < \bolddelta{2k+1}},\\
  B^{(2k+2)}_{(e,x)} &= \set{(a,s) \in \dom(S_{(e,x)}) \setminus B^{(2k+1)}_{(e,x)}}{ 2k+1 \leq a \leq 2k+2}.
\end{align*}
For $e>1$ and $2 \leq k \leq 2n+2$, let
\begin{align*}
  D^{(k)}_{(e,x)} & = \set{(a,s) \in B^{(k)}_{(e,x)}}{ \beta_{(e,x),(a,s)} \text{ is $S_{(e,x)}$-essentially continuous}},\\
  E^{(k)}_{(e,x)} & = B^{(k)}_{(e,x)} \setminus D^{(k)}_{(e,x)}.
\end{align*}
If $(a,s) \in B^{(k)}_{(e,x)}$ and $a>1$, let $(S_{(e,x)})_{\tree}(a,s) = U_{(e,x),(a,s)}$, let $\Phi_{(e,x),(a,s)}$ be the $\comp{\leq k-1}{(W_{(e,x)} \otimes U_{(e,x),(a,s)})}$-potential partial level $\leq k-1$ tower induced by $\beta_{(e,x),(a,s)}$, let $\lh(\Phi_{(e,x),(a,s)})$ be the length of the second coordinate of $\Phi_{(e,x),(a,s)}$, let $(\mathbf{B}_{(e,x),(a,s),i})_{i < v_{(e,x),(a,s)}}$ be the $\comp{\leq k-1}{(W_{(e,x)} \otimes U_{(e,x),(a,s)})}$-signature of $\beta_{(e,x),(a,s)}$, let $\sigma_{(e,x),(a,s)}$ be $\comp{\leq k-1}{(W_{(e,x)} \otimes U_{(e,x),(a,s)})}$-factoring map induced by $\beta_{(e,x),(a,s)}$, and let $(\gamma_{(e,x),(a,s),i})_{i \leq v_{(e,x),(a,s)}}$ be the $\comp{\leq k-1}{(W_{(e,x)} \otimes U_{(e,x),(a,s)})}$-approximation sequence of $\beta_{(e,x),(a,s)}$. Let
\begin{displaymath}
  \phi^1 : \set{\beta_{(e,x),(a,s)}}{(e,x) \in \dom(X), (a,s) \in B^{(1)}_{(e,x)}} \to Z^1
\end{displaymath}
be a bijection such that $Z^1$ is a level-1 tree and $v < v' \eqiv \phi^1(v) <_{BK} \phi^1(v')$. For $2 \leq k \leq 2n+2$, let
\begin{align*}
    \phi^k: &\set{(\mathbf{B}_{(e,x),(a,s),i}, \gamma_{(e,x), (a,s),i})_{i < l}}{(e,x) \in \dom(X), (a,s) \in B^{(k)}_{(e,x)}, l < \lh(\Phi_{(e,x),(a,s)})}\\
&~ \to Z^k \cup \se{\emptyset}
\end{align*}
be a bijection such that $Z^k$ is a tree of level-1 trees and $v \subseteq v' \eqiv \phi^k (v) \subseteq \phi^k(v')$, $v<_{BK} v' \eqiv \phi^k(v) <_{BK}\phi^k(v')$. Let
\begin{displaymath}
  Q = (\comp{1}{Q}, \dots, \comp{2n+2}{Q})
\end{displaymath}
be a level $\leq 2n+2$ tree where $\dom(\comp{k}{Q}) = Z^k$ and for $2 \leq k \leq 2n+2$,
\begin{align*}
  \comp{k}{Q}[\phi^k ((w_{(e,x),(a,s),i}, \gamma_{(e,x),(a,s),i})_{i < \lh(\Phi_{(e,x),(a,s)})})\concat (-1)]& = \Phi_{(e,x),(a,s)} \text{ when } (a,s) \in D^{(k)}_{(e,x)},\\
  \comp{k}{Q}[\phi^k ((w_{(e,x),(a,s),i}, \gamma_{(e,x),(a,s),i})_{i < \lh(\Phi_{(e,x),(a,s)})})]& = \Phi_{(e,x),(a,s)} \text{ when } (a,s) \in E^{(k)}_{(e,x)}.
\end{align*}
Let $\pi$ factor $(X,T,Q)$, where $\pi(1,x) = (1, \mathbf{t}_{(e,x)}, \emptyset)$ if $d_{(1,x)} = 1$, $\pi(e,x) = (d_{(e,x)}, \mathbf{t}_{(e,x)}, \tau_{(e,x)})$ if $d_{(e,x)}>1$, where $\tau_{(e,x)}$ factors $(S_{(e,x)}, Q, W_{(e,x)})$, $\tau_{(e,x)}(a,s)$ is equal to 
\begin{itemize}
\item  $ (1, \phi^1(\beta_{(e,x),(a,s)}),\emptyset)$ if $(a,s) \in B^{(1)}_{(e,x)}$,
\item $(k, (\phi^k (\mathbf{B}_{(e,x),(a,s),i}, \gamma_{(e,x),(a,s),i})_{i < \lh(\Phi_{(e,x),(a,s)})}\concat (-1)) \concat \Phi_{(e,x),(a,s)}, \sigma_{(e,x),(a,s)})$ if $(a,s) \in D^{(k)}_{(e,x)}$, $k>1$,
\item  $(k, (\phi^k (\mathbf{B}_{(e,x),(a,s),i}, \gamma_{(e,x),(a,s),i})_{i < \lh(\Phi_{(e,x),(a,s)})}) \concat \Phi_{(e,x),(a,s)}, \sigma_{(e,x),(a,s)})$ if $(a,s) \in E^{(k)}_{(e,x)}$, $k>1$.
\end{itemize}
The fact that $\theta$ is an isomorphism implies that $\pi$ minimally factors $(X,T,Q)$.

By analyzing the representative functions, we obtain the following lemmas in parallel to Lemmas~\ref{lem:continuity}-\ref{lem:discontinuity_2}. 
\begin{lemma}
  \label{lem:continuity_another}
  Suppose $Q,T$ are finite level $\leq 2n+1$ trees and $\pi$ factors $(Q,T)$. Suppose $\gamma < j^Q(\bolddelta{2n+1})$ and $\cf^{\mathbb{L}_{\bolddelta{2n+3}}[j^Q(T_{2n+2})]}(\gamma) = \seed^Q_{(d,\mathbf{q})}$, $(d,\mathbf{q}) \in \exdesc(Q)$ is regular.  Then
  \begin{enumerate}
  \item $\pi^T$ is continuous at $\gamma$ iff $(\pi,T)$ is continuous at $(d,\mathbf{q})$.
  \item Suppose $(\pi,T)$ is discontinuous at $(d,\mathbf{q})$. Let $(Q^{+}, \pi^{+})$ be the $T$-decomposition of $\pi$. Then $\pi^T_{\sup}(\gamma) = (\pi^{+})^T \circ j^{Q,Q^{+}}_{\sup} (\gamma)$.
  \end{enumerate}
\end{lemma}

\begin{lemma}
  \label{lem:discontinuity_another_1}
  Suppose $(Q^{-},{(d,q,P)})$ is a partial level $\leq 2n+2$ tree and $Q$ is a completion of $(Q^{-}, {(d,q,P)})$. Suppose $T$ is a level $\leq 2n+2$ tree and $\pi,\pi'$ both factor $(Q,T)$, $\pi$ and $\pi'$ agree on $\dom(Q^{-})$, $\pi'(d,q) = \pred(\pi, T, (d,q))$. Then for any $\gamma <\bolddelta{2n+3}$ such that $\cf^{\mathbb{L}_{\bolddelta{2n+3}}[j^{Q^{-}}(T_{2n+2})]}(\gamma) = \seed^{Q^{-}}_{\ucf(Q^{-},(d,q,P))}$, we have
  \begin{displaymath}
    \pi^T \circ j^{Q^{-},Q}_{\sup} (\gamma) = (\pi')^T_{\sup} \circ j^{Q^{-},Q} (\gamma).
  \end{displaymath}
\end{lemma}

\begin{lemma}
 \label{lem:discontinuity_another_2} 
Suppose $(Q,(d,q,P))$ is a partial level $\leq 2n+2$ tree, $\ucf(Q,(d,q,P)) = (d^{*}, \mathbf{q}^{*})$ and $\pi$ factors $(Q,T)$. Suppose $\gamma < \bolddelta{2n+3}$ and either
  \begin{enumerate}
  \item $d=0$, $Q^{+} = Q$, $\pi' = \pi$, $\cf^{\mathbb{L}[j^Q(T_{2n+2})]}(\gamma) = \omega$, or
  \item $d > 0$, $Q^{+}$ is a completion of $Q$, $\pi'$ factors $(Q^{+}, T)$, $\pi = \pi' \res \dom(Q)$, $\pi'(d,q) = \pred(\pi, T, (d^{*},\mathbf{q}^{*})) $, $\cf^{\mathbb{L}_{\bolddelta{2n+3}}[j^Q(T_{2n+2})]}(\gamma) = \seed^Q_{(d^{*}, \mathbf{q}^{*})}$.
  \end{enumerate}
  Then
  \begin{displaymath}
    \pi^T(\gamma) = (\pi')^T_{\sup} \circ j^{Q,Q^{+}}(\gamma).
  \end{displaymath}
\end{lemma}

(\ref{item:rep_R}:$n$) follows from (\ref{item:Q_direct_limit_wf}:$n$). The proof of (\ref{item:gamma_r_continuous_type}:$n$) generalizes Lemma~\ref{lem:gamma_r_continuous_type}, using (\ref{item:tensor_product_QW}:$n$) when necessary. 
The proof of (\ref{item:respecting_R}:$n$) generalizes the lower levels in an obvious way, using (\ref{item:regular_cardinals_Q}:$n$) when necessary. 

The proof of (\ref{item:R_description}:$n$) is an easy generalization of the lower levels, using Lemmas~\ref{lem:continuity_another}-\ref{lem:discontinuity_another_2} when necessary.

(\ref{item:YT_desc_order}:$n$) simply follows from definitions. (\ref{item:YT_desc_target_extension}:$n$) follows from (\ref{item:TQ_desc_origin_extension}:$n$) and (\ref{item:q_respecting}:$k$)(\ref{item:respecting_R}:$k$) for $k \leq n$. (\ref{item:YT_desc_origin_extension}:$n$) follows from Lemmas~\ref{lem:discontinuity_another_1}-\ref{lem:discontinuity_another_2}  and (\ref{item:q_respecting}:$k$)(\ref{item:respecting_R}:$k$) for $k \leq n$.

The proof of (\ref{item:order_type_embed_another}:$n$) is similar to (\ref{item:order_type_embed}:$n$).

We outline the proof of (\ref{item:represent_cofinal}:$n$). Suppose $R, (d,q,P), k$ are as given. The case $k<n$ follows from (\ref{item:represent_cofinal}:$k$). The case $d=2n+1$ follows from  (\ref{item:represent_another}:$n-1$). Assume now $d=2n+2$. Ordinals of the form $\tau^{j^P(M_{2n+2,\infty}^{-}(x))} ( x, \seed^P_{((0))} )$ are cofinal in $j^P(\bolddelta{2n+1})$. Fix such an $\alpha = \tau^{j^P(M_{2n+2,\infty}^{-}(x))}(x, \seed^P_{((0))})$ and we build a $\Pi^1_{2n+2}$-wellfounded level-($2n+2$) tree $T$ such that $T_{\tree}(((0))) = P$ and $\llbracket ((0)) \rrbracket_T > \alpha$. Indeed, we build $T$ satisfying $T_{\tree}(((0))) = P$ and for any $z \in [\dom(T)]$, 
$T(z) \DEF \cup_i T(z \res i)$ is a level $\leq 2n+1$ tree as the ``join'' of $(W_{(z)_i,(z)_{i+1}})_{  -1\leq i < \omega}$, where $(z)_{-1}$ is computable from $(z)_0$ so that if $(z)_0$ codes $(x, P, \gamma, Q, \vec{\beta}, \gcode{\sigma})$, a level $\leq 2n+1$ code for an ordinal in $\bolddelta{2n+1}$ relative to $x$, then $(z)_{-1} $ codes $(x, P, \gamma, Q^{(2n)}_0, \vec{\beta}^{(2n)}, \gcode{\tau})$, $\vec{\beta}^{(2n)}$ respects $Q^{(2n)}_0$, and 
$W_{(v,v')}$ is $\Pi^1_{2n+1}$-wellfounded iff $v,v'$ code $\bar{v}, \bar{v}'$,  level $\leq 2n+1$ codes for ordinals in $\bolddelta{2n+1}$ relative to $x$, such that $\sharpcode{\bar{v}} > \sharpcode{\bar{v}'}$. 
 
(\ref{item:represent}:$n$) follows from (\ref{item:represent_cofinal}:$n$) by a straightforward generalization of the corresponding arguments in \cite{sharpII,sharpIII}. The reader who can follow us this far should have no problem figuring out the details. The proof of (\ref{item:represent_another}:$n$) is similar to (\ref{item:represent}:$n$).

\section*{Acknowledgements}
The breakthrough ideas of this paper were obtained during the AIM workshop on Descriptive inner model theory, held in Palo Alto, and the Conference on Descriptive Inner Model Theory, held in Berkeley, both in June, 2014. 
The author greatly benefited from conversations with Rachid Atmai and Steve Jackson that took place in these two conferences. 
The final phase of this paper was completed whilst the author was a visiting fellow at the Isaac Newton Institute for Mathematical Sciences in the programme `Mathematical, Foundational and Computational Aspects of the Higher Infinite' (HIF) in August and September, 2015 funded by NSF Career grant DMS-1352034 and EPSRC grant EP/K032208/1.

\bibliography{sharp}{}

\begin{thebibliography}{10}

\bibitem{woodin-handbook}
Peter Koellner and W.~Hugh Woodin.
\newblock Large cardinals from determinacy.
\newblock In {\em Handbook of set theory. {V}ols. 1, 2, 3}, pages 1951--2119.
  Springer, Dordrecht, 2010.

\bibitem{mos_dst}
Yiannis~N. Moschovakis.
\newblock {\em Descriptive set theory}, volume 155 of {\em Mathematical Surveys
  and Monographs}.
\newblock American Mathematical Society, Providence, RI, second edition, 2009.

\bibitem{nee_opt_I}
Itay Neeman.
\newblock Optimal proofs of determinacy.
\newblock {\em Bull. Symbolic Logic}, 1(3):327--339, 1995.

\bibitem{nee_opt_II}
Itay Neeman.
\newblock Optimal proofs of determinacy. {II}.
\newblock {\em J. Math. Log.}, 2(2):227--258, 2002.

\bibitem{SUW}
Ralf Schindler, Sandra Uhlenbrock, and Hugh Woodin.
\newblock Mice with finitely many {W}oodin cardinals from optimal determinacy
  hypotheses, in preparation.

\bibitem{sol_delta13coding}
Robert~M. Solovay.
\newblock A {$\Delta ^{1}_{3}$} coding of the subsets of {$\omega _{\omega }$}.
\newblock In {\em Wadge degrees and projective ordinals. {T}he {C}abal
  {S}eminar. {V}olume {II}}, volume~37 of {\em Lect. Notes Log.}, pages
  346--363. Association for Symbolic Logic, La Jolla, CA; Cambridge University
  Press, Cambridge, 2012.

\bibitem{hod_as_a_core_model}
J.~R. Steel and W.~Hugh Woodin.
\newblock {HOD} as a core model.
\newblock In {\em Ordinal Definability and Recursion Theory. {T}he {C}abal
  {S}eminar. {V}olume {III}}, volume~43 of {\em Lect. Notes Log.}, pages
  257--346. Cambridge University Press, Cambridge, 2016.

\bibitem{sharpI}
Yizheng Zhu.
\newblock The higher sharp {I}: on ${M}_1^{\#}$.

\bibitem{sharpII}
Yizheng Zhu.
\newblock The higher sharp {II}: on ${M}_2^{\#}$.

\bibitem{sharpIII}
Yizheng Zhu.
\newblock The higher sharp {III}: an {EM} blueprint of $0^{3\#}$ and the
  level-4 {K}echris-{M}artin.

\end{thebibliography}
\bibliographystyle{plain}

\end{document}